\documentclass[11pt]{amsart}
\usepackage{amssymb}
\usepackage{graphicx}
\usepackage{url}
\usepackage{enumerate}
\usepackage{transparent}
\usepackage{enumitem}
\usepackage{verbatim}
\usepackage{wrapfig} 
\usepackage{color}
\usepackage{leftidx}
\usepackage{tikz}
\usetikzlibrary{decorations.markings}
\usepackage{array}
\usepackage{booktabs}
\usepackage{tikz-cd}
\usetikzlibrary{arrows}
\usepackage{amsmath}
\usepackage{calligra}
\usepackage{setspace}
\usepackage{mathrsfs}
\usepackage{geometry}
\usepackage{adjustbox}
\usepackage{amsthm}
\usepackage[colorlinks=true,linkcolor=violet,citecolor=violet]{hyperref}

\usepackage{libertine}
\usepackage{libertinust1math}
\usepackage[T1]{fontenc}

\usepackage{indentfirst}
\usepackage{fancyhdr}

\def\Spec{{Spec}\,}

\def\Im{\mathrm{Im}\,}
\def\Re{\mathrm{Re}\,}

\def\Tr{\mathrm{tr}}

\def\Hess{{Hess}}
\def\Det{\mathrm{det}}

\newtheorem{theorem}{Theorem}[section]
\newtheorem{lemma}[theorem]{Lemma}
\newtheorem{corollary}[theorem]{Corollary}
\newtheorem{proposition}[theorem]{Proposition}
\newtheorem{conjecture}[theorem]{Conjecture}
\newtheorem{remark}[theorem]{Remark}

\theoremstyle{definition}
\newtheorem{definition}[theorem]{Definition}
\newtheorem{question}[theorem]{Question}
\newtheorem{example}[theorem]{Example}

\newtheorem*{definitionproposition}{Definition-Proposition}
\newtheorem*{assumption}{Assumption}
\newtheorem*{convention}{Notations}
\newtheorem*{acknowledgement}{Acknowledgement}

\newcommand{\arrowIn}{
\tikz \draw[-stealth] (-1pt,0) -- (1pt,0);
}
\newcommand{\arrowIIn}{
\tikz \draw[-{stealth}{stealth}] (-1pt,0) -- (1pt,0);
}
\newcommand{\arrowBox}{
\tikz \draw[-{square}] (-1pt,0) -- (1pt,0);
}

\title{On 4-dimensional Ricci-flat ALE manifolds}
\author{Mingyang Li }
\date{}

\newcommand{\Addresses}{{
  \bigskip
  \footnotesize

  \textsc{{Department of Mathematics, University of California, Berkeley,
    CA 94720, USA.}}\par\nopagebreak
  \textit{E-mail address}: \href{mailto:mingyang_li@berkeley.edu}{\texttt{mingyang\_li@berkeley.edu}}

}}

\begin{document}

\begin{abstract}
In this paper, we prove:
\begin{itemize}
\item There is a one-to-one correspondence between: 
\begin{itemize}
\item Hermitian non-K\"ahler ALE gravitational instantons $(M,h)$;
\item Bach-flat K\"ahler orbifolds $(\widehat{M},\widehat{g})$ of complex dimension 2 with exactly one orbifold point $q$, such that the scalar curvature $s_{\widehat{g}}$ satisfies $s_{\widehat{g}}(q)=0$ while being positive elsewhere.
\end{itemize}
\item There is no Hermitian non-K\"ahler ALE gravitational instanton with structure group contained in $SU(2)$, except for the Eguchi-Hanson metric with reversed orientation.
\end{itemize}

A 4-dimensional Ricci-flat metric being \textit{Hermitian non-K\"ahler} is equivalent to being \textit{non-trivially conformally K\"ahler}.
\end{abstract}

\maketitle

\setcounter{tocdepth}{1}
\tableofcontents

\section{Introduction}

The following question is a long-standing problem:
\begin{question}
Is there an ALE Ricci-flat 4-manifold that is neither self-dual nor anti-self-dual?
\end{question}

In this paper, we refer to complete non-compact Ricci-flat 4-manifolds $(M,h)$ with $\int_{M}|Rm_h|_h^2<\infty$ as \textit{gravitational instantons}. It is a famous conjecture by Bando-Kasue-Nakajima \cite{bkn,nakajima} that any ALE gravitational instanton is either self-dual or anti-self-dual.
In this paper, we give a negative answer to the problem introduced at the beginning in a special case. By saying that a $4$-manifold $(M,h)$ is \emph{asymptotically locally Euclidean} (ALE), we mean
\begin{definition}\label{def:ALE}
A Riemannian 4-manifold $(M,h)$ is ALE with order $\tau$, if there is a smooth diffeomorphism $\Phi:M\setminus K\to (\mathbb{R}^4\setminus\overline{B_R(0)})/\Upsilon$,
where $K$ is a large compact subset of $M$ and $\Upsilon\subset SO(4)$ is a finite group acting freely on $S^3$, such that
$$\left|\nabla_h^k(h-\Phi^*h_E)\right|_h=O(\rho^{-\tau-k})$$
as $\rho\to\infty$ for any $k\geq0$. Here $\rho$ is the distance function under $h$ to some base point $p$, $h_E$ is the standard Euclidean metric on $\mathbb{R}^4/\Upsilon$, and $\nabla_h$ is the Levi-Civita connection of $h$. 
\end{definition}

There is the following curvature decomposition for oriented 4-manifolds by identifying the curvature tensor $Rm$ as an automorphism of the bundle of two-forms $\Lambda^2=\Lambda^+\oplus\Lambda^-$
\begin{align*}
    Rm=\left(\begin{matrix}
        \frac{s}{12}+W^+ & Ric^0\\
        Ric^0 & \frac{s}{12}+W^{-}
    \end{matrix}\right).
\end{align*}
Here $s$ is the scalar curvature, $W^{\pm}$ are the self-dual and anti-self-dual Weyl curvatures, and $Ric^0$ is the trace-less Ricci curvature. In particular, for oriented Einstein 4-manifolds, we only have the scalar curvature part $s$ and the Weyl curvature part $W^{\pm}$. Locally the self-dual Weyl curvature $W^+:\Lambda^+\to\Lambda^+$ is a $3\times3$ matrix.
By the work of Derdzi\'nski \cite{derdzinski}, oriented Einstein 4-manifolds can be broadly classified into the following three cases.

\begin{definitionproposition}
An oriented Einstein 4-manifold $(M,h)$ can be classified into one of the following types:
\begin{itemize}
\item Type I: If $W^+ \equiv 0$, then the metric is anti-self-dual.
\item Type II: If $W^+$ has exactly two distinct eigenvalues, treated as an automorphism $W^+:\Lambda^+\to\Lambda^+$ everywhere, then up to passing to a double cover\footnote{It was pointed out by Claude LeBrun to the author that one should include passing to a double cover here. The author is very grateful for that.}, there exists a compatible complex structure $J$ such that $(M,h,J)$ is Hermitian and the conformal metric $g=\lambda^{2/3}h$ is K\"ahler, where $\lambda=2\sqrt{6}|W^+|_h$. The scalar curvature $s_g$ of $g$ satisfies $|s_g|=\lambda^{1/3}$.
\item Type III: If $W^+$ generically has three distinct eigenvalues, then $(M,h)$ is not locally conformally K\"ahler under this orientation.
\end{itemize}
\end{definitionproposition}

Type I ALE gravitational instantons are precisely the \textit{anti-self-dual} ALE gravitational instantons. A very important class of such gravitational instantons are those \textit{hyperk\"ahler} ones, which are completely classified by Kronheimer \cite{kronheimer1,kronheimer2}. Finite quotients of hyperk\"ahler gravitational instantons, which are still Type I, are also classified by \cite{s,wright}.

For Type II ALE gravitational instantons, after passing to a double cover if necessarily, they are precisely the \textit{Hermitian non-K\"ahler} ALE gravitational instantons.  It turns out that the conformal metric $g=\lambda^{2/3}h$ is not only K\"ahler, but also \textit{Bach-flat}, and \textit{extremal K\"ahler} in the sense of Calabi. Note that in the complex 2-dimensional case, a Bach-flat K\"ahler metric is automatically extremal K\"ahler \cite{lebrun2}.

This paper focuses on Type II (equivalently, after passing to the cover being Hermitian non-K\"ahler) ALE gravitational instantons, specifically proving the following two main theorems. 
\begin{theorem}\label{main:correspondence}
There is a one-to-one correspondence between: 
\begin{itemize}
\item Hermitian non-K\"ahler ALE gravitational instantons $(M,h)$;
\item Bach-flat K\"ahler orbifolds $(\widehat{M},\widehat{g})$ of complex dimension 2 with exactly one orbifold point $q$, whose scalar curvature $s_{\widehat{g}}$ satisfies $s_{\widehat{g}}>0$ except at $q$ while $s_{\widehat{g}}(q)=0$.
\end{itemize}
The structure group of a Hermitian non-K\"ahler ALE gravitational instanton must be contained in $U(2)$, in the sense that outside of a suitable compact set, the end is biholomorphic to $\mathbb{B}^*/\Gamma$, where $\mathbb{B}^*\subset\mathbb{C}^2$ is the standard punctured unit ball and $\Gamma\subset U(2)$.
\end{theorem}

\begin{theorem}\label{main:classificationsu2}
With the exception of the Eguchi-Hanson metric with reversed orientation, there are no Hermitian non-K\"ahler ALE gravitational instantons with structure group contained in $SU(2)$.
\end{theorem}

Note that the group $\Gamma$ in Theorem \ref{main:correspondence} is different from the group $\Upsilon$ in Definition \ref{def:ALE}--they are conjugate as subgroups of $SO(4)$.  We shall refer to the group $\Gamma$ as the \textit{structure group}.
Here the Eguchi-Hanson space is usually understood as a hyperk\"ahler space. However, it turns out that it also carries a complex structure whose orientation is opposite to the hyperk\"ahler orientation. The metric is Hermitian non-K\"ahler under this complex structure with opposite orientation.
Metrics corresponding to Hermtian non-K\"ahler ALE gravitational instantons in Theorem \ref{main:correspondence} will be refered to as \emph{special Bach-flat K\"ahler metrics} for simplicity.
\begin{definition}
A Bach-flat K\"ahler metric $\widehat{g}$ on a compact complex 2-dimensional orbifold with only one orbifold point $q$ is said to be \textit{special Bach-flat K\"ahler}, if its scalar curvature $s_{\widehat{g}}$ is positive at all points except at the orbifold point $q$, where $s_{\widehat{g}}(q)=0$.
\end{definition}

Now we explain some main ideas in the proof: 
\begin{itemize}
\item For a Hermitian non-K\"ahler ALE gravitational instanton $(M,h)$, with $\lambda:= 2\sqrt{6}|W^+|_h$, the  extremal K\"ahler metric $g:=\lambda^{2/3}h$ turns out to be incomplete, whose metric completion is just adding one point. Using a singularity removal argument, we can show that $g$ extends to a smooth orbifold extremal K\"ahler metric $\widehat{g}$ (Section \ref{sec:removal}). The metric $\widehat{g}$ is the special Bach-flat K\"ahler metric in the correspondence of Theorem \ref{main:correspondence}.

\item Denote the compactified special Bach-flat K\"ahler orbifold as $(\widehat{M},\widehat{g})$. In Section \ref{subsec:complexstructure}, it is shown that $\widehat{M}$ is a log del Pezzo surface, which is a Fano surface with only one quotient singularity. This is proved based on an important observation by LeBrun \cite{lebrun}. We therefore conclude that the compactified surface $\widehat{M}$ must be rational.

\item Each special Bach-flat K\"ahler orbifold  gives us a pair $(\widehat{M},\mathfrak{E})$, where $\widehat{M}$ is the underlying orbifold and $\mathfrak{E}$ is the holomorphic extremal vector field associated to the extremal K\"ahler metric $\widehat{g}$. 
Such pairs have the following algebraic properties:
\begin{itemize}
\item The orbifold $\widehat{M}$ has only one orbifold point $q$.
\item The orbifold $\widehat{M}$ is log del Pezzo.
\item The orbifold point $q$ is an isolated fixed point of the action by $\mathfrak{E}$. Moreover, the weights of the $\mathfrak{E}$ action on the tangent space of $q$ are the same. 
\end{itemize}
It will be an algebraic geometry problem to classify all such pairs with the above algebraic properties, and we classify them under the additional technical assumption that the orbifold group $\Gamma\subset SU(2)$. Importantly, we find that there exist only a finite number of such pairs when the orbifold group is in $SU(2)$.
This is done in Section \ref{sec:case12}-\ref{sec:case3}.

\item These classified pairs $(\widehat{M},\mathfrak{E})$ are candidates for special Bach-flat K\"ahler orbifolds with structure group in $SU(2)$.
If an extremal K\"ahler metric exists with $\mathfrak{E}$ as the holomorphic extremal vector field on these candidates, we will show in Section \ref{sec:nonexistence} that the minimum of its scalar curvature can never be zero, except for the pair corresponding to the Eguchi-Hanson metric with reversed orientation. Together with Theorem \ref{main:correspondence}, this proves Theorem \ref{main:classificationsu2}.
\end{itemize}

Note that in the smooth setting, the holomorphic extremal vector field of an extremal K\"ahler metric always generates a holomorphic $\mathbb{C}^*$-action, as a consequence of \cite{futaki}. 
Theorem \ref{main:correspondence}, \ref{main:classificationsu2} have several interesting corollaries.
During the course of our proof of the main theorems, we also prove
\begin{corollary}[Finiteness of topological types]\label{cor:topofinite}
For each fixed finite subgroup $\Gamma\subset U(2)$, there exist at most finitely many diffeomorphism types of Hermitian non-K\"ahler ALE gravitational instantons with structure group $\Gamma$.
\end{corollary}

We will discuss later that we actually expect that there is no Type II ALE gravitational instanton at all, except for the Eguchi-Hanson metric with reversed orientation.
We also have the following corollary which will be established during the proof of Theorem \ref{main:correspondence}
\begin{corollary}\label{cor:hermitiandecay}
For a Hermitian ALE gravitational instanton, if its self-dual Weyl curvature $W^+$ decays faster than $\rho^{-6}$, then it must be K\"ahler.
\end{corollary}

Relating to previous works, compact 4-dimensional Hermitian Einstein metrics are classified by LeBrun \cite{lebrun3}. LeBrun's work \cite{lebrun,lebrun3} has shown that compact Hermitian Einstein 4-manifolds are limited to being either K\"ahler-Einstein, the Page metric in \cite{page}, or the Chen-LeBrun-Weber metric in \cite{clw}. Our Theorem \ref{main:correspondence} and \ref{main:classificationsu2} can be treated as a non-compact analogue of  \cite{clw,lebrun3}.

It is important to note that proving the non-existence of Hermitian non-K\"ahler ALE gravitational instantons with structure group not in $SU(2)$ requires a classification of log del Pezzo surfaces with only one $U(2)$ singularity and a holomorphic vector field $\mathfrak{E}$ having the same weights at the singularity point. While the case-by-case classification we will use for the $SU(2)$ case can be extended to the $U(2)$ case, it will be more complicated. Unlike the $SU(2)$ case, the $U(2)$ case may yield an infinite number of candidates $(\widehat{M},\mathfrak{E})$.
Given the difficulty in achieving the condition that the scalar curvature of the Bach-flat K\"ahler metric vanishes precisely at the orbifold point, we make the following conjecture:
\begin{conjecture}
There exist no Hermitian non-K\"ahler ALE gravitational instantons, without any restrictions on the structure group, except for the Eguchi-Hanson metric with reversed orientation.
\end{conjecture}

We finally would like to remark that
very little is currently known about general ALE gravitational instantons. Lock and Viaclovsky \cite{lv2} showed that if the minimal resolution of $\mathbb{C}^2/\Gamma$ with $\Gamma\subset U(2)$ or its further blow-up supports a Ricci-flat ALE metric, then $\Gamma\subset SU(2)$ and the metric is hyperk\"ahler, through some topological calculation based on the Hitchin-Thorpe inequality. In Biquard-Hein \cite{biquardhein}, they constructed an optimal ALE coordinate and an associated renormalized volume, and showed that the renormalized volume has to be nonpositive. This volume is zero if and only if the Ricci-flat manifold is a flat cone $\mathbb{R}^4/\Gamma$. It is notable that our result in this paper is based on complex-analytic method. To solve the problem we mentioned at the beginning, the difficulty enssentially lies in
\begin{question}
Is there a Type III ALE gravitational instanton?
\end{question}



\begin{convention}
\ 

\begin{itemize}
\item On a Riemannian manifold $(M,h)$ with a chosen base point $p$, $\rho$ denotes the distance function to $p$.
\item If $T$ is a tensor on $(M,h)$, then we say that $T=O'_h(\rho^{-\tau})$ if, for any integer $k\geq0$, 
$$|\nabla^k_hT|_h= O(\rho^{-\tau-k}),$$ 
as $\rho$ approaches infinity. We will omit $h$ when it is clear from the context.
\item Given a holomorphic vector field $\mathfrak{E}$, by flows of this vector field, we mean the flows of the real vector field $\Re \mathfrak{E}$.
\item We will use $H_r$ to denote the Hirzebruch surface $\mathbb{P}(\mathcal{O}\oplus\mathcal{O}(r))$.
\item For an extremal K\"ahler metric $g$, by \textit{extremal vector field} we mean $J\nabla_gs_g$. By \textit{holomorphic extremal vector field} we mean $\nabla_{g}^{1,0}s_g$.
\end{itemize}
\end{convention}

\begin{acknowledgement}
I would like to express my sincere gratitude to my advisor, Song Sun, for introducing me to this problem and for providing constant support throughout the course of my research. I am grateful for his willingness to engage in fruitful discussions and for generously sharing his ideas with me. I would like to extend my thanks to Junsheng Zhang for providing me with a helpful example. I am grateful to Bing Wang for his warm hospitality during my visit to the Institute of Geometry and Physics at USTC, where this work was completed. I am thankful to Gon\c{c}alo Oliveira for pointing out a mistake in an earlier version of this paper.

This work is supported by NSF grant DMS-2004261 and Simons Collaboration Grant on Special Holonomy
in Geometry, Analysis, and Physics (488633, S.S.).
\end{acknowledgement}

\section{Preliminary results}
\label{sec:preliminary}

In this section we introduce the Type I-III classification by Derdzi\'nski \cite{derdzinski} metioned in the introduction.

\begin{proposition}[Derdzi\'nski]\label{prop:d1}
For an oriented Einstein 4-manifold $(M,h)$, it must be of Type I, Type II, or Type III.
\end{proposition}
\begin{proof}
Let $\mu$, $\nu$, and $\eta$ be the three eigenvalues of $W^+$ (some of which possibly could coincide). The function $\Delta:=(\mu-\nu)(\nu-\eta)(\eta-\mu)$ is well-defined and real analytic because Einstein metrics are real analytic. Therefore, we must have either $\Delta\neq0$ and $W^+$ has three distinct eigenvalues generically, or $\Delta\equiv0$ and $W^+$ has at most two distinct eigenvalues everywhere. If the second case happens, Proposition 5 in \cite{derdzinski} implies that we either have $W^+\equiv0$, or $W^+$ has exactly two distinct eigenvalues everywhere which are both nowhere vanishing.
\end{proof}

\begin{proposition}[Derdzi\'nski]\label{prop:d2}
Let $(M,h)$ be an oriented Einstein 4-manifold that is of Type II. Then up to passing to a double cover, there exists a complex structure $J$ on $M$ such that $(M,h,J)$ is Hermitian and conformally K\"ahler. The conformal metric $g=\lambda^{2/3}h$ is K\"ahler under $J$, where
$$\lambda=2\sqrt{6}|W^+|_h.$$
It turns out that the scalar curvature $s_g$ of $g$ satisfies $|s_g|=\lambda^{1/3}$. The function $\lambda$ is positive everywhere because $W^+\neq0$ everywhere.
\end{proposition}
\begin{proof}
This proposition follows from Proposition 5 in \cite{derdzinski}. We briefly outline the construction of $J$ for the reader's convenience. Since $(M,h)$ is of Type II, $W^+$ has exactly two distinct eigenvalues at every point, and due to its tracelessness, two of its three eigenvalues coincide. Derdzi\'nski proved that the eigen-2-forms $\pm\omega$ of $W^+$ corresponding to the eigenvalue with multiplicity one, normalized to have unit length under $g$, are parallel with respect to the metric $g$. Passing to a double cover if necessary, one can pick a global eigen-two-form $\omega$.
Now the eigen-2-form $\omega$ serves as the K\"ahler form, and together with $g$, determine the complex structure $J$.
\end{proof}

The next lemma, which was first discovered by Goldberg and Sachs, shows that if an Einstein 4-dimensional metric is Hermitian but non-ASD, then it must be Type II. 
\begin{lemma}[Goldberg-Sachs]\label{lem:gs}
Let $(M,h,J)$ be a Hermitian Einstein 4-manifold, with the orientation given by the complex structure $J$. Then the self-dual Weyl curvature $W^+$ of $h$ must be $J$-invariant, hence $W^+:\Lambda^+\to\Lambda^+$ has at most two distinct eigenvalues at every point of $M$.
\end{lemma}

Now we move to Hermitian non-K\"ahler ALE gravitational instantons. Note that as a general fact the decay order $\tau$ is always at least 4.

\begin{theorem}[Bando-Kasue-Nakajima]\label{bkn}
Let $(M,h)$ be a 4-dimensional Ricci-flat manifold with Euclidean volume growth and finite $\int_M|Rm_h|_h^2$, then $(M,h)$ is ALE with order 4.
\end{theorem}

Let us next briefly discuss the associated K\"ahler metric $g$. Recall that for a Hermitian non-K\"ahler ALE gravitational instanton $(M,h)$, we define $\lambda:=2\sqrt{6}|W^+|_h$, and the K\"ahler metric is given by $g=\lambda^{2/3}h$. Straightforward calculation shows $|s_g|=\lambda^{1/3}>0$. By looking at the conformal change of scalar curvature equation and examining it at the maximum/minimum of $s_g$, it follows $s_g>0$, therefore $s_g=\lambda^{1/3}$. As the K\"ahler metric $g$ is conformal to a Ricci-flat metric, it is \textit{Bach-flat} K\"ahler, which therefore is \textit{extremal K\"ahler} in the sense of Calabi \cite{lebrun2}.
In this case, $(M,g)$ automatically possesses a Killing field $J\nabla_gs_g$. Since it preserves the conformal factor $\lambda^{1/3}$, $J\nabla_gs_g$ is also Killing with respect to $(M,h)$.

\begin{proposition}\label{killingfieldconformal}
The vector field $\mathcal{K}=J\nabla_g\lambda^{{1}/{3}}$ is a Killing field with respect to $g$, hence also a Killing field with respect to $h$.
\end{proposition}

From basic calculations, we find that $\mathcal{K}=-J\nabla_h\lambda^{-1/3}$. Due to the decay of the Riemannian curvature, the conformal factor $\lambda=2\sqrt{6}|W^+|_h$ decays as $\lambda=O_h'(\rho^{-6})$ as $\rho\to\infty$. Note $J\nabla_gs_g$ is the imaginary part of the holomorphic extremal vector field
$$J\nabla_gs_g=-2\Im\nabla^{1,0}_gs_g.$$ 

Next we introduce the Eguchi-Hanson metric but with reversed orientation, which is an example of Type II ALE gravitational instanton.
\begin{example}[Eguchi-Hanson]\label{example:eguchihanson1}
The Eguchi-Hanson metric was constructed in \cite{eguchihanson}. Let $\sigma_x,\sigma_y,\sigma_z$ be left-invariant orthogonal coframes for the sphere $S^3$. The Eguchi-Hanson metric can be explicitly written down as
\begin{equation}
h=\frac{1}{1-\left(\frac{a}{r}\right)^4}dr^2+r^2\left(\sigma_x^2+\sigma_y^2+\left(1-\left(\frac{a}{r}\right)^4\right)\sigma_z^2\right)
\end{equation}
with $r>a$, and the entire metric is obtained by considering the metric completion.
With the following orthonormal basis 
$$e^0=\sqrt{\frac{1}{1-\left(\frac{a}{r}\right)^4}}dr,\ e^1=r\sigma_x,\ e^2=r\sigma_y,\ e^3=r\sqrt{\left(1-\left(\frac{a}{r}\right)^4\right)}\sigma_z$$
and orientation given by $-e^0\wedge e^1\wedge e^2\wedge e^3$, the self-dual Weyl curvature was calculated in \cite{eguchihanson} as
\begin{equation}
W^+=\left(\begin{matrix}-\frac{2a^4}{r^6}& &\\&-\frac{2a^4}{r^6}&\\& & \frac{4a^4}{r^6}\end{matrix}\right).
\end{equation} 
Here we are using $e^1\wedge e^0+e^2\wedge e^3,\ e^2\wedge e^0+e^3\wedge e^1,e^3\wedge e^0+e^1\wedge e^2$ as a basis for $\Lambda^+$. So it is clear that the metric with this orientation is of Type II. The complex structure is given by $J:e^0\to e^3,e^2\to e^1$. The function $\lambda$ is $\lambda=2\sqrt{6}|W^+|=24\frac{a^4}{r^6}$. The conformal metric $g=\lambda^{2/3}h$ is K\"ahler and is incomplete. The scalar curvature of $g$ is $s_g=\lambda^{1/3}=2\sqrt[3]{3a^4}\frac{1}{r^2}$, which is decaying to 0.
Under the complex structure $J$, the Eguchi-Hanson space as a complex surface is $\mathcal{O}(2)$. It can be compactified to the Hirzebruch surface $H_2=\mathbb{P}(\mathcal{O}\oplus\mathcal{O}(2))$ by adding a curve $C_\infty=\mathbb{P}^1$ with self-intersection $-2$ at infinity, and this divisor can be contracted to an orbifold point with orbifold group $\mathbb{Z}_2$. The orbifold $\widehat{M}$ we get after the contraction is log del Pezzo, and the K\"ahler metric $g$ extends to $\widehat{M}$ as an orbifold K\"ahler metric $\widehat{g}$. The pair $(\widehat{M},\widehat{g})$ is the corresponding special Bach-flat K\"ahler orbifold for the Eguchi-Hanson space with reversed orientation, appearing in Theorem \ref{main:correspondence}.
\end{example}

\section{Compactification}
\label{sec:compactification}

In this section, for a Hermitian non-K\"ahler gravitational instanton $(M,h)$, we prove that the conformal K\"ahler space $(M,g)$ can be naturally compactified to a K\"ahler orbifold by adding one point. This establishs the correspondence stated in Theorem \ref{main:correspondence}.

\subsection{Metric completion and lower bound on the Killing field $\mathcal{K}$}

By fixing an ALE coordinate for the Hermitian non-K\"ahler gravitational instanton $(M,h)$, it is easy to conclude that there is an Euclidean Killing field $\mathcal{K}_\infty$ on $\mathbb{R}^4/\Upsilon$, such that $\mathcal{K}-\mathcal{K}_\infty=O'(\rho^{-3})$. The rate can be chosen as $3$ since the rate of $h$ asymptotic to the Euclidean metric is at least 4 by Bando-Kasue-Nakajima. The norm of the Killing field $\mathcal{K}=-J\nabla_h\lambda^{-1/3}$ at most grows like $\rho$ by basic comparison geometry over the ALE end. Integrating this from the base point $p$, one gets
$$\lambda^{-1/3}(x)\leq\lambda^{-1/3}(p)+\int_{px}|\nabla_h\lambda^{-1/3}|_h\leq\lambda^{-1/3}(p)+C\int_{px}\rho\leq C\rho^2(x)$$
for some constant $C$ when $\rho$ is large enough. This implies $\lambda\geq C\rho^{-6}$. With the curvature decay ensured by Bando-Kasue-Nakajima, it follows that
\begin{equation}\label{eq:lambdadecay}
C_1\rho^{-6}\leq\lambda\leq C_2\rho^{-6}
\end{equation}
holds on the end of $M$.

\begin{proposition}\label{prop:nonvanishing}
The Euclidean Killing field $\mathcal{K}_\infty$ is nowhere vanishig.
\end{proposition}
\begin{proof}
We prove the proposition by contradiction.
Suppose that the Euclidean Killing field $\mathcal{K}_\infty=\sum_{i,j}\alpha_{ij}x_i\frac{\partial}{\partial x_j}$ vanishes at some point. Then it must vanish along the entire ray, starting from origin passing through the point in $\mathbb{R}^4/\Upsilon$.
Along this ray we have
\begin{equation}\label{directionv}
\nabla_{h}\lambda^{-1/3}=J\mathcal{K}_\infty+O(\rho^{-3})=O(\rho^{-3}).
\end{equation}
Integrating (\ref{directionv}), one gets that along this ray
$\lambda^{-1/3}\leq C$, which contradicts the fact that $\lambda$ is decaying to 0.
\end{proof}
\begin{remark}
Equation (\ref{eq:lambdadecay}) implies the existence of constants $C_1$ and $C_2$ such that
\begin{align}
C_1\rho^{-6}\leq |W^+|_h\leq C_2\rho^{-6}.
\end{align}
This result has an interesting consequence: if the self-dual Weyl curvature of a Hermitian ALE gravitational instanton decays faster than $\rho^{-6}$, then the manifold must be K\"ahler. This shows Corollary \ref{cor:hermitiandecay}.
\end{remark}

The bound (\ref{eq:lambdadecay}) tells us that the conformal factor satisfies $C_1\rho^{-4}\leq \lambda^{2/3}\leq C_2\rho^{-4}$ at infinity, therefore the following proposition is a direct computation.

\begin{proposition}\label{removal1}
The metric completion of $(M,g)$ is  adding one point $q$ at infinity. 
\end{proposition}

The metric completion will be denoted by $\widehat{M}=M\cup\{q\}$.

\subsection{Removing singularity}
\label{sec:removal}

In this subsection we derive estimates on the curvature tensor $Rm_g$ of the K\"ahler metric $g$ near the added point $q$, in order to show that $g$ extends as a smooth orbifold metric across $q$.
We define $r_q$ to be the distance function to $q$ in $\widehat{M}$ equipped with the length space structure induced by the metric $g$. A tensor $T$ on $\widehat{M}\setminus\{q\}$ is said to be $O'_{g}(r_q^{\delta})$ if for any integer $k\geq0$, there exists a constant $C$ such that 
$|\nabla_g^kT|_g\leq Cr_q^{\delta-k}$ 
near the added point $q$.

\begin{theorem}\label{thm:removal}
The extremal K\"ahler metric $g$ on $M$ extends to a smooth extremal K\"ahler orbifold metric $\widehat{g}$ on $\widehat{M}$. The metric $\widehat{g}$ is Bach-flat.
\end{theorem}

The rest of this subsection is devoted to proving this theorem.
Only need to show that $g$ can be extended smoothly. This theorem already establishes one direction of the correspondence in Theorem \ref{main:correspondence}.

\begin{proposition}\label{curvaturebound}
$|Rm_g|_g$ is bounded near $q$ and $\nabla_g Rm_g=O_g'(r_q)$.
\end{proposition}
\begin{proof}
Since the Riemannian curvature decomposes to the scalar curvature part, the traceless Ricci curvature part, and the Weyl curvature part, we only need to treat them separately. In the proof $\nabla$ refers to the Levi-Civita connection of $g$.
For the scalar curvature, $s_g=\lambda^{1/3}=O'_h(\rho^{-2})$. It is a direct calculation to check that $s_g=O'_g(r_q^2)$ near $q$, and particularly it is bounded. 
For the Weyl curvature, $W=O_h'(\rho^{-2})$. Based on the conformal invariance of the Weyl curvature, it is again straightforward to compute that $W=O_g'(r_q^2)$.

The estimate on the Ricci curvature however does not follow directly from the conformal change.
The following computation was used by Tanno \cite{tanno}.
In dimension 4 we have
\begin{equation}
2\nabla_{\alpha}W^{\alpha}_{\beta\gamma\eta}=(\nabla_{\eta}R_{\beta\gamma}-\nabla_{\gamma}R_{\beta\eta})-\frac16(g_{\beta\gamma}\nabla_{\eta}s_g-g_{\beta\eta}\nabla_{\gamma}s_g).
\end{equation}
Here $R_{\alpha\beta}$ is the Ricci curvature.
This shows
\begin{align}\label{eq:commutericci}
\nabla_{\eta}R_{\beta\gamma}-\nabla_{\gamma}R_{\beta\eta}&=2\nabla_{\alpha}W^{\alpha}_{\beta\gamma\eta}+\frac16(g_{\beta\gamma}\nabla_{\eta}s_g-g_{\beta\eta}\nabla_{\gamma}s_g)\notag\\
&=O_g'(r_q),
\end{align}
as a $(3,0)$ tensor.
The K\"ahler condition gives that the Ricci curvature is $J$-invariant and the complex structure $J$ is parallel, so we have
$$R_{\beta\gamma}=R_{\mu\nu}J^\mu_{\beta}J^{\nu}_{\gamma}.$$
\begin{equation}\label{eq:bg}
\nabla_{\alpha}R_{\beta\gamma}=\nabla_{\alpha}R_{\mu\nu}J^\mu_{\beta}J^{\nu}_{\gamma}.
\end{equation}
Therefore by applying (\ref{eq:commutericci}),
\begin{align}\label{eq:ag}
\nabla_{\alpha}R_{\beta\gamma}&=
\nabla_\beta R_{\alpha\gamma}+O_g'(r_q)\notag\\
&=\nabla_\beta R_{\mu\nu}J^\mu_\alpha J^\nu_\gamma+O_g'(r_q)\notag\\
&=\nabla_{\mu}R_{\beta\nu}J_{\alpha}^{\mu}J_{\gamma}^{\nu}+O_g'(r_q),
\end{align}
Similarly,
\begin{equation}\label{eq:ab}
\nabla_{\alpha}R_{\beta\gamma}=\nabla_{\mu}R_{\nu\gamma}J_{\alpha}^\mu J_{\beta}^\nu+O_g'(r_q).
\end{equation}
Combining (\ref{eq:bg}), (\ref{eq:ag}), and (\ref{eq:ab}), we have
\begin{align}
\nabla_{\alpha}R_{\beta\gamma}&=\nabla_{\mu}R_{\nu\gamma}J_{\alpha}^\mu J_{\beta}^\nu+O_g'(r_q)\notag\\
&=\nabla_{\mu}R_{\phi\psi}J_{\nu}^\phi J_{\gamma}^{\psi}J_{\alpha}^\mu J_{\beta}^\nu+O_g'(r_q)\notag\\
&=\nabla_{\theta}R_{\phi\tau}J_{\mu}^\theta J_{\psi}^{\tau}J_{\nu}^\phi J_{\gamma}^{\psi}J_{\alpha}^\mu J_{\beta}^\nu+O_g'(r_q)\notag\\
&=-\nabla_{\alpha}R_{\beta\gamma}+O_g'(r_q).
\end{align}
Thus $\nabla_\alpha R_{\beta\gamma}=O_g'(r_q)$ as a $(3,0)$ tensor, which shows the Ricci curvature of $g$ is also bounded near $q$.
\end{proof}

\begin{proposition}\label{volumebound}
There are positive constants $V_1,V_2$ such that
\begin{equation}
V_1r^4\leq \mathrm{Vol}(B_{g}(q,r))\leq V_2r^4.
\end{equation}
\end{proposition}
\begin{proof}
This is because $g=\lambda^{2/3}h$ with $h$ ALE and $C_1\rho^{-6}<\lambda<C_2\rho^{-6}$.
\end{proof}
\begin{proposition}
The tangent cone of $(\widehat{M},g)$ at $q$ is unique, isometric to a Euclidean quotient $\mathbb{R}^4/\Gamma$. Moreover, the convergence to the tangent cone is in $C^\infty$ topology.
\end{proposition}
\begin{proof}
Propositions \ref{curvaturebound} and \ref{volumebound} imply that each tangent cone must be a flat cone, thus is isomorphic to some $\mathbb{R}^4/\Gamma$. The estimate $|\nabla_g^{k}Rm_g|_g=O_g'(r_q^{2-k})$ shows that when we consider the rescaled spaces $(\widehat{M},s_i^{-2}g,q)$ with $s_i\to0$, the quantity $|\nabla^k_{s_i^{-2}g}Rm_{s_i^{-2}g}|_{s_i^{-2}g}$ is uniformly bounded, from which we obtain the $C^\infty$ convergence to each tangent cone.
\end{proof}

It is easy to see the specific relation between $\Upsilon$ and $\Gamma$. Let $\mathbb{R}^4/\Upsilon$ be equipped with the Euclidean metric $g_\Upsilon$. Then the conformal metric $\frac{1}{r^4}g_{\Upsilon}$ is again flat, with $r$ being the distance function to the origin. With the conformal flat metric $\frac{1}{r^4}g_{\Upsilon}$, it is isometric to another Euclidean quotient (after completion), which is exactly $\mathbb{R}^4/\Gamma$.
Denote by $B_\Gamma$ the standard metric ball in $\mathbb{R}^4/\Gamma$ with radius 1 and $B_{\Gamma}^*$ the corresponding punctured ball.
\begin{proposition}\label{c1a}
There is a $C^3$ diffeomorphism $E:B_\Gamma^*\to B_1(q)\setminus\{q\}$, such that $E^*g$ extends to a $C^{1,\alpha}$ orbifold metric on $B_\Gamma$.
\end{proposition}
\begin{proof}
The information about the tangent cone already tells us we can find a smooth diffeomorphism $E':B_{\Gamma}^*\to B_1(q)\setminus\{q\}$, such that $E'^*g$ extends to a $C^0$ orbifold metric. See Proposition 5.10 in \cite{donaldsonsun}. Then with the bound on the Riemannian curvature, Proposition 5.14 in \cite{donaldsonsun} gives us the desired $C^3$ diffeomorphism $E$.
\end{proof}

From now on, by passing to the orbifold cover, we can always assume that we are working in the case where $\Gamma=\{e\}$.
Since $g$ is $C^{1,\alpha}$, the complex structure, which is compatible with $g$, is also $C^{1,\alpha}$. The Newlander-Nirenberg integrability theorem says, modifying by a $C^{2,\alpha}$ diffeomorphism, the complex structure $J$ is standard near the point $q$, meaning that there is a $C^{2,\alpha}$ coordinates system near $q$ in which the complex structure is standard. In the following we will work in this standard complex coordinate.

\begin{proposition}\label{prop:scal}
In the above complex coordinate, there are functions $f_i$ that are smooth in terms of the complex coordinate, such that 
$\frac{\partial s_g}{\partial \overline{z}_{j}}=f_ig_{i\overline{j}}$.
In particular, when $g_{i\overline{j}}$ is $C^{p,\alpha}$ in the fixed complex coordinates system, $s_g$ is $C^{p+1,\alpha}$. 
\end{proposition}
\begin{proof}
The K\"ahler metric $g$ is $C^{1,\alpha}$ in this complex coordinate.  The holomorphic vector field $\nabla^{1,0}s_g=\frac{\partial s_g}{\partial \overline{z}_j} g^{i\overline{j}}\frac{\partial}{\partial z_i}$ in this coordinate  can be extended across the origin by the Hartogs' theorem, since $\nabla s_g$ is bounded  by Proposition \ref{curvaturebound}. In particular, coefficients of this holomorphic vector field, $\frac{\partial s_g}{\partial \overline{z}_j}g^{i\overline{j}}$, must be $C^{\infty}$ functions in this coordinate system. 
From this, setting $f_i=\frac{\partial s_g}{\partial \overline{z}_j}g^{i\overline{j}}$,
one gets
$\frac{\partial s_g}{\partial \overline{z}_k}=f_ig_{i\overline{k}}$.
The proposition follows.
\end{proof}

Next we improve the regularity of the metric around the point $q$. Proposition \ref{c1a} already shows that we can treat $g$ as a $C^{1,\alpha}$ metric on $B_1(q)$. 
We fix a base K\"ahler metric $h_{i\overline{j}}$ that is smooth in the complex coordinate. Then 
the scalar curvature $s(h_{i\overline{j}}+\varphi_{i\overline{j}})$ of any K\"ahler  metric $h_{i\overline{j}}+\varphi_{i\overline{j}}$ can be computed via
\begin{equation}\label{eq:MA}
\Det\left(h_{i\overline{j}}+\varphi_{i\overline{j}}\right)=e^F\Det(h_{i\overline{j}}),
\end{equation}
\begin{equation}\label{eq:scal}
s\left(h_{i\overline{j}}+\varphi_{i\overline{j}}\right)=-\Delta_\varphi F+\Tr_\varphi Ric(h).
\end{equation}
Here, $\Delta_\varphi$ denotes the the Laplace operator of the K\"ahler metric $h_{i\overline{j}}+\varphi_{i\overline{j}}$.

\begin{proposition}\label{prop:MAbootstrap}
For $p\geq3$, $C^{p,\alpha}$ bound on $F$ and $C^{p,\alpha}$ bound on $\varphi$ give $C^{p+1,\alpha}$ bound on $\varphi$. 
\end{proposition}
\begin{proof}
This follows from a standard bootstrap arguments for the Monge-Amp\`ere equation (\ref{eq:MA}).
\end{proof}

\begin{proposition}
The metric $g$ on $B_1(q)\setminus\{q\}$ extends to a metric on $B_1(q)$ that is smooth in the standard complex coordinate.
\end{proposition}
\begin{proof}
In the complex coordinate, we have $C^{1,\alpha}$ bound on $g_{i\overline{j}}$. By choosing the potential function $\varphi$ suitably, it can be assumed that $\varphi$ is  $C^{3,\alpha}$ bounded, where $g_{i\overline{j}}=h_{i\overline{j}}+\varphi_{i\overline{j}}$. 

Starting with $C^{k+2,\alpha}$ bound on $\varphi$, we have $C^{k,\alpha}$ bound on $\varphi_{i\overline{j}}$. Together with Proposition \ref{prop:scal}, the elliptic equation for $F$ (\ref{eq:scal}) at least has $C^{k,\alpha}$ bounded coefficients. Schauder estimate gives $C^{k+2,\alpha}$ bound on $F$. Choosing $p=k+2$ in Proposition \ref{prop:MAbootstrap}, one gets $C^{k+3,\alpha}$ bound on $\varphi$. This shows, beginning from $k=1$, by applying this bootstrap argument repeatly, we can finally get $C^{\infty}$ bound on $\varphi$ in the  standard complex coordinate. The proposition is proved.
\end{proof}

This finishes the proof of Theorem \ref{thm:removal}. 
The extended complex structure on $\widehat{M}$ is denoted by $\widehat{J}$.
Theorem \ref{thm:removal} has the following corollaries.

\begin{corollary}
On $(\widehat{M},\widehat{g})$,
the Killing field $\mathcal{K}=J\nabla_gs_g$ extends to  the Killing field $\widehat{J}\nabla_{\widehat{g}}s_{\widehat{g}}$, which vanishes at the orbifold point $q$. 
Moreover, $s_{\widehat{g}}\geq0$ and $s_{\widehat{g}}=0$ only at $q$. The scalar curvature $s_{\widehat{g}}$ is a Morse-Bott function. 
\end{corollary}
\begin{proof}
The vanishing of the extended Killing field at $q$ follows from the fact that $|\mathcal{K}|_g\leq C\rho^{-1}$. The vanishing of $s_{\widehat{g}}(q)$ follows from the decay of $s_g=\lambda^{1/3}$ as $O(\rho^{-2})$. Outside of $q$, $s_{\widehat{g}}>0$ because $\lambda=2\sqrt{6}|W^+|_h>0$. The last statement follows from Lemma 1 in \cite{lebrun2}.
\end{proof}

\begin{corollary}\label{behaviorofs}
Near the point $q$, we have
$s_{\widehat{g}}=\frac12ar_q^2+O(r_q^3)$
for some positive constant $a$. There exists a constant $t_0$ such that the time $t_0$ flow of the Killing field $\widehat{J}\nabla_{\widehat{g}}s_{\widehat{g}}$ is the identity map.
\end{corollary}
\begin{proof}
Since the Killing field $\widehat{J}\nabla_{\widehat{g}}s_{\widehat{g}}$ vanishes at $q$, we have $\nabla_{\widehat{g}}s_{\widehat{g}}(q)=0$.
The Morse-Bott property implies that $\Hess(s_{\widehat{g}})$ is nondegenerate at $q$. 
As the K\"ahler metric is Bach-flat, there is the equation 
$$0=s_gRic_g^0+2Hess_0(s_g),$$
which is equation (11) in \cite{lebrun2}. So we have $\Hess_0(s_{\widehat{g}})(q)=0$. Therefore, $\Hess(s_{\widehat{g}})(q)=a\widehat{g}$ for some positive $a$, and we conclude that $s_{\widehat{g}}=\frac12ar_q^2+O(r_q^3)$ near $q$.
Because $\Hess(s_{\widehat{g}})(q)=a\widehat{g}$, the time $2\pi/a$ flow of $\widehat{J}\nabla_{\widehat{g}}s_{\widehat{g}}$ fixes the point $q$ and its tangent space, meaning that the time $2\pi/a$ flow of this Killing field is the identity map.
\end{proof}

\subsection{Complex structure on the compactification}
\label{subsec:complexstructure}

The special property of $s_{\widehat{g}}$ leads to the following proposition, which is based on a key observation by LeBrun in \cite{lebrun}.
\begin{proposition}\label{logdel}
The underlying complex orbifold
$(\widehat{M},\widehat{J})$ is a log del Pezzo surface. That is, $(\widehat{M},\widehat{J})$ is a normal projective surfaces with at worst quotient singularities and ample anticanonical bundle.
\end{proposition}
\begin{proof}
First, we show the existence of a K\"ahler current in $2\pi c_1(-K_{\widehat{M}})$. Then we prove $-K_{\widehat{M}}$ is positive using regularization. Note that although $-K_{\widehat{M}}$ is only $\mathbb{Q}$-Cartier, for simplicity we will work as if it were Cartier.

The conformal relation $g=\lambda^{2/3}h=s_g^2h$  and the Ricci-flat property of $h$ imply 
\begin{equation}\label{conformalricci}
Ric_{g,ab}=-2s_g^{-1}\nabla_{a}\nabla_{b}s_g+\left(-s_g^{-1}\Delta s_g+3|d\log s_g|_g^2\right)g_{ab}.
\end{equation}
Here we are using the Levi-Civita connection of $g$. 
Taking the trace of (\ref{conformalricci}), we get
$s_g=-6s_g^{-1}\Delta s_g-12|d\log s_g|_g^2$.
Therefore,
\begin{equation}\label{classricci}
Ric_{g,ab}+2s_g^{-1}\nabla_a\nabla_bs_g=\left(\frac{s_g}{6}+|d\log s_g|_g^2\right)g_{ab}.
\end{equation}
This suggests that we should consider the current
$T=\rho_g+2\sqrt{-1}\partial\bar{\partial}\log{s_g}$
in the class $2\pi c_1(-K_{\widehat{M}})$, where $\rho_g$ is the Ricci form of $g$. With (\ref{classricci}), the associated symmetric form $T_{J}:= T(\cdot,J\cdot)$ is given by
\begin{equation}
T_{J,ab}=\frac{s_g}{6}g_{ab}+s_g^{-2}\left(|ds_g|_g^2g_{ab}-(ds_g)_a(ds_g)_b-(Jds_g)_a(Jds_g)_b\right).
\end{equation}
It is clear that $T_J$ is strictly positive, hence the current $T$ is a K\"ahler current.

To show the ampleness, if $\widehat{g}_{-K}$ denotes the hermitian metric on the line bundle $-K_{\widehat{M}}$ induced by $\widehat{g}$, then the form $\rho_g+2\sqrt{-1}\partial\bar{\partial}\log s_g$ is the curvature form of the singular hermitian metric $\widehat{g}_{-K}e^{-2\log{s_g}}$. The current $T$ as the curvature of the singular metric $g_{-K}e^{-2\log s_g}$ on $-K_{\widehat{M}}$, having Lelong number 
\begin{equation}
\nu(T,x)=
\left\{
\begin{aligned}
0,\ \text{if $x\neq q$,}\\
2,\ \text{if $x=q$,}
\end{aligned}
\right.\notag
\end{equation}
because of Corollary \ref{behaviorofs}. Taking a  Riemannian normal coordinate near $q$,  we have
$$\widehat{g}_{ij}=\delta_{ij}+O(r_q^2),\ \partial_k \widehat{g}_{ij}=O(r_q),\ 
s_{\widehat{g}}=\frac12ar_q^2+O(r_q^3).$$
So,
\begin{align*}
\frac{s_g}{6}g_{ab}+s_g^{-2}\left(|ds_g|^2g_{ab}-(ds_g)_a(ds_g)_b-(Jds_g)_a(Jds_g)_b\right)
\geq\frac{1}{12}ar_q^2g_{ab}+\epsilon\frac{1}{r_q^2}g_{ab}
\end{align*}
for some small $\epsilon$,
as a symmetric 2-form near $q$. 
It is a consequence of the standard regularization theorem of Demailly that there is 
a smooth $(1,1)$ form $\theta$ in $2\pi c_1(-K_{\widehat{M}})$
that is strictly positive.
\end{proof}

Classical result now says that $(\widehat{M},\widehat{J})$ is rational, since it is log del Pezzo.

\begin{proposition}\label{lem:theminimalresolution}
The minimal resolution $\widetilde{M}$ of $\widehat{M}$, denoted as $r:\widetilde{M}\to \widehat{M}$, is a smooth rational surface.
\end{proposition}

\subsection{The correspondence}\label{sec:correspondence}

The previous discussion has demonstrated that the compactification of a Hermitian non-K\"ahler ALE gravitational instanton is a special Bach-flat K\"ahler orbifold. To complete the proof of Theorem \ref{main:correspondence}, we still need to prove the other direction of our correspondence.
\begin{theorem}\label{correspondence}
Given a special  Bach-flat K\"ahler orbifold $(\widehat{M},\widehat{g})$ with the orbifold point $q$, there exists a Hermitian non-K\"ahler ALE gravitational instanton $(M,h)$, where $M=\widehat{M}\setminus\{q\}$ and $h=s_{\widehat{g}}^{-2}\widehat{g}$. 
\end{theorem}

\begin{proof}
Because $\widehat{g}$ is  Bach-flat K\"ahler, after the conformal change, we have
$Ric_{h}^0=Ric_{\widehat{g}}^0+2s_{\widehat{g}}^{-1}\Hess_0(s_{\widehat{g}})=0$. 
The vanishing of the traceless Ricci curvature implies $h$ is Einstein. Conformal relation also gives
$$s_h=s_{\widehat{g}}^3+6s_{\widehat{g}}\Delta s_{\widehat{g}}-12|\nabla s_{\widehat{g}}|^2.$$
The right side is a continuous function on $\widehat{M}$ which should be a constant because $s_h$ is a constant. Since $s_{\widehat{g}}$ achieves its minimum at $q$, $\nabla s_{\widehat{g}}(q)=0$.  Evaluating this equation at $q$, we get $s_h=0$.  Therefore $h$ is a Ricci-flat metric.
Near $q$, $s_{\widehat{g}}=\frac12ar_q^2+O(r_q^3)$. Hence $h=s_{\widehat{g}}^{-2}\widehat{g}$ has Euclidean volume growth. It is also direct to check that $\int|Rm_h|_h^2<\infty$ over the end.
By \cite{bkn}, it is ALE.
\end{proof}

This finshes the proof of Theorem \ref{main:correspondence}.

\section{Holomorphic vector fields and minimal resolutions}
\label{sec:holfields}

In Section \ref{sec:holfields}, we take a detour to study weights of holomorphic vector fields, which is necessary for us to classify all the log del Pezzo surfaces coming from Hermitian non-K\"ahler ALE gravitational instantons. 
We conclude previous sections before we start.
With the correspondence Theorem \ref{main:correspondence}, to study Hermitian non-K\"ahler ALE gravitational instantons $(M,h)$, it suffices to study special Bach-flat K\"ahler orbifolds $(\widehat{M},\widehat{g})$. Some key features of special Bach-flat K\"ahler orbifolds are:
\begin{itemize}
\item The underlying complex surface $\widehat{M}$ is log del Pezzo, with only one orbifold point $q$.
\item The underlying complex surface $\widehat{M}$ carries a holomorphic $\mathbb{C}^*$-action: the extremal vector field $\mathcal{K}$ induces a holomorphic $S^1$-action by Corollary \ref{behaviorofs}, so there is the associated holomorphic $\mathbb{C}^*$-action.
\item The holomorphic  $\mathbb{C}^*$-action has same weights at the orbifold point $q$, in the sense that it fixes $q$ and the two weights of the induced $\mathbb{C}^*$-action on the tangent space at $q$ coincide, because of Corollary \ref{behaviorofs}.
\end{itemize}
These already are very strong constraints. Denote the infinitesimal generator of the holomorphic $\mathbb{C}^*$-action as $\mathfrak{E}$.
The holomorphic $\mathbb{C}^*$-action always can be taken as primitive and the holomorphic vector field $\mathfrak{E}$ can always be taken as flowing out of the orbifold point $q$. We shall ultimately classify all pairs $(\widehat{M},\mathfrak{E})$ that satisfy the above conditions, with one more technical assumption that the orbifold group is in $SU(2)$.
They are candidates for special Bach-flat K\"ahler orbifolds with structure group $\Gamma\subset SU(2)$.

Before proceeding, we fix some notations. By a $\mathbb{C}^*$-action on $\mathbb{C}^2/\Gamma$ given by $t\curvearrowright(x,y)=(t^\theta x,t^\tau y)$, we mean that after passing to the orbifold cover $\mathbb{C}^2$, the action is given by $t\curvearrowright(x,y)=(t^\theta x,t^\tau y)$. The action must commute with the $\Gamma$-action on $\mathbb{C}^2$. By the infinitesimal generator of a $\mathbb{C}^*$-action, we mean the holomorphic vector field generated by $\partial_z$.

\subsection{Quotient singularities and their minimal resolutions}\label{sec:minimalresolutions}

Minimal resolutions of quotient singularities $\mathbb{C}^2/\Gamma$ were studied by Brieskorn. In this section we briefly recall their structures. Materials in this section are taken from Lock-Viaclovsky \cite{lv2,lv}.

Consider the quotient singularity $\mathbb{C}^2/\Gamma$, where $\Gamma\subset U(2)$ is a finite subgroup and we require that $\mathbb{C}^2/\Gamma$ only has an orbifold singularity at the origin. 
The group $SU(2)$ can be identified with unit quaternions $z_1+z_2j\in\mathbb{H}$, where $z_1,z_2\in\mathbb{C}$. This identification allows us to view $SU(2)$ as the unit sphere $S^3$. Define the map $\phi:S^3\times S^3\to SO(4)$ by 
$$\phi(q_1,q_2)(h)=q_1h\overline{q_2}$$ 
for $h\in\mathbb{H}$, where we are taking quaternions multiplication. This map $\phi$ is a double cover of $SO(4)$, and when we restrict it to $S^1\times S^3$ with $S^1\subset S^3$ understood as the set of unit $z_1\in \mathbb{C}\subset\mathbb{H}$, it provides a double cover of $U(2)$. Now finite subgroups of $U(2)$ acting freely on $S^3$ are described by Table \ref{subgroups}.

\begin{table}[ht]
\centering
\caption{Finite subgroups of $U(2)$ acting freely on $S^3$.}
\label{subgroups}
\begin{tabular}{ l l l }
\hline
 $\Gamma\subset U(2)$ & Conditions & Order \\ 
\hline\hline
 $L(m,n)$ & $(m,n)=1$ & $n$ \\  
\specialrule{0em}{2pt}{2pt}
 $\phi(L(1,2m)\times D_{4n}^*)$ & $(m,2n)=1$ & $4mn$ \\
\specialrule{0em}{2pt}{2pt}
$\phi(L(1,2m)\times T^*)$ & $(m,6)=1$ & $24m$\\
\specialrule{0em}{2pt}{2pt}
$\phi(L(1,2m)\times O^*)$ & $(m,6)=1$ & $48m$\\
\specialrule{0em}{2pt}{2pt}
$\phi(L(1,2m)\times I^*)$ & $(m,30)=1$ & $120m$\\
\specialrule{0em}{2pt}{2pt}
$\mathfrak{J}^2_{m,n}=$Index-2 diagonal $\subset\phi(L(1,4m)\times D_{4n}^*)$ & $(m,2)=2,(m,n)=1$ & $4mn$\\
\specialrule{0em}{2pt}{2pt}
$\mathfrak{J}^3_m=$Index-3 diagonal $\subset\phi(L(1,6m)\times T^*)$ & $(m,6)=3$ & $24m$\\  
\hline
\end{tabular}
\end{table}

Here, the group $L(q,p)$ with coprime $p,q$ denotes the cyclic subgroup of $U(2)$ generated by
$$\left(
\begin{matrix}
\exp(\frac{2\pi i}{p})&0\\
0&\exp(\frac{2\pi iq}{p})
\end{matrix}
\right).$$
The finite subgroups of $SU(2)$, denoted by $D_{4n}^*$, $T^*$, $O^*$, $I^*$, correspond to the binary dihedral, tetrahedral, octahedral, and icosahedral groups, respectively. Using the $ADE$ classification for finite subgroups of $SU(2)$, we have  that $A_n$ corresponds to $L(-1,n+1)$, $D_{n+2}$ corresponds to $D_{4n}^*$, and $E_6, E_7,E_8$ corresponds to $T^*, O^*, I^*$, respectively.
Set $[\alpha,\beta]=\phi(\alpha,\beta)$ with $\alpha,\beta\in SU(2)$ for simplicity.
We can write down the generators of the above groups in a more explicit manner as in Table \ref{generators}.

\begin{table}[ht]
\centering
\caption{Generators of finite subgroups of $U(2)$ acting freely on $S^3$.}
\label{generators}

\begin{tabular}{ l l }
\hline
 $\Gamma\subset U(2)$ & Generators \\ 
\hline\hline
\specialrule{0em}{1pt}{1pt}
 $L(q,p)$ & $[e^{2\pi ik/p},e^{2\pi i(1-k)/p}]$ \small{with $2k\equiv(q+1)\mathrm{mod}\ p$}\\  
\specialrule{0em}{2pt}{2pt}
 $\phi(L(1,2m)\times D_{4n}^*)$ & $[e^{\pi i/m},1],[1,e^{\pi i/n}],[1,j]$ \\
\specialrule{0em}{2pt}{2pt}
$\phi(L(1,2m)\times T^*)$ & $[e^{\pi i/m},1],[1,(1+i+j-k)/2],[1,(1+i+j+k)/2]$\\
\specialrule{0em}{2pt}{2pt}
$\phi(L(1,2m)\times O^*)$ & $[e^{\pi i/m},1],[1,e^{\pi i/4}],[1,(1+i+j+k)/2]$\\
\specialrule{0em}{2pt}{2pt}
$\phi(L(1,2m)\times I^*)$ & $[e^{\pi i/m},1],[1,(1+\tau i-\tau^{-1}k)/2],[1,(\tau+i+\tau^{-1}j)/2]$\\
 &\small{with $\tau=(1+\sqrt{5})/2$}\\
\specialrule{0em}{2pt}{2pt}
$\mathfrak{J}^2_{m,n}$ & $[e^{\pi i/m},1],[1,e^{\pi i/n}],[e^{\pi i/(2m)},j]$\\
\specialrule{0em}{2pt}{2pt}
$\mathfrak{J}^3_m$ & $[e^{\pi i/m},1],[1,i],[1,j],[e^{\pi i/(3m)},(-1-i-j+k)/2]$\\  
\hline
\end{tabular}
\end{table}

Now we describe the structure of the minimal resolution of $\mathbb{C}^2/\Gamma$. 

For the cyclic case, when the group $\Gamma$ is $L(q,p)$, the orbifold singularity is a Hirzebruch-Jung singularity. The exceptional divisors of its minimal resolution is a chain of rational curves with self-intersection $-e_i$, as illustrated below
$$
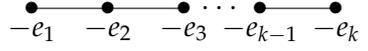
\begin{tikzpicture}
[decoration={markings, 
    mark= at position 0.5 with {\arrow{stealth}}}
]
\draw (0,0)--(1,0);
\draw (1,0)--(2,0);
\draw (3,0)--(4,0);

\node at (2.5,0){$\cdots$};
\node at (0,0){$\bullet$};
\node at (1,0){$\bullet$};
\node at (2,0){$\bullet$};
\node at (3,0){$\bullet$};
\node at (4,0){$\bullet$};

\node[below] at (0,0){$-e_1$};
\node[below] at (1,0){$-e_2$};
\node[below] at (2,0){$-e_3$};
\node[below] at (3,0){$-e_{k-1}$};
\node[below] at (4,0){$-e_k$};
\end{tikzpicture}.
$$
Each vertex denotes a rational curve with self-intersection $-e_i$, and if there is a segment between two vertices, it means the two curves intersect transversely at one point. The numbers $e_i$ are determined by the relatively prime integers $1\leq q<p$, via the Hirzebruch-Jung continued fraction expansion with $e_i\geq2$:
\begin{equation}\label{continuefraction}
\frac{q}{p}=\frac{1}{\displaystyle e_1-\frac{1}{\displaystyle e_2-\cdots\frac{1}{e_k}}}.
\end{equation}

The minimal resolutions of non-cyclic finite subgroups of $U(2)$ that act freely on $S^3$ have exceptional curves consisting of three chains of rational curves, each intersecting a single central rational curve.
$$
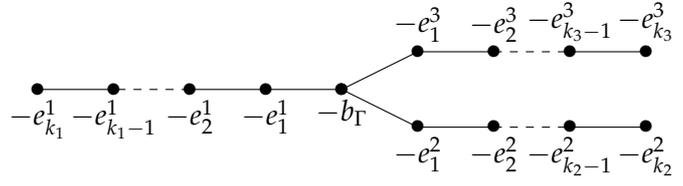
\begin{tikzpicture}
[decoration={markings, 
    mark= at position 0.5 with {\arrow{stealth}}}
]
\draw (0,0)--(1,-0.5);
\draw (1,-0.5)--(2,-0.5);
\draw[dashed] (2,-0.5)--(3,-0.5);
\draw (3,-0.5)--(4,-0.5);

\node at (0,0){$\bullet$};
\node at (1,-0.5){$\bullet$};
\node at (2,-0.5){$\bullet$};
\node at (3,-0.5){$\bullet$};
\node at (4,-0.5){$\bullet$};

\node[below] at (0,0){$-b_\Gamma$};
\node[below] at (1,-0.5){$-e_1^2$};
\node[below] at (2,-0.5){$-e_2^2$};
\node[below] at (3,-0.5){$-e_{k_2-1}^2$};
\node[below] at (4,-0.5){$-e_{k_2}^2$};

\draw (0,0)--(-1,0);
\draw (-1,0)--(-2,0);
\draw[dashed] (-2,0)--(-3,0);
\draw (-3,0)--(-4,0);

\node at (0,0){$\bullet$};
\node at (-1,0){$\bullet$};
\node at (-2,0){$\bullet$};
\node at (-3,0){$\bullet$};
\node at (-4,0){$\bullet$};

\node[below] at (-1,0){$-e_1^1$};
\node[below] at (-2,0){$-e_2^1$};
\node[below] at (-3,0){$-e_{k_1-1}^1$};
\node[below] at (-4,0){$-e_{k_1}^1$};

\draw (0,0)--(1,0.5);
\draw (1,0.5)--(2,0.5);
\draw[dashed] (2,0.5)--(3,0.5);
\draw (3,0.5)--(4,0.5);

\node at (1,0.5){$\bullet$};
\node at (2,0.5){$\bullet$};
\node at (3,0.5){$\bullet$};
\node at (4,0.5){$\bullet$};

\node[above] at (1,0.5){$-e_1^3$};
\node[above] at (2,0.5){$-e_2^3$};
\node[above] at (3,0.5){$-e_{k_3-1}^3$};
\node[above] at (4,0.5){$-e_{k_3}^3$};
\end{tikzpicture}.
$$
Each chain $\{e_i^j\}$ with fixed $j=1,2,3$ is the chain of exceptional curves of the minimal resolution for some $L(\alpha_j,\beta_j)$. The number $b_\Gamma$ here is given by
$$b_\Gamma=2+\frac{4m}{|\Gamma|}\left(m-\left(m\ \mathrm{mod}\ \frac{|\Gamma|}{4m}\right)\right),$$
with $m$ as in Table \ref{subgroups}.
The groups $L(\alpha_j,\beta_j)$ for each $\Gamma$ is given by Table \ref{noncyclicL}.

\begin{table}[h]
\centering
\caption{$L(\alpha_j,\beta_j)$ for non-cyclic subgroups.}
\label{noncyclicL}
\begin{tabular}{ l  l }
\hline
 $\Gamma\subset U(2)$   & $L(\alpha_j,\beta_j)$ \\ 
\hline\hline
 $\phi(L(1,2m)\times D_{4n}^*)$  & $L(1,2),L(1,2),L(-m,n)$\\
\specialrule{0em}{2pt}{2pt}
$\phi(L(1,2m)\times T^*)$  & $L(1,2),L(-m,3),L(-m,3)$\\
\specialrule{0em}{2pt}{2pt}
$\phi(L(1,2m)\times O^*)$  &$L(1,2),L(-m,3),L(-m,4)$\\
\specialrule{0em}{2pt}{2pt}
$\phi(L(1,2m)\times I^*)$   & $L(1,2),L(-m,3),L(-m,5)$\\
\specialrule{0em}{2pt}{2pt}
$\mathfrak{J}^2_{m,n}$  & $L(1,2),L(1,2),L(-m,n)$\\
\specialrule{0em}{2pt}{2pt}
$\mathfrak{J}^3_m$ & $L(1,2),L(1,3),L(2,3)$\\  
\hline
\end{tabular}
\end{table}

\subsection{Weights of holomorphic vector fields at fixed points}\label{sec:weightsofholo}

For a primitive holomorphic $\mathbb{C}^*$-action and its infinitesimal generator $\mathfrak{E}$, we define
\begin{definition}
The weights of the holomorphic vector field $\mathfrak{E}$ at a fixed point $p$ is the pair $[\theta,\tau]$, where $\theta$ and $\tau$ are the weights of the $\mathbb{C}^*$-action on the tangent space at $p$. If $p$ is an orbifold point, weights can be similarly defined by passing to the local orbifold cover.
\end{definition}

Weights could be fractions when $p$ is an orbifold point. For example, the primitive $\mathbb{C}^*$-action on $\mathbb{C}^2/\mathbb{Z}_2$, given by $t\curvearrowright(x,y)=(t^{1/2} x,t^{1/2} y)$ on the orbifold cover $\mathbb{C}^2$, has weights $[1/2,1/2]$ at the origin.
The following simple lemma says near a smooth fixed point $q$, there is a holomorphic coordinate such that the holomorphic $\mathbb{C}^*$-action with weights $[\theta,\tau]$ is standard.
\begin{lemma}\label{localformlemma}
There exists a holomorphic coordinate $(x,y)$ near $p$ such that the $\mathbb{C}^*$-action is given by $t\curvearrowright(x,y)=(t^\theta x,t^\tau y)$, and $\mathfrak{E}=\theta x\frac{\partial}{\partial x}+\tau y\frac{\partial}{\partial y}$.
\end{lemma}

The following proposition explains how blow-up changes weights of a holomorphic $\mathbb{C}^*$-action.
\begin{proposition}\label{prop:blowupweights}
For a non-trivial holomorphic $\mathbb{C}^*$-action with weights $[\theta,\tau]$ at a fixed point $p$, 
which is locally given by $t\curvearrowright(x,y)=(t^\theta x,t^\tau y)$,
the $\mathbb{C}^*$-action lifts to the blow-up at $p$. And,
\begin{itemize}
\item if $\theta=\tau$, the exceptional curve $E$ is a fixed curve of the lifted action, points in which all have weights $[\theta,0]$;
\item if $\theta\neq\tau$, the exceptional curve $E$ admits two fixed points $p_1,p_2$, with weights $[\theta_1,\tau_1],[\theta_2,\tau_2]$. Weights of them are $[\theta,\tau-\theta]$, $[\theta-\tau,\tau]$. 
\end{itemize}
\end{proposition}
\begin{proof}



Direct calculation.
\end{proof}
\begin{remark}\label{rem:afterblowup}
The above proposition can be explained in the following intuitive way. Take $\theta>\tau>0$ for an example. Arrows in the following picture shows the direction of the flows of the $\mathbb{C}^*$-action. Then a blow-up at the fixed point changes weights by
$$
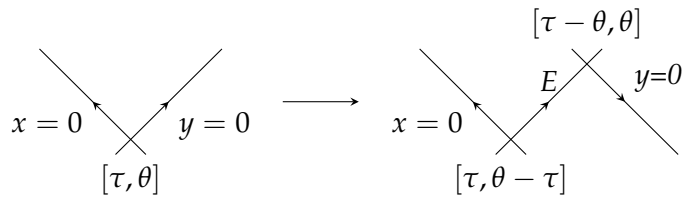
\begin{tikzpicture}
[decoration={markings, 
    mark= at position 0.5 with {\arrow{stealth}}}
]

\draw[postaction={decorate}] (0.2-5,-0.2)--node[below left]{$x=0$}(-1.2-5,1.2);
\draw[postaction={decorate}] (-0.2-5,-0.2) --node[below right]{$y=0$}(1.2-5,1.2);
\node[below=0.2cm] at (-5,0) {$[\tau,\theta]$};

\draw[-stealth] (-3,0.5)--(-2,0.5);

\draw[postaction={decorate}] (0.2,-0.2)--node[below left]{$x=0$}(-1.2,1.2);
\draw[postaction={decorate}] (-0.2,-0.2) --node[above]{$E$}(1.2,1.2);
\draw[postaction={decorate}]
(0.8,1.2)--node[above right]{y=0}(2.2,-0.2);
\node[below=0.2cm] at (0,0) {$[\tau,\theta-\tau]$};
\node[above=0.2cm] at (1,1){$[\tau-\theta,\theta]$};
\end{tikzpicture}
$$
Here $E$ is the exceptional curve. Other situations are similar.
\end{remark}

\subsection{Toric geometry of $\mathbb{C}^2/L(q,p)$ and weights of actions}\label{sec:Lqp}

In this subsection, we consider cyclic groups $\Gamma=L(q,p)$. The minimal resolution of $\mathbb{C}^2/\Gamma$ has been described in Section \ref{sec:minimalresolutions}. Using the languague of toric geometry (see \cite{fulton} for example), the minimal resolution of $\mathbb{C}^2/\Gamma$ can be described more precisely.
Recall that for $q/p$ we have the Hirzebruch-Jung continued fraction
$$\frac{q}{p}=\frac{1}{e_1-\displaystyle\frac{1}{e_2-\ldots-\frac{1}{e_k}}}.$$
\begin{theorem}\label{toricresolution}
Define the vertices $v_0,\ldots,v_{k+1}$ with $v_0=(0,1)$, $v_1=(1,0)$, $v_{i+1}=e_iv_i-v_{i-1}$. Notice that  $v_{k+1}=(p,-q)$. Then the fan $F_{L(q,p)}$ of the minimal resolution $\widetilde{\mathbb{C}^2/\Gamma}$ is given by Figure \ref{fanminimal}.
The exceptional set is as follows. The curve $E_i$ corresponds to the ray given by $v_i$ in the fan $F_{L(q,p)}$.
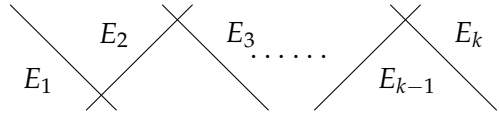
$$
\begin{tikzpicture}
\draw(0.2,-0.2)--node[below left]{$E_1$}(-1.2,1.2);
\draw(-0.2,-0.2)--node[above left]{$E_2$}(1.2,1.2);
\draw(0.8,1.2)--node[above right]{$E_3$}(2.2,-0.2);
\node at (2.5,0.5){$\cdots\cdots$};
\draw(2.8,-0.2)--node[below right]{$E_{k-1}$}(4.2,1.2);
\draw(3.8,1.2)--node[above right]{$E_{k}$}(5.2,-0.2);
\end{tikzpicture}
$$
\begin{figure}[h]
\centering
\begin{tikzpicture}
\path[draw](0,0)--(0,1)node[anchor=west]{$v_0$};
\path[draw](0,0)--(4,-3) node[anchor=north west]{$(p,-q)$};
\path[draw](0,0)--(1,0)node[anchor=west]{$v_1$};
\path[draw](0,0)--(2,-1)node[anchor=west]{$v_2$};
\path[draw](0,0)--(3,-2)node[anchor=west]{$v_3$};
\node at (0,-3) {$\circ$};
\node at (1,-3) {$\circ$};
\node at (2,-3) {$\circ$};
\node at (3,-3) {$\circ$};
\node at (4,-3) {$\circ$};
\node at (0,-2) {$\circ$};
\node at (1,-2) {$\circ$};
\node at (2,-2) {$\circ$};
\node at (3,-2) {$\circ$};
\node at (4,-2) {$\circ$};
\node at (0,-1) {$\circ$};
\node at (1,-1) {$\circ$};
\node at (2,-1) {$\circ$};
\node at (3,-1) {$\circ$};
\node at (4,-1) {$\circ$};
\node at (0,0) {$\circ$};
\node at (1,0) {$\circ$};
\node at (2,0) {$\circ$};
\node at (3,0) {$\circ$};
\node at (4,0) {$\circ$};
\node at (0,1) {$\circ$};
\node at (1,1) {$\circ$};
\node at (2,1) {$\circ$};
\node at (3,1) {$\circ$};
\node at (4,1) {$\circ$};
\end{tikzpicture}
\caption{The fan $F_{L(q,p)}$ of the minimal resolution $\widetilde{\mathbb{C}^2/\Gamma}$.}
\label{fanminimal}
\end{figure}
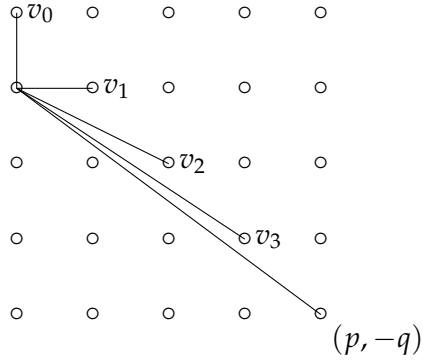
\end{theorem}
See Section 2.6 of \cite{fulton} for a proof of the above theorem.
The holomorphic $\mathbb{C}^*$-action $t\curvearrowright(x,y)=(t^\theta x,t^\tau y)$ on $\mathbb{C}^2/\Gamma$ can be lifted to the minimal resolution of $\mathbb{C}^2/\Gamma$, whose weights can be described by the following theorem.

\begin{theorem}\label{thm:cyclicweights}
Consider the $\mathbb{C}^*$-action on $\mathbb{C}^2/L(q,p)$ defined by $t\curvearrowright(x,y)=(t^\theta x,t^\tau y)$ with $\theta\geq\tau\geq0$ and its minimal resolution $\widetilde{\mathbb{C}^2/\Gamma}$ given by Theorem \ref{toricresolution}.
Given vertices $v_0,\ldots,v_k$ as in Theorem \ref{toricresolution}, write $v_i=(v_{i,U},v_{i,V})$ and for $i=0,\ldots,k+1$ define
$$w_i=\theta (pv_{i,V}+qv_{i,U})-\tau v_{i,U}.$$
Then the weights of the lifted action are:
$$
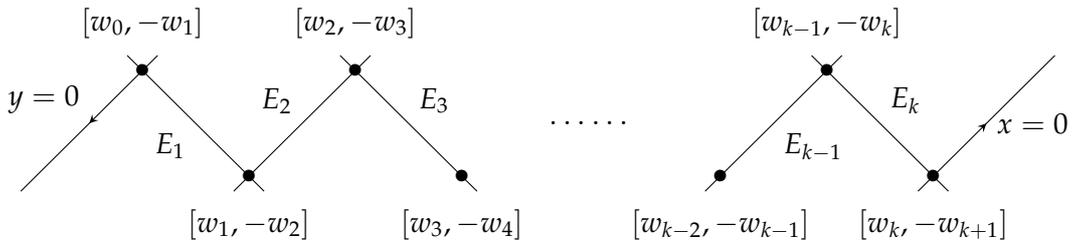
\begin{tikzpicture}
[decoration={markings, 
    mark= at position 0.5 with {\arrow{stealth}}}
]
\draw(-1.6,1.6)--node[below left]{$E_1$}(0.2,-0.2);
\draw(-0.2,-0.2)--node[above left]{$E_2$}(1.6,1.6);
\draw(1.2,1.6)--node[above right]{$E_3$}(3,-0.2);

\node at (4.5,0.7){$\cdots\cdots$};

\draw(6,-0.2)--node[below right]{$E_{k-1}$}(7.8,1.6);
\draw(7.4,1.6)--node[above right]{$E_{k}$}(9.2,-0.2);

\draw[postaction={decorate}](-1.2,1.6)--node[above left]{$y=0$}(-3,-0.2);
\draw[postaction={decorate}](9-0.2,-0.2)--node[right]{$x=0$}(10.6,1.6);

\node at (-1.4,1.4){$\bullet$};
\node at (0,0){$\bullet$};
\node at (1.4,1.4){$\bullet$};
\node at (2.8,0){$\bullet$};
\node at (6.2,0){$\bullet$};
\node at (7.6,1.4){$\bullet$};
\node at (9,0){$\bullet$};

\node[above=0.3cm] at (-1.4,1.4){$[w_0,-w_1]$};
\node[below=0.3cm] at (0,0){$[w_1,-w_2]$};
\node[above=0.3cm] at (1.4,1.4){$[w_2,-w_3]$};
\node[below=0.3cm] at (2.8,0){$[w_3,-w_4]$};
\node[below=0.3cm] at (6.2,0){$[w_{k-2},-w_{k-1}]$};
\node[above=0.3cm] at (7.6,1.4){$[w_{k-1},-w_{k}]$};
\node[below=0.3cm] at (9,0){$[w_k,-w_{k+1}]$};
\end{tikzpicture}
$$
Notice that there is the inductive relation $w_{i+1}=e_iw_i-w_{i-1}$, and we have $w_0=\theta p$, $w_1=\theta q-\tau$, $w_{k}=\theta-\tau v_{k,U}$, and $w_{k+1}=-\tau p$. 
\end{theorem}
\begin{proof}
This should be a standard computation but the author cannot find suitable references in previous literatures. We will provide a proof in the appendix.
\end{proof}

Next several corollaries follow from Theorem \ref{thm:cyclicweights}.
\begin{corollary}\label{cor:alpha1}
Consider the $\mathbb{C}^*$-action on $\mathbb{C}^2/L(q,p)$ defined by $t\curvearrowright(x,y)=(t^\theta x, y)$ with $\theta>0$. Then the weights of the lifted action and direction of flows on the exceptional set are:
$$
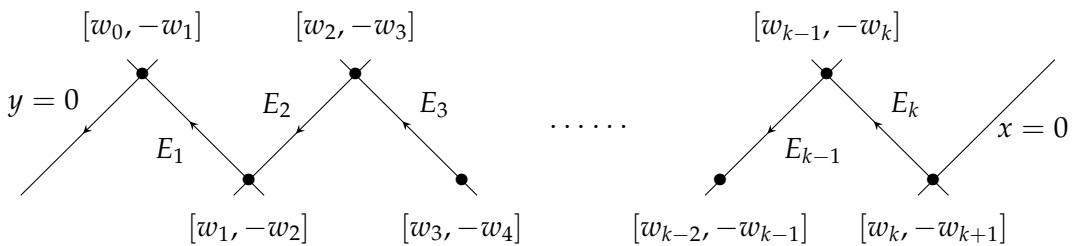
\begin{tikzpicture}
[decoration={markings, 
    mark= at position 0.5 with {\arrow{stealth reversed}}}
]
\draw[postaction={decorate}](-1.6,1.6)--node[below left]{$E_1$}(0.2,-0.2);
\draw[postaction={decorate}](-0.2,-0.2)--node[above left]{$E_2$}(1.6,1.6);
\draw[postaction={decorate}](1.2,1.6)--node[above right]{$E_3$}(3,-0.2);

\node at (4.5,0.7){$\cdots\cdots$};

\draw[postaction={decorate}](6,-0.2)--node[below right]{$E_{k-1}$}(7.8,1.6);
\draw[postaction={decorate}](7.4,1.6)--node[above right]{$E_{k}$}(9.2,-0.2);

\draw[postaction={decorate}](-3,-0.2)--node[above left]{$y=0$}(-1.2,1.6);
\draw(10.6,1.6)--node[right]{$x=0$}(9-0.2,-0.2);

\node at (-1.4,1.4){$\bullet$};
\node at (0,0){$\bullet$};
\node at (1.4,1.4){$\bullet$};
\node at (2.8,0){$\bullet$};
\node at (6.2,0){$\bullet$};
\node at (7.6,1.4){$\bullet$};
\node at (9,0){$\bullet$};

\node[above=0.3cm] at (-1.4,1.4){$[w_0,-w_1]$};
\node[below=0.3cm] at (0,0){$[w_1,-w_2]$};
\node[above=0.3cm] at (1.4,1.4){$[w_2,-w_3]$};
\node[below=0.3cm] at (2.8,0){$[w_3,-w_4]$};
\node[below=0.3cm] at (6.2,0){$[w_{k-2},-w_{k-1}]$};
\node[above=0.3cm] at (7.6,1.4){$[w_{k-1},-w_{k}]$};
\node[below=0.3cm] at (9,0){$[w_k,-w_{k+1}]$};
\end{tikzpicture}
$$
In this case, $x=0$ is a fixed curve.
\end{corollary}
\begin{proof}
$\tau=0$ implies $w_i=\theta (q,p)\cdot (v_{i,U},v_{i,V})>0$, and the corollary follows.
\end{proof}

\begin{corollary}\label{cor:weightsdeterminepq}
Consider the $\mathbb{C}^*$-action on $\mathbb{C}^2/L(q,p)$ defined by $t\curvearrowright(x,y)=(t^\theta x, t^\tau y)$ with $\theta\geq\tau\geq0$. Then,  when $\theta\neq0$ or $\tau\neq0$, the weights $[w_0,-w_1]$ or the weights
$[w_k,-w_{k+1}]$, together with $[\theta,\tau]$, uniquely determine the numbers $p$ and $q$, respectively.
\end{corollary}
\begin{proof}
Recall
$w_0=\theta p,\ w_1=\theta q-\tau,\ w_k=\theta-\tau v_{k,U},\ w_{k+1}=-\tau p$.
It is clear that once we know $[w_0,-w_1]$ when $\theta\neq0$, we can find $q$ and $p$. As for $[w_k,-w_{k+1}]$, notice that
$$\left\langle\left(\begin{matrix}
\exp(\frac{2\pi i}{p})&0\\
0&\exp(\frac{2\pi iq}{p})
\end{matrix}\right)\right\rangle
=
\left\langle\left(\begin{matrix}
\exp(\frac{2\pi iq^{-1:p}}{p})&0\\
0&\exp(\frac{2\pi i}{p})
\end{matrix}\right)\right\rangle.
$$
Replacing $x$ by $y$ and $y$ by $x$ and treating $L(q,p)$ as $L(q^{-1:p},p)$, we learn that $w_k=-\tau q^{-1:p}+\theta$. Thus, when $\tau\neq0$, $[w_k,w_{k+1}]$ determines $q^{-1:p}$ and $p$, which determines $q$ and $p$.
\end{proof}

\begin{corollary}\label{cor:su2weights}
In the case that $\Gamma\subset SU(2)$,
$w_{i+1}-w_i=w_i-w_{i-1}$.
In particular, $w_i$ is decreasing, and
we have
$-\tau p=w_{k+1}<\ldots<w_0=\theta p$.
\end{corollary}
\begin{proof}
Follows from $e_i=2$ trivially.
\end{proof}




\subsection{$\mathbb{C}^2/\Gamma$ for non-cyclic $\Gamma$ and weights of actions}\label{sec:generalgamma}

When $\Gamma$ is non-cyclic, there are elements in $\Gamma$ swapping the two standard complex planes in $\mathbb{C}^2$.
Therefore, any holomorphic $\mathbb{C}^*$-action on $\mathbb{C}^2/\Gamma$ must have same weights at the orbifold point, and is given by $t\curvearrowright(x,y)=(t^\theta x,t^\theta y)$. Taking $\theta=\frac{1}{2m}$, we can ensure the holomorphic $\mathbb{C}^*$-action on $\mathbb{C}^2/\Gamma$ is primitive.
From the proof of Theorem 4.1 in \cite{lv}, we have the following lemma:

\begin{lemma}\label{lemma:firstresolutionnoncyclic}
For a non-cyclic group $\Gamma$ in Table \ref{noncyclicL}, set $\Gamma'\subset\Gamma$ to be the subgroup of $\Gamma$ generated by the same generators listed in Table \ref{generators} but without $[e^{\pi i/m},1]$. Then $\mathcal{O}(-2m)/\Gamma'$ has three orbifold points lying on the quotient of the zero-section $\mathbb{P}^1\subset\mathcal{O}(-2m)$. The structure groups of these three orbifold points are precisely the $L(\alpha_j,\beta_j)$ listed in Table \ref{noncyclicL}.
\end{lemma} 
Contracting the $\mathbb{P}^1\subset\mathcal{O}(-2m)$ to the origin gives the minimal resolution $\mathcal{O}(-2m)\to\mathbb{C}^2/L(1,2m)$. Therefore, the map that contracts $\mathbb{P}^1/\Gamma'$ in $\mathcal{O}(-2m)/\Gamma'$ to the origin provides the map
$\kappa_{\Gamma}:\mathcal{O}(-2m)/\Gamma'\to\mathbb{C}^2/\Gamma$,
with $\mathbb{P}^1/\Gamma'$ serving as the exceptional set.
The minimal resolution $\widetilde{\mathbb{C}^2/\Gamma}$ is then just the minimal resolution of $\mathcal{O}(-2m)/\Gamma'$.
Therefore, above lemma implies that we only need to determine the weights of the lifted $\mathbb{C}^*$-action on $\mathcal{O}(-2m)/\Gamma'$ at the three orbifold points and then apply Theorem \ref{thm:cyclicweights}, to study the weights on the minimal resolution $\widetilde{\mathbb{C}^2/\Gamma}$.
\begin{proposition}\label{prop:orbifoldweights}
With the $\mathbb{C}^*$-action on $\mathbb{C}^2/\Gamma$ given by $t\curvearrowright(x,y)=(t^{1/2m}x,t^{1/2m}y)$, where $m$ is the number appearing in Table \ref{generators},
weights at the three orbifold points in $\mathcal{O}(-2m)/\Gamma'$ are all $[0,1]$.
\end{proposition}
\begin{proof}
By passing to the covering space, we can reduce to the case where $\Gamma'={1}$. In other words, we only need to compute the weights of points in the exceptional curve of the map
$\mathcal{O}(-2m)\to\mathbb{C}^2/L(1,2m)$.
The multi-valued map from $\mathcal{O}(-2m)$ to the orbifold cover $\mathbb{C}^2$ of $\mathbb{C}^2/L(1,2m)$ can be expressed as
$$[x:y]\times s\mapsto(xs^{1/2m},ys^{1/2m}).$$
The action with weights $[1/2m,1/2m]$ lifted to the orbifold cover $\mathbb{C}^2$ can be expressed as
$$t\curvearrowright(xs^{1/2m},ys^{1/2m})=(xs^{1/2m}t^{1/2m},ys^{1/2m}t^{1/2m}).$$
After lifting to $\mathcal{O}(-2m)$, the action is given by
$t\curvearrowright[x:y]\times s=[x:y]\times st$.
It is clear that the exceptional curve is a fixed curve and points in the exceptional curve all have weights $[0,1]$.
\end{proof}

An application of Theorem \ref{thm:cyclicweights} gives:
\begin{theorem}\label{thm:noncyclicweights}
For $\mathbb{C}^2/\Gamma$ with $\Gamma$ being one of the non-cyclic groups in Table \ref{noncyclicL}, consider the $\mathbb{C}^*$-action defined by $t\curvearrowright(x,y)=(t^{1/2m}x,t^{1/2m}y)$.
Its minimal resolution consists of three chains of rational curves, which are chains of exceptional curves
of the minimal resolutions of $L(\alpha_k,\beta_k)$,
connected by one central rational curve. The central rational curve is a fixed curve under the lifted action. Weights of the lifted action on these three chains are given by Corollary \ref{cor:alpha1} with $\theta=1$.
\end{theorem}
\begin{proof}
Proposition \ref{prop:orbifoldweights} shows that the action at the orbifold points is locally given by $t\curvearrowright(x,y)=(tx,y)$ with a suitable coordinate. At the same time, the central curve in $\widetilde{\mathbb{C}^2/\Gamma}$ is the proper transform of $\mathbb{P}^1/\Gamma'$ in $\mathcal{O}(-2m)/\Gamma'$, which is a fixed curve under the lifted action. 
The theorem follows from Corollary \ref{cor:alpha1} now.
\end{proof}

\section{Log del Pezzo surfaces with holomorphic vector fields}
\label{sec:case12}

In Section \ref{sec:case12} and \ref{sec:case3}, we classify the pairs $(\widehat{M},\mathfrak{E})$ where
\begin{itemize}
\item The complex surface $\widehat{M}$ is a log del Pezzo surface with only one $SU(2)$ orbifold point $q$.
\item The vector field $\mathfrak{E}$ is a holomorphic vector field on $\widehat{M}$, which generates a primitive holomorphic $\mathbb{C}^*$-action that flows out of $q$.
\item The holomorphic $\mathbb{C}^*$-action has same weights at $q$.
\end{itemize}
As mentioned at the beginning of Section \ref{sec:holfields}, these pairs $(\widehat{M},\mathfrak{E})$ are candidates for special Bach-flat K\"ahler orbifolds $(\widehat{M},\widehat{g})$ with structure group in $SU(2)$.

Denote the minimal resolution of $\widehat{M}$ by $r:\widetilde{M}\to\widehat{M}$.
To classify such pairs $(\widehat{M},\mathfrak{E})$, one only needs to classify their minimal resolution $\widetilde{M}$ with the lifted $\mathbb{C}^*$-action. So it suffices to classify $(\widetilde{M},\mathfrak{E})$. Because $\widetilde{M}$ is smooth rational, it must have $\mathbb{P}^2$ or $H_k$ with $k\geq2$ as its minimal model. The following is standard because there is no curve in $\widetilde{M}$ with self-intersection less than $-2$.

\begin{lemma}
The minimal resolution $\widetilde{M}$ either is an iterative blow-up of $\mathbb{P}^2$, or is exactly $H_2=\mathbb{P}(\mathcal{O}\oplus\mathcal{O}(2))$.
\end{lemma}

\begin{proposition}\label{prop:added}
If $\widetilde{M}=H_2=\mathbb{P}(\mathcal{O}\oplus\mathcal{O}(2))$, then the correpsonding $\widehat{M}$ is $H_2$ contracting $C_\infty$. The complex surface $M=\widehat{M}\setminus\{q\}$ is $\mathcal{O}(2)$.
\end{proposition}
\begin{proof}
This is clear since the $(-2)$-curve $C_\infty$ is the only curve with negative self-intersection in $\widetilde{M}$, therefore $M=\widehat{M}\setminus\{q\}=H_2\setminus C_\infty$.  
\end{proof}
It is direct to check that the Hitchin-Thorpe inequality \cite{nakajima} holds with an equality on $\mathcal{O}(2)$. Consequently, any ALE gravitational instanton with $\mathcal{O}(2)$ as the underlying space, under the correct orientation, must be anti-self-dual, hence hyperk\"ahler, under this orientation. This exactly correpsonds to the Eguchi-Hanson space.

So in the following, we only need to consider iterative blow-ups of $\mathbb{P}^2$. We only need to consider pairs $(\widetilde{M},\mathfrak{E})$ that are iterative blow-ups of some pairs $(\mathbb{P}^2,\mathfrak{F})$, where $\mathfrak{F}$ is a holomorphic vector field on $\mathbb{P}^2$ that generates a primitive holomorphic $\mathbb{C}^*$-action. When we blow up $\mathbb{P}^2$ to get $\widetilde{M}$, blow-ups can only be performed at the fixed points of $\mathfrak{F}$.
As a result, classifying these pairs $(\widehat{M},\mathfrak{E})$  becomes a case-by-case study. We introduce some notations first.

\begin{convention}
\ 

\begin{itemize}
\item By curves, we particularly mean rational curves that are the closure of $\mathbb{C}^*$-orbits.
\item We will describe the structure of each surface by picture. In the following pictures, each black line refers to a curve. Arrow on each line represents the direction of the flow on the curve. Thin black lines with two arrows refer to $(-1)$-curves. Thick black lines refer to $(-2)$-curves. When there is a black box $\blacksquare$ on a line, it means the curve is a fixed curve.
\item The blow-up map from $\widetilde{M}$ to $\mathbb{P}^2$ is denoted by $\pi$. The map $\pi$ is realized by  successive blow-ups of $\mathbb{P}^2$, and the intermediate surfaces are denoted by $M_i$:
$$\widetilde{M}=M_n\xrightarrow{\pi_n}M_{n-1}\xrightarrow{\pi_{n-1}}\ldots\xrightarrow{\pi_2}M_1\xrightarrow{\pi_1}M_0=\mathbb{P}^2.$$
Each $\pi_i$ is a blow-up at a point $p_i\in M_{i-1}$. We use $Bl_{p_1,\ldots,p_k}\mathbb{P}^2$ to denote the blow-up of $\mathbb{P}^2$, where the $i$-th blow-up $\pi_i$ is taken at $p_i\in Bl_{p_1,\ldots,p_{i-1}}\mathbb{P}^2$. The exceptional curve of the $i$-th blow-up $\pi_i$ is denoted by $E_i$. 
\item We use $[x:y:z]$ as coordinate on $\mathbb{P}^2$.
The curves $\{x=0\},\{y=0\},\{z=0\}$ in $\mathbb{P}^2$ and their proper transforms in each $M_i$ are denoted by $X,Y,Z$.
\item Denote the exceptional set of the minimal resolution $r:\widetilde{M}\to\widehat{M}$ by $E$, which is the cycle of all $(-2)$-curves in $\widetilde{M}$ and is of type $A_n,D_n$, or $E_n$. 
\item We will use $E_i$ to denote the proper transform of the curve $E_i$ as well for simplicity.
\end{itemize}
\end{convention}

To find all possible $(\widetilde{M},\mathfrak{E})$, our strategy is to proceed the following inductive steps.
\begin{enumerate}[label=\textbf{Step   \arabic*.}]
\item Choose a pair $(M_0,\mathfrak{F}_0)=(\mathbb{P}^2,\mathfrak{F})$. Equivalently, we choose a primitive holomorphic $\mathbb{C}^*$-action on $\mathbb{P}^2$, with the infinitesimal generator $\mathfrak{F}$.
\item Assuming we have determined $(M_{i-1},\mathfrak{F}_{i-1})$, we then identify all the fixed points of $\mathfrak{F}_{i-1}$ in $M_{i-1}$ that are not contained in any $(-2)$-curve.
\item Possible $(M_i,\mathfrak{F}_i)$ then are the blow-up $\pi_i$ of $(M_{i-1},\mathfrak{F}_{i-1})$ at one of the fixed points $p_i$ identified in \textbf{Step 2}. 
 It must be ensured that finally in $\widetilde{M}=M_n$ the exceptional set $E$ is of $ADE$ type, and $\mathfrak{E}=\mathfrak{F}_n$ has correct weights so that the $\mathfrak{E}$ action on the contraction $\widehat{M}$ has same weights at $q$. In particular, Corollary \ref{cor:blowup} and Lemma \ref{lem:blowup} in the following Section \ref{subsec:flows} hold.
\end{enumerate}

The $\mathbb{C}^*$-action on $\widehat{M}$ given by $\mathfrak{E}$ has same weights at the orbifold point $q$ in the cyclic case, if and only if the action on $E\subset \widetilde{M}$
has the weights given by Corollary \ref{cor:su2weights} with $\theta=\tau$. In the non-cyclic case, the $\mathbb{C}^*$-action on $\widehat{M}$ always has same weights at the orbifold point.  We make the following definition before we proceed.

\begin{definition}
The \textit{attractive set} $c_+$ is defined as the set where generic $\mathbb{C}^*$-orbits flow into. The \textit{repulsive set} $c_-$ is defined as the set where generic $\mathbb{C}^*$-orbits flow out of. 
\end{definition}
It is clear that $c_{\pm}$ of each $(M_i,\mathfrak{F}_i)$ can only be a fixed point or a fixed curve, since they arise from $(\mathbb{P}^2,\mathfrak{F})$.

\subsection{$\mathbb{C}^*$-actions on $\mathbb{P}^2$}\label{sec:actiononp2}

For a pair $(\mathbb{P}^2,\mathfrak{F})$, the primitive holomorphic $\mathbb{C}^*$-action generated by $\mathfrak{F}$ and its weights must take one of the following three forms.

\begin{enumerate}
\item The $\mathbb{C}^*$-action is given by $t\curvearrowright[x:y:z]=[t x:t y:z]$. In this case, the fixed points set of $\mathfrak{F}$ is the curve $Z$ and the point $[0:0:1]$.
$$
\label{tikz:case1}
\begin{tikzpicture}
[decoration={markings, 
    mark= at position 0.5 with {\arrow{stealth }}}
]
\draw[postaction={decorate}](1.2,1.732+1.732*0.2)--node[above left]{$X$}(-0.2,-0.2*1.732);
\draw[postaction={decorate}](0.8,1.732+1.732*0.2)--node[above right]{$Y$}(2.2,-1.732*0.2);

\draw(-0.2*1.732,0)--node[below]{$Z$}node[
    sloped,
    pos=0.5,
    allow upside down]{\arrowBox}(2+0.2*1.732,0);

\node[above=0.4cm] at (1,1.732){$[1,1]$};
\node[below left=0.3cm] at (0,0){$[-1,0]$};
\node[below right=0.3cm] at (2,0){$[0,-1]$};
\end{tikzpicture}.
$$
The attractive set $c_+$ is $Z$ and the repulsive set $c_-$ is the point $[0:0:1]$.

\item The $\mathbb{C}^*$-action is given by $t\curvearrowright[x:y:z]=[t^{-1} x:t^{-1} y:z]$. In this case, the fixed points set of $\mathfrak{F}$ is the curve $Z$ and the point $[0:0:1]$.
$$
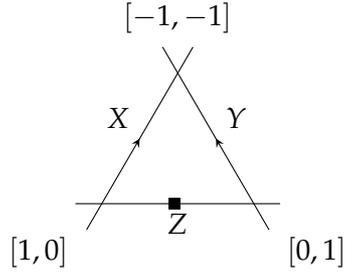
\begin{tikzpicture}
[decoration={markings, 
    mark= at position 0.5 with {\arrow{stealth }}}
]
\draw[postaction={decorate}](-0.2,-0.2*1.732)--node[above left]{$X$}(1.2,1.732+1.732*0.2);
\draw[postaction={decorate}](2.2,-1.732*0.2)--node[above right]{$Y$}(0.8,1.732+1.732*0.2);

\draw(-0.2*1.732,0)--node[below]{$Z$}node[
    sloped,
    pos=0.5,
    allow upside down]{\arrowBox}(2+0.2*1.732,0);

\node[above=0.4cm] at (1,1.732){$[-1,-1]$};
\node[below left=0.3cm] at (0,0){$[1,0]$};
\node[below right=0.3cm] at (2,0){$[0,1]$};
\end{tikzpicture}.
$$
The attractive set $c_+$ is the point $[0:0:1]$ and the repulsive set $c_-$ is $Z$.

\item The $\mathbb{C}^*$-action is given by $t\curvearrowright[x:y:z]=[t^\alpha x:t^\beta y:z]$, with coprime $\alpha>\beta>0$. In this case, the fixed points set of $\mathfrak{F}$ consists of the points $[1:0:0],[0:1:0]$, and $[0:0:1]$.
$$
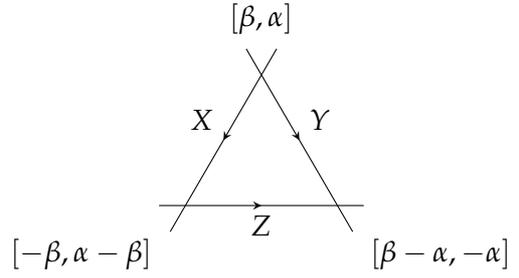
\begin{tikzpicture}
[decoration={markings, 
    mark= at position 0.5 with {\arrow{stealth }}}
]
\draw[postaction={decorate}](1.2,1.732+1.732*0.2)--node[above left]{$X$}(-0.2,-0.2*1.732);
\draw[postaction={decorate}](0.8,1.732+1.732*0.2)--node[above right]{$Y$}(2.2,-1.732*0.2);
\draw[postaction={decorate}](-0.2*1.732,0)--node[below]{$Z$}(2+0.2*1.732,0);

\node[above=0.4cm] at (1,1.732){$[\beta,\alpha]$};
\node[below left=0.3cm] at (0,0){$[-\beta,\alpha-\beta]$};
\node[below right=0.3cm] at (2,0){$[\beta-\alpha,-\alpha]$};
\end{tikzpicture}.
$$
The attractive set $c_+=[1:0:0]$ and the repulsive set $c_-=[0:0:1]$. 
\end{enumerate}

We therefore divide into three possibilities:
\begin{enumerate}[label=(\roman*)]
\item The $\mathbb{C}^*$-action on $\widetilde{M}$ is lifted from $t\curvearrowright[x:y:z]=[t x:t y:z]$ on $\mathbb{P}^2$.\label{case1}
\item The $\mathbb{C}^*$-action on $\widetilde{M}$ is lifted from $t\curvearrowright[x:y:z]=[t^{-1} x:t^{-1} y:z]$ on $\mathbb{P}^2$;\label{case2}
\item The $\mathbb{C}^*$-action on $\widetilde{M}$ is lifted from
$t\curvearrowright[x:y:z]=[t^\alpha x:t^\beta y:z]$ on $\mathbb{P}^2$ with coprime $\alpha>\beta>0$.\label{case3}
\end{enumerate} 
We will refer to these three cases as case (i), case (ii), and case (iii) in the following discussion.
Before we proceed, we prove Corollary \ref{cor:topofinite}, the topological finiteness of Hermitian non-K\"ahler ALE gravitational instantons.  For this proposition we do not require that $\Gamma\subset SU(2)$.

\begin{proposition}
For a log del Pezzo surface $\widehat{M}$ with only one $U(2)$ orbifold singularity, the degree of its minimal resolution $K_{\widetilde{M}}^2$ has a lower bound $k_\Gamma$ that only depends on $\Gamma$. 
\end{proposition}

\begin{proof}
As $\widetilde{M}$ is the minimal resolution of $\widehat{M}$, one has
$-K_{\widetilde{M}}=-r^*K_{\widehat{M}}-\sum a_iC_i$.
Here $C_i$ are the exceptional curves.  The discrepancies $a_i$ and the cycle of exceptional curves are determined by the group $\Gamma$.  
Therefore, the degree of $\widetilde{M}$ satisfies $K_{\widetilde{M}}^2=(r^*K_{\widehat{M}}+\sum a_iC_i)^2=K_{\widehat{M}}^2+(\sum a_iC_i)^2\geq(\sum a_iC_i)^2$. The number $k_\Gamma:= (\sum a_iC_i)^2$ depends only on $\Gamma\subset U(2)$, as claimed.
\end{proof}

The lower bound on $K_{\widetilde{M}}^2$ implies that if $\widetilde{M}$ has $\mathbb{P}^2$ as its minimal model, then at most $9-k_\Gamma$ blow-ups can be performed. If $\widetilde{M}$ has a Hizebruch surface as a minimal model, then at most  $8-k_\Gamma$ blow-ups can be performed.  This proves Corollary \ref{cor:topofinite}.

\subsection{Flows on $\widetilde{M}$}
\label{subsec:flows}

In this subsection, we establish some technical lemmas regarding the structure of $\widetilde{M}$. Because $\widetilde{M}$ cannot contain curves with intersection number $\leq -3$, the following lemma is straightforward:
\begin{lemma}\label{lem:blowup}
Fixed point $p\in M_i$ lying on a $(-2)$-curve cannot be blown up.
\end{lemma}

\begin{lemma}\label{lem:preblowup}
Let $p\in M_i$ be a transverse intersection of two curves $C_1,C_2\subset M_i$, whose proper transforms in $\widetilde{M}$ are both contained in $E$. Assume that weights $[\theta,\tau]$ at $p$ are not strictly positive or strictly negative, then $p$
cannot be blown up in the resolution $\pi:\widetilde{M}\to\mathbb{P}^2$.
\end{lemma}

\begin{proof}
Without loss of generality we can assume $\theta\geq0\geq\tau$. If there were a blow-up at $p\in M_i$, then there would not be any fixed curve in the preimage of $p$ in any $M_j$ with $j>i$ because of Proposition \ref{prop:blowupweights}. The preimage of $p$ then necessarily contains a $(-1)$-curve. As $C_1$ and $C_2$ have their proper transforms in $E$ and $E$ entirely consists of $(-2)$-curves, the appearance of the $(-1)$-curve will make $E$ disconnected, contradiction. 
\end{proof}

As a corollary, we have:
\begin{corollary}\label{cor:blowup}
Let $p\in M_i$ be a transverse intersection of two curves $C_1,C_2\subset M_i$, where each $C_i$ is either a curve whose proper transform is contained in $E$, or a curve with negative self-intersection.
Assume that weights at $p$ are not strictly positive or negative. Then $p$
cannot be blown up in the resolution $\pi:\widetilde{M}\to\mathbb{P}^2$.
\end{corollary}
\begin{proof}
If such $p_i$ is blown up, then proper transform of $C_i$ in $\widetilde{M}$ has self intersection $\leq-2$. This shows that proper transform of $C_i$ is contained in $E$ and the corollary follows from the above lemma.
\end{proof}

\subsection{Case \ref{case1}}\label{sec:case1}
In this subsection, we focus on case \ref{case1} and show that the pair $(\widetilde{M},\mathfrak{E})$ admits three possibilities.

\begin{proposition}\label{prop:case1}
In case \ref{case1}, there are three possible $(\widetilde{M},\mathfrak{E})$.
\end{proposition} 

\begin{proof}
Recall flows on $\mathbb{P}^2$ in case \ref{case1} is given in Section \ref{sec:actiononp2}.
Following our strategy, we analyze as follows. 

If there is no blow-up at $X\cap Y$, then the repulsive set $c_-$ is exactly $X\cap Y$ in $\widetilde{M}$. Since $c_-\subset E$ in $\widetilde{M}$, without loss of generality, we can assume that $X\subset E$ in $\widetilde{M}$. Then $X$ in $\widetilde{M}$ must have self-intersection $-2$. To achieve this,  one needs to perform blow-ups in $X$ inside $\mathbb{P}^2$. Therefore we take $p_1=X\cap Z$.  Flows and weights on $M_1$ are
$$
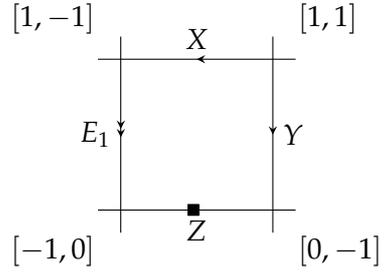
\begin{tikzpicture}
[decoration={markings, 
    mark= at position 0.5 with {\arrow{stealth }}}
]

\draw(0,2.3)--node[left]{$E_1$}node[sloped,pos=0.5,allow upside down]{\arrowIIn}(0,-0.3);

\draw[postaction={decorate}](2,2.3)--node[right]{$Y$}(2,-0.3);

\draw(-0.3,0)--node[below]{$Z$}node[
    sloped,
    pos=0.5,
    allow upside down]{\arrowBox}(2+0.3,0);

\draw[postaction={decorate}](2+0.3,2)--node[above]{$X$}(-0.3,2);

\node[above left=0.2cm] at (0,2){$[1,-1]$};
\node[below left=0.2cm] at (0,0){$[-1,0]$};
\node[below right=0.2cm] at (2,0){$[0,-1]$};
\node[above right=0.2cm] at (2,2){$[1,1]$};
\end{tikzpicture}.
$$
After the blow-up $\pi_1$ the curve $X$ in $M_1$ is a $0$-curve, and we still need to take blow-ups in $X$. However,
\begin{itemize}
\item $X\cap Y$ cannot be blown up because of our assumption.
\item $X\cap E_1$ cannot be blown up because of Corollary \ref{cor:blowup}.
\end{itemize}
Hence, the situation that there is no blow-up at $X\cap Y$ can not happen.

If there is a blow-up at $X\cap Y$, we can take $M_1$ as the blow-up at $p_1=X\cap Y$. Flows and weights on $M_1$ are 

$$
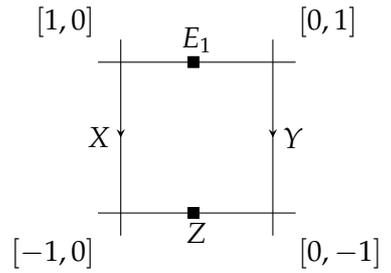
\begin{tikzpicture}
[decoration={markings, 
    mark= at position 0.5 with {\arrow{stealth }}}
]

\draw[postaction={decorate}](0,2.3)--node[left]{$X$}(0,-0.3);

\draw[postaction={decorate}](2,2.3)--node[right]{$Y$}(2,-0.3);

\draw(-0.3,0)--node[below]{$Z$}(2+0.3,0)node[
    sloped,
    pos=0.5,
    allow upside down]{\arrowBox};

\draw(-0.3,2)--node[above]{$E_1$}node[
    sloped,
    pos=0.5,
    allow upside down]{\arrowBox}(2+0.3,2);

\node[above left=0.2cm] at (0,2){$[1,0]$};
\node[below left=0.2cm] at (0,0){$[-1,0]$};
\node[below right=0.2cm] at (2,0){$[0,-1]$};
\node[above right=0.2cm] at (2,2){$[0,1]$};
\end{tikzpicture}.
$$
There is still no $(-2)$-curve in $M_1$. The repulsive set $c_-$ is already the fixed curve $E_1$. So the proper transform of $E_1$ in $\widetilde{M}$ must be a $(-2)$-curve. Since $E_1$ in $M_1$ has self-intersection $-1$, another blow-up is necessary in $E_1\subset M_1$. It can be assumed that the blow-up $\pi_2$ is at $p_2=X\cap E_1$. Flows and weights on $M_2$ are
$$
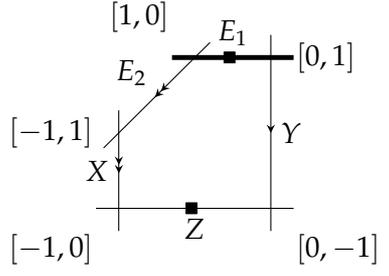
\begin{tikzpicture}
[decoration={markings, 
    mark= at position 0.5 with {\arrow{stealth }}}
]

\draw(0,1.3)--node[left]{$X$}(0,-0.3)node[sloped,pos=0.5,allow upside down]{\arrowIIn};

\draw[postaction={decorate}](2,2.3)--node[right]{$Y$}(2,-0.3);

\draw(-0.3,0)--node[below]{$Z$}(2+0.3,0)node[
    sloped,
    pos=0.5,
    allow upside down]{\arrowBox};

\draw[line width=0.65mm]
(0.7,2)--node[above]{$E_1$}(2+0.3,2)node[
    sloped,
    pos=0.5,
    allow upside down]{\arrowBox};

\draw(1+0.2,2.2)--node[above left]{$E_2$}(0-0.2,0.8)node[sloped,pos=0.5,allow upside down]{\arrowIIn};

\node[left=0.2cm] at (0,1){$[-1,1]$};
\node[below left=0.2cm] at (0,0){$[-1,0]$};
\node[above left=0.2cm] at (1,2){$[1,0]$};
\node[right=0.2cm] at (2,2){$[0,1]$};
\node[below right=0.2cm] at (2,0){$[0,-1]$};
\end{tikzpicture}.
$$
Now,
\begin{itemize}
\item $X\cap E_2$ cannot be blown up because of Corollary \ref{cor:blowup}.
\item $X\cap Z$ cannot be blown up because then $X$ will be a $(-2)$-curve, which makes $E$ disconnected in $\widetilde{M}$.
\end{itemize}
If there are further blow-ups in $M_2$, the only option is to blow up points in $Z$ other than the point $X\cap Z$.
The intersection number of $Z$ is $1$ in $M_2$.
If there are more blow-ups, then we can  suppose  $p_3=Z\cap Y$. Flows and weights on $M_3$ are
$$
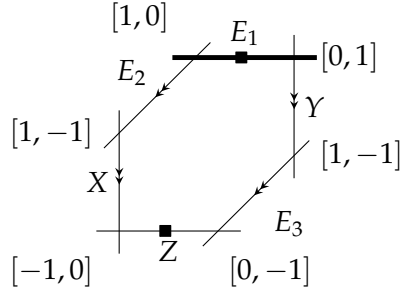
\begin{tikzpicture}
[decoration={markings, 
    mark= at position 0.5 with {\arrow{stealth }}}
]

\draw(0,1.3)--node[left]{$X$}(0,-0.3-0.3)node[sloped,pos=0.5,allow upside down]{\arrowIIn};

\draw(2.3,2.3)--node[right]{$Y$}(2.3,0.4)node[sloped,pos=0.5,allow upside down]{\arrowIIn};

\draw(-0.3,-0.3)--node[below]{$Z$}(1+0.6,-0.3)node[
    sloped,
    pos=0.5,
    allow upside down]{\arrowBox};

\draw[line width=0.65mm](0.7,2)--node[above]{$E_1$}node[
    sloped,
    pos=0.5,
    allow upside down]{\arrowBox}(2+0.3+0.3,2);

\draw(1+0.2,2.2)--node[above left]{$E_2$}(0-0.2,0.8)node[sloped,pos=0.5,allow upside down]{\arrowIIn};

\draw(2.5,0.9)--node[below right=0.1cm]{$E_3$}(1+0.1,-0.5)node[sloped,pos=0.5,allow upside down]{\arrowIIn};

\node[left=0.2cm] at (0,1){$[1,-1]$};
\node[below left=0.2cm] at (0,-0.3){$[-1,0]$};
\node[above left=0.2cm] at (1,2){$[1,0]$};
\node[right=0.2cm] at (2.3,2){$[0,1]$};
\node[below right=0.2cm] at (1.1,-0.3){$[0,-1]$};
\node[right=0.2cm] at (2.3,0.7){$[1,-1]$};

\end{tikzpicture}.
$$
Now,
\begin{itemize}
\item $X\cap E_2,E_1\cap E_2,Y\cap E_1,Y\cap E_3$ cannot be blown up because of Corollary \ref{cor:blowup}.
\item $X\cap Z,Z\cap E_3$ cannot be blown up because then $X$ or $E_3$ will be a $(-2)$-curve, which makes $E$ disconnected in $\widetilde{M}$.
\end{itemize}
In $M_3$, $Z$ is a $0$-curve. Any further blow-up in $M_3$ must be made at points in $Z$ other than $X\cap Z$ and $Z\cap E_3$. If there is a further blow-up at $p_4\in Z$, then flows and weights on $M_4$ are
$$
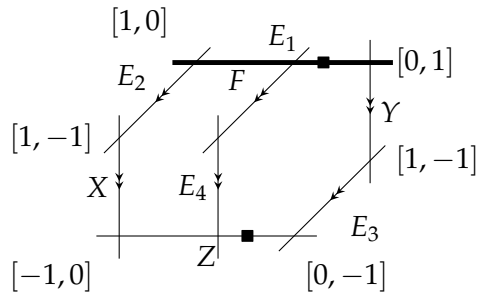
\begin{tikzpicture}
[decoration={markings, 
    mark= at position 0.5 with {\arrow{stealth }}}
]

\draw(0,1.3)--node[left]{$X$}(0,-0.3-0.3)node[sloped,pos=0.5,allow upside down]{\arrowIIn};

\draw(0+1.3,1.3)--node[left]{$E_4$}(0+1.3,-0.3-0.3)node[sloped,pos=0.5,allow upside down]{\arrowIIn};

\draw(2.3+1,2.3)--node[right]{$Y$}(2.3+1,0.4)node[sloped,pos=0.5,allow upside down]{\arrowIIn};

\draw(-0.3,-0.3)--node[below]{$Z$}(1+0.6+1,-0.3)node[
    sloped,
    pos=0.7,
    allow upside down]{\arrowBox};

\draw[line width=0.65mm](0.7,2)--node[above]{$E_1$}(2+0.3+0.3+1,2)node[
    sloped,
    pos=0.7,
    allow upside down]{\arrowBox};

\draw(1+0.2,2.2)--node[above left]{$E_2$}(0-0.2,0.8)node[sloped,pos=0.5,allow upside down]{\arrowIIn};

\draw(1+0.2+1.3,2.2)--node[above left]{$F$}(0-0.2+1.3,0.8)node[sloped,pos=0.5,allow upside down]{\arrowIIn};

\draw(2.5+1,0.9)--node[below right=0.1cm]{$E_3$}(1+0.1+1,-0.5)node[sloped,pos=0.5,allow upside down]{\arrowIIn};

\node[left=0.2cm] at (0,1){$[1,-1]$};
\node[below left=0.2cm] at (0,-0.3){$[-1,0]$};
\node[above left=0.2cm] at (1,2){$[1,0]$};
\node[right=0.2cm] at (2.3+1,2){$[0,1]$};
\node[below right=0.2cm] at (1.1+1,-0.3){$[0,-1]$};
\node[right=0.2cm] at (2.3+1,0.7){$[1,-1]$};

\end{tikzpicture}.
$$
Here $F$ is the proper transform of the curve in $M_3$, flowing from $E_1$ to $p_4$.
Any further blow-up in $M_4$ must be made in $Z$. However, since $Z$ is already a $(-1)$-curve, no more points in $Z$ can be blown up.  Therefore, we conclude that in case \ref{case1}, the only possible $(\widetilde{M},\mathfrak{E})$ are $M_2$, $M_3$, and $M_4$, with the primitive action with weights described above.
\end{proof}

Using Corollary \ref{cor:su2weights}, it is direct to check the action by $\mathfrak{E}$ on the corresponding $\widehat{M}$ for these three cases above all have same weights at the orbifold point.
Structures of these three possible $(\widetilde{M},\mathfrak{E})$ are listed as follows:

\begin{minipage}{0.5\textwidth}

\paragraph{\underline{\underline{Case 1A}}}
\label{1A}

\begin{itemize}
\item The orbifold group is $A_1$.
\item The degree of 
$$\widetilde{M}=Bl_{X\cap Y,E_1\cap X}\mathbb{P}^2$$ is 7 and the picard number of $\widehat{M}$ is 2.
\item $K_{\widetilde{M}}=-3Z+E_1+2E_2$. 
\end{itemize}

\end{minipage}%
\begin{minipage}{0.4\textwidth}
$$
\begin{tikzpicture}
[decoration={markings, 
    mark= at position 0.5 with {\arrow{stealth }}}
]

\draw(0,1.3)--node[left]{$X$}(0,-0.3)node[sloped,pos=0.5,allow upside down]{\arrowIIn};

\draw[postaction={decorate}](2,2.3)--node[right]{$Y$}(2,-0.3);

\draw(-0.3,0)--node[below]{$Z$}(2+0.3,0)node[
    sloped,
    pos=0.5,
    allow upside down]{\arrowBox};

\draw[line width=0.65mm]
(0.7,2)--node[above]{$E_1$}(2+0.3,2)node[
    sloped,
    pos=0.5,
    allow upside down]{\arrowBox};

\draw(1+0.2,2.2)--node[above left]{$E_2$}(0-0.2,0.8)node[sloped,pos=0.5,allow upside down]{\arrowIIn};

\node[left=0.2cm] at (0,1){$[-1,1]$};
\node[below left=0.2cm] at (0,0){$[-1,0]$};
\node[above left=0.2cm] at (1,2){$[1,0]$};
\node[right=0.2cm] at (2,2){$[0,1]$};
\node[below right=0.2cm] at (2,0){$[0,-1]$};

\end{tikzpicture}
$$
\end{minipage}

\begin{minipage}{0.5\textwidth}

\paragraph{\underline{\underline{Case 1B}}}
\label{1B}

\begin{itemize}
\item The orbifold group is $A_1$.
\item The degree of 
$$\widetilde{M}=Bl_{X\cap Y,E_1\cap X,Y\cap Z}\mathbb{P}^2$$ is 6 and the picard number of $\widehat{M}$ is 3.
\item $K_{\widetilde{M}}=-3Z+E_1+2E_2-2E_3$. 
\end{itemize}

\end{minipage}%
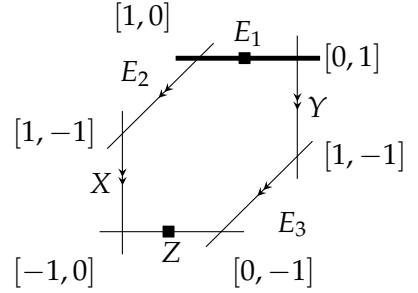
\begin{minipage}{0.4\textwidth}
$$
\begin{tikzpicture}
[decoration={markings, 
    mark= at position 0.5 with {\arrow{stealth }}}
]

\draw(0,1.3)--node[left]{$X$}(0,-0.3-0.3)node[sloped,pos=0.5,allow upside down]{\arrowIIn};

\draw(2.3,2.3)--node[right]{$Y$}(2.3,0.4)node[sloped,pos=0.5,allow upside down]{\arrowIIn};

\draw(-0.3,-0.3)--node[below]{$Z$}(1+0.6,-0.3)node[
    sloped,
    pos=0.5,
    allow upside down]{\arrowBox};

\draw[line width=0.65mm](0.7,2)--node[above]{$E_1$}(2+0.3+0.3,2)node[
    sloped,
    pos=0.5,
    allow upside down]{\arrowBox};

\draw(1+0.2,2.2)--node[above left]{$E_2$}(0-0.2,0.8)node[sloped,pos=0.5,allow upside down]{\arrowIIn};

\draw(2.5,0.9)--node[below right=0.1cm]{$E_3$}(1+0.1,-0.5)node[sloped,pos=0.5,allow upside down]{\arrowIIn};

\node[left=0.2cm] at (0,1){$[1,-1]$};
\node[below left=0.2cm] at (0,-0.3){$[-1,0]$};
\node[above left=0.2cm] at (1,2){$[1,0]$};
\node[right=0.2cm] at (2.3,2){$[0,1]$};
\node[below right=0.2cm] at (1.1,-0.3){$[0,-1]$};
\node[right=0.2cm] at (2.3,0.7){$[1,-1]$};

\end{tikzpicture}
$$
\end{minipage}

\begin{minipage}{0.5\textwidth}

\paragraph{\underline{\underline{Case 1C}}}
\label{1C}

\begin{itemize}
\item The orbifold group is $A_1$.
\item The degree of 
$$\widetilde{M}=Bl_{X\cap Y,E_1\cap X,Y\cap Z,p_4}\mathbb{P}^2$$ is 5 and the picard number of $\widehat{M}$ is 4. Here $p_4$ is a point in $Z$ different from $X\cap Z,E_3\cap Z$. 
\item $K_{\widetilde{M}}=-3Z+E_1+2E_2-2E_3-2E_4$.
\end{itemize}

\end{minipage}%
\begin{minipage}{0.4\textwidth}
$$
\begin{tikzpicture}
[decoration={markings, 
    mark= at position 0.5 with {\arrow{stealth }}}
]

\draw(0,1.3)--node[left]{$X$}(0,-0.3-0.3)node[sloped,pos=0.5,allow upside down]{\arrowIIn};

\draw(0+1.3,1.3)--node[left]{$E_4$}(0+1.3,-0.3-0.3)node[sloped,pos=0.5,allow upside down]{\arrowIIn};

\draw(2.3+1,2.3)--node[right]{$Y$}(2.3+1,0.4)node[sloped,pos=0.5,allow upside down]{\arrowIIn};

\draw(-0.3,-0.3)--node[below]{$Z$}(1+0.6+1,-0.3)node[
    sloped,
    pos=0.7,
    allow upside down]{\arrowBox};

\draw[line width=0.65mm](0.7,2)--node[above]{$E_1$}(2+0.3+0.3+1,2)node[
    sloped,
    pos=0.7,
    allow upside down]{\arrowBox};

\draw(1+0.2,2.2)--node[above left]{$E_2$}(0-0.2,0.8)node[sloped,pos=0.5,allow upside down]{\arrowIIn};

\draw(1+0.2+1.3,2.2)--node[above left]{$F$}(0-0.2+1.3,0.8)node[sloped,pos=0.5,allow upside down]{\arrowIIn};

\draw(2.5+1,0.9)--node[below right=0.1cm]{$E_3$}(1+0.1+1,-0.5)node[sloped,pos=0.5,allow upside down]{\arrowIIn};

\node[left=0.2cm] at (0,1){$[1,-1]$};
\node[below left=0.2cm] at (0,-0.3){$[-1,0]$};
\node[above left=0.2cm] at (1,2){$[1,0]$};
\node[right=0.2cm] at (2.3+1,2){$[0,1]$};
\node[below right=0.2cm] at (1.1+1,-0.3){$[0,-1]$};
\node[right=0.2cm] at (2.3+1,0.7){$[1,-1]$};

\end{tikzpicture}
$$
\end{minipage}

\subsection{Case \ref{case2}}\label{sec:case2}
We now move on to the classification of pairs $(\widetilde{M},\mathfrak{E})$ in case \ref{case2}.  Recall weights and flows on $\mathbb{P}^2$ in case \ref{case2} is given in Section \ref{sec:actiononp2}. 
Note that $c_-$ in $\mathbb{P}^2$ is already the fixed curve $Z$, having self-intersection 1. The repulsive set $c_-$ in $\widetilde{M}$ is  the proper transform of $Z$, which must be a $(-2)$-curve. Three blow-ups in $Z$ must be performed to achieve this. We assume that the first blow-up $\pi_1$ is at $p_1=X\cap Z$. Flows and weights on $M_1$ are
$$
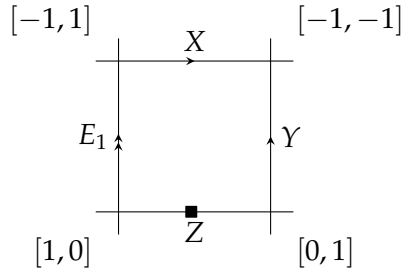
\begin{tikzpicture}
[decoration={markings, 
    mark= at position 0.5 with {\arrow{stealth }}}
]

\draw(0,-0.3)--node[left]{$E_1$}(0,2.3)node[sloped,pos=0.5,allow upside down]{\arrowIIn};

\draw[postaction={decorate}](2,-0.3)--node[right]{$Y$}(2,2.3);

\draw(-0.3,0)--node[below]{$Z$}(2+0.3,0)node[
    sloped,
    pos=0.5,
    allow upside down]{\arrowBox};

\draw[postaction={decorate}](-0.3,2)--node[above]{$X$}(2+0.3,2);

\node[above left=0.2cm] at (0,2){$[-1,1]$};
\node[below left=0.2cm] at (0,0){$[1,0]$};
\node[below right=0.2cm] at (2,0){$[0,1]$};
\node[above right=0.2cm] at (2,2){$[-1,-1]$};
\end{tikzpicture}.
$$
In $M_1$, $Z$ is a $0$-curve. Two more blow-ups must be performed in  $Z$,  but $Z\cap E_1$ can not be blown up because of Corollary \ref{cor:blowup}. Therefore, the two additional blow-ups must be performed at distinct points in $Z$ that are different from the point $Z\cap E_1$. After two such blow-ups, the resulting surface is $M_3$, flows and weights on which are
$$
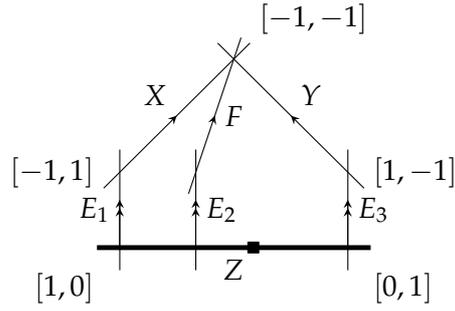
\begin{tikzpicture}
[decoration={markings, 
    mark= at position 0.5 with {\arrow{stealth }}}
]

\draw[shorten >=-0.3cm,shorten <=-0.3cm](0,0)--node[left]{$E_1$}node[sloped,pos=0.3,allow upside down]{\arrowIIn}(0,1);

\draw[shorten >=-0.3cm,shorten <=-0.3cm][postaction={decorate}](3,1)--node[above right]{$Y$}(1.5,2.5);

\draw[shorten >=-0.3cm,shorten <=-0.3cm](3,0)--node[sloped,pos=0.3,allow upside down]{\arrowIIn}node[right]{$E_3$}(3,1);

\draw[shorten >=-0.3cm,shorten <=-0.3cm][postaction={decorate}](1,1)--node[right]{$F$}(1.5,2.5);

\draw[shorten >=-0.3cm,shorten <=-0.3cm](1,0)--node[right]{$E_2$}(1,1)node[sloped,pos=0.3,allow upside down]{\arrowIIn};

\draw[line width=0.65mm][shorten >=-0.3cm,shorten <=-0.3cm](0,0)--node[below]{$Z$}(3,0)node[
    sloped,
    pos=0.5,
    allow upside down]{\arrowBox};

\draw[shorten >=-0.3cm,shorten <=-0.3cm][postaction={decorate}](0,1)--node[above left]{$X$}(1.5,2.5);

\node[left=0.2cm] at (0,1){$[-1,1]$};
\node[below left=0.2cm] at (0,0){$[1,0]$};
\node[below right=0.2cm] at (3,0){$[0,1]$};
\node[above right=0.2cm] at (1.5,2.5){$[-1,-1]$};
\node[right=0.2cm] at (3,1){$[1,-1]$};
\end{tikzpicture}.
$$
Now, points that can be blown up in $M_3$ are $X\cap E_1$, $F\cap E_2$, $Y\cap E_3$, and $X\cap Y$.

\begin{lemma}\label{lem:blowupxy}
If there is a further blow-up $\pi_4$ at $p_4=X\cap Y\in M_3$, no further blow-ups can be made in $M_4$.
\end{lemma}
\begin{proof}
Flows and weights on $M_4$ are
$$
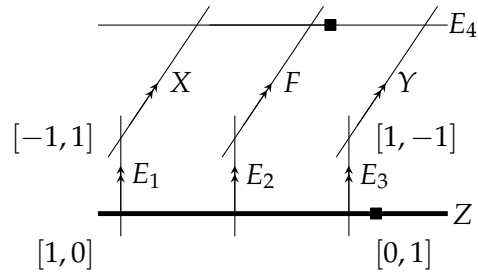
\begin{tikzpicture}
[decoration={markings, 
    mark= at position 0.5 with {\arrow{stealth }}}
]

\draw[shorten >=-0.3cm,shorten <=-0.3cm](0,0)--node[right]{$E_1$}(0,1)node[sloped,pos=0.3,allow upside down]{\arrowIIn};

\draw[shorten >=-0.3cm,shorten <=-0.3cm](3,1)--node[right]{$Y$}(4,2.5)node[sloped,pos=0.3,allow upside down]{\arrowIIn};

\draw[shorten >=-0.3cm,shorten <=-0.3cm](3,0)--node[right]{$E_3$}(3,1)node[sloped,pos=0.3,allow upside down]{\arrowIIn};

\draw[shorten >=-0.3cm,shorten <=-0.3cm](1.5,1)--node[right]{$F$}(2.5,2.5)node[sloped,pos=0.3,allow upside down]{\arrowIIn};

\draw[shorten >=-0.3cm,shorten <=-0.3cm](1.5,0)--node[right]{$E_2$}(1.5,1)node[sloped,pos=0.3,allow upside down]{\arrowIIn};

\draw[line width=0.65mm][shorten >=-1.3cm,shorten <=-0.3cm](0,0)--(3,0)node[
    sloped,
    pos=0.7,
    allow upside down]{\arrowBox};

\draw[shorten >=-1.3cm,shorten <=-0.3cm](0,2.5)--(3,2.5)node[
    sloped,
    pos=0.5,
    allow upside down]{\arrowBox};

\draw[shorten >=-0.3cm,shorten <=-0.3cm](0,1)--node[right]{$X$}(1,2.5)node[sloped,pos=0.3,allow upside down]{\arrowIIn};

\node[left=0.2cm] at (0,1){$[-1,1]$};
\node[below left=0.2cm] at (0,0){$[1,0]$};
\node[below right=0.2cm] at (3,0){$[0,1]$};
\node[right=0.2cm] at (3,1){$[1,-1]$};
\node at (4.5,0){$Z$};
\node at (4.5,2.5){$E_4$};
\end{tikzpicture}.
$$
Now,
\begin{itemize}
\item Points in $E_4$ can not be blown up because $E_4$ is already a $(-1)$-curve. 
\item $X\cap E_1,F\cap E_2,Y\cap E_3$ can not be blown up because of Corollary \ref{cor:blowup}. 
\end{itemize}
The lemma is proved.
\end{proof}

If instead there is a further blow-up $\pi_4$ at $p_4=X\cap E_1$ (the cases of $p_4=F\cap E_2$ or $p_4=Y\cap E_4$ are similar), flows and weights on $M_4$ are
$$
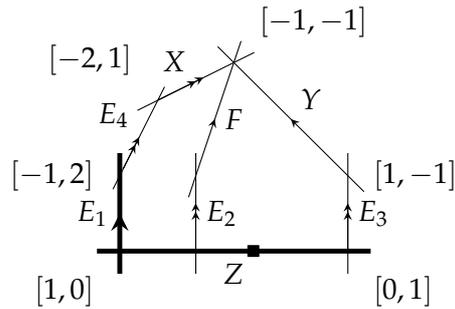
\begin{tikzpicture}
[decoration={markings, 
    mark= at position 0.5 with {\arrow{stealth }}}
]

\draw[postaction={decorate}][line width=0.65mm][shorten >=-0.3cm,shorten <=-0.3cm](0,0)--node[left]{$E_1$}(0,1);

\draw[shorten >=-0.3cm,shorten <=-0.3cm][postaction={decorate}](3,1)--node[above right]{$Y$}(1.5,2.5);

\draw[shorten >=-0.3cm,shorten <=-0.3cm](3,0)--node[right]{$E_3$}(3,1)node[sloped,pos=0.3,allow upside down]{\arrowIIn};

\draw[shorten >=-0.3cm,shorten <=-0.3cm][postaction={decorate}](1,1)--node[right]{$F$}(1.5,2.5);

\draw[shorten >=-0.3cm,shorten <=-0.3cm](1,0)--node[right]{$E_2$}(1,1)node[sloped,pos=0.3,allow upside down]{\arrowIIn};

\draw[line width=0.65mm][shorten >=-0.3cm,shorten <=-0.3cm](0,0)--node[below]{$Z$}(3,0)node[
    sloped,
    pos=0.5,
    allow upside down]{\arrowBox};

\draw[shorten >=-0.2cm,shorten <=-0.2cm](0,1)--node[above left]{$E_4$}(0.5,2)node[sloped,pos=0.3,allow upside down]{\arrowIIn};

\draw[shorten >=-0.3cm,shorten <=-0.3cm](0.5,2)--node[above left]{$X$}(1.5,2.5)node[sloped,pos=0.3,allow upside down]{\arrowIIn};

\node[left=0.2cm] at (0,1){$[-1,2]$};
\node[below left=0.2cm] at (0,0){$[1,0]$};
\node[below right=0.2cm] at (3,0){$[0,1]$};
\node[above right=0.2cm] at (1.5,2.5){$[-1,-1]$};
\node[right=0.2cm] at (3,1){$[1,-1]
$};
\node[above left=0.2cm] at (0.5,2){$[-2,1]
$};
\end{tikzpicture}.
$$
\begin{lemma}\label{lem:blowupxe}
If there is a further blow-up $\pi_4$ at $p_4=X\cap E_1\in M_3$, further blow-ups can only be taken at $F\cap E_2,Y\cap E_3$.
\end{lemma}
\begin{proof}
$E_1\cap E_4,X\cap E_4$ can not be blown up because of Corollary \ref{cor:blowup}. $X\cap Y$ cannot be blown up since this will make $E$ disconnected in $\widetilde{M}$. Points in $Z$ cannot be blown up because of Lemma \ref{lem:blowup}.
\end{proof}

To summarize, based on Lemma \ref{lem:blowupxy} and Lemma \ref{lem:blowupxe}, we obtain 5 possible $(\widetilde{M},\mathfrak{E})$ by examining all the possibilities, denoted by Case 2A, 2B, 2C, 2D, and 2E.
However, it turns out that for Case 2C the action by $\mathfrak{E}$ on $\widehat{M}$ does not have same weights at the orbifold point $q$. Therefore, there are actually only 4 possible structures of $(\widetilde{M},\mathfrak{E})$ in case \ref{case2}.

\begin{proposition}\label{prop:case2}
There are 4 possible pairs $(\widetilde{M},\mathfrak{E})$ in case \ref{case2}, namely the following Case 2A, Case 2B, Case 2D, and Case 2E.
\end{proposition}

\begin{minipage}{0.5\textwidth}

\paragraph{\underline{\underline{Case 2A}}}
\label{2A}

\begin{itemize}
\item The orbifold group is $A_1$.
\item The degree of $$\widetilde{M}=Bl_{X\cap Z,p_2,Y\cap Z}\mathbb{P}^2$$ is 6 and the picard number of $\widehat{M}$ is 3. Here $p_2$ is a point in $Z$ different from $E_1\cap Z,Y\cap Z$. 
\item $K_{\widetilde{M}}=-3Z-2E_1-2E_2-2E_3$.
\end{itemize}

\end{minipage}%
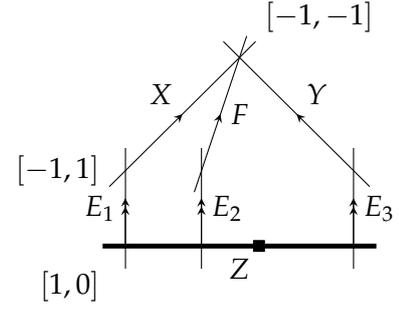
\begin{minipage}{0.4\textwidth}
$$
\begin{tikzpicture}
[decoration={markings, 
    mark= at position 0.5 with {\arrow{stealth }}}
]

\draw[shorten >=-0.3cm,shorten <=-0.3cm](0,0)--node[left]{$E_1$}(0,1)node[sloped,pos=0.3,allow upside down]{\arrowIIn};

\draw[shorten >=-0.3cm,shorten <=-0.3cm][postaction={decorate}](3,1)--node[above right]{$Y$}(1.5,2.5);

\draw[shorten >=-0.3cm,shorten <=-0.3cm](3,0)--node[right]{$E_3$}(3,1)node[sloped,pos=0.3,allow upside down]{\arrowIIn};

\draw[shorten >=-0.3cm,shorten <=-0.3cm][postaction={decorate}](1,1)--node[right]{$F$}(1.5,2.5);

\draw[shorten >=-0.3cm,shorten <=-0.3cm](1,0)--node[right]{$E_2$}(1,1)node[sloped,pos=0.3,allow upside down]{\arrowIIn};

\draw[line width=0.65mm][shorten >=-0.3cm,shorten <=-0.3cm](0,0)--node[below]{$Z$}(3,0)node[
    sloped,
    pos=0.5,
    allow upside down]{\arrowBox};

\draw[shorten >=-0.3cm,shorten <=-0.3cm][postaction={decorate}](0,1)--node[above left]{$X$}(1.5,2.5);

\node[left=0.2cm] at (0,1){$[-1,1]$};
\node[below left=0.2cm] at (0,0){$[1,0]$};
\node[above right=0.2cm] at (1.5,2.5){$[-1,-1]$};

\end{tikzpicture}
$$
\end{minipage}

\begin{minipage}{0.5\textwidth}

\paragraph{\underline{\underline{Case 2B}}}
\label{2B}

\begin{itemize}
    \item The pair $(\widetilde{M},\mathfrak{E})$ is biholomorphic to the pair of Case \hyperref[1C]{1C}.
\end{itemize}

\end{minipage}%
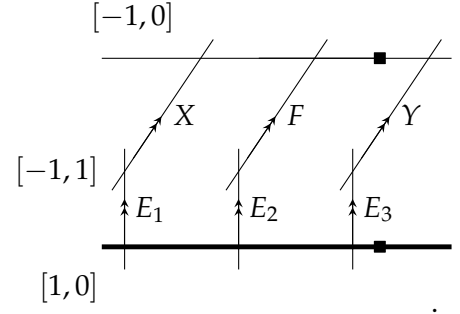
\begin{minipage}{0.4\textwidth}
$$
\begin{tikzpicture}
[decoration={markings, 
    mark= at position 0.5 with {\arrow{stealth }}}
]

\draw[shorten >=-0.3cm,shorten <=-0.3cm](0,0)--node[right]{$E_1$}(0,1)node[sloped,pos=0.3,allow upside down]{\arrowIIn};

\draw[shorten >=-0.3cm,shorten <=-0.3cm](3,1)--node[right]{$Y$}(4,2.5)node[sloped,pos=0.3,allow upside down]{\arrowIIn};

\draw[shorten >=-0.3cm,shorten <=-0.3cm](3,0)--node[right]{$E_3$}(3,1)node[sloped,pos=0.3,allow upside down]{\arrowIIn};

\draw[shorten >=-0.3cm,shorten <=-0.3cm](1.5,1)--node[right]{$F$}(2.5,2.5)node[sloped,pos=0.3,allow upside down]{\arrowIIn};

\draw[shorten >=-0.3cm,shorten <=-0.3cm](1.5,0)--node[right]{$E_2$}(1.5,1)node[sloped,pos=0.3,allow upside down]{\arrowIIn};

\draw[line width=0.65mm][shorten >=-1.3cm,shorten <=-0.3cm](0,0)--(3,0)node[
    sloped,
    pos=0.7,
    allow upside down]{\arrowBox};

\draw[shorten >=-1.3cm,shorten <=-0.3cm](0,2.5)--(3,2.5)node[
    sloped,
    pos=0.7,
    allow upside down]{\arrowBox};

\draw[shorten >=-0.3cm,shorten <=-0.3cm](0,1)--node[right]{$X$}(1,2.5)node[sloped,pos=0.3,allow upside down]{\arrowIIn};

\node[left=0.2cm] at (0,1){$[-1,1]$};
\node[below left=0.2cm] at (0,0){$[1,0]$};
\node[above left=0.2cm] at (1,2.5){$[-1,0]$};

\end{tikzpicture}
$$
\end{minipage}

\begin{minipage}{0.5\textwidth}

\paragraph{\underline{\underline{Case 2C}}}
\label{2C}

\begin{itemize}
    \item This can be excluded since by Theorem \ref{thm:cyclicweights}, the action by $\mathfrak{E}$ on $\widehat{M}$ at the orbifold point $q$ has weights $[\frac23,\frac13]$.
\end{itemize}

\end{minipage}%
\begin{minipage}{0.4\textwidth}
$$
\begin{tikzpicture}
[decoration={markings, 
    mark= at position 0.5 with {\arrow{stealth }}}
]

\draw[postaction={decorate}][line width=0.65mm][shorten >=-0.3cm,shorten <=-0.3cm](0,0)--node[left]{$E_1$}(0,1);

\draw[shorten >=-0.3cm,shorten <=-0.3cm][postaction={decorate}](3,1)--node[above right]{$Y$}(1.5,2.5);

\draw[shorten >=-0.3cm,shorten <=-0.3cm](3,0)--node[right]{$E_3$}(3,1)node[sloped,pos=0.3,allow upside down]{\arrowIIn};

\draw[shorten >=-0.3cm,shorten <=-0.3cm][postaction={decorate}](1,1)--node[right]{$F$}(1.5,2.5);

\draw[shorten >=-0.3cm,shorten <=-0.3cm](1,0)--node[right]{$E_2$}(1,1)node[sloped,pos=0.3,allow upside down]{\arrowIIn};

\draw[line width=0.65mm][shorten >=-0.3cm,shorten <=-0.3cm](0,0)--node[below]{$Z$}(3,0)node[
    sloped,
    pos=0.5,
    allow upside down]{\arrowBox};

\draw[shorten >=-0.2cm,shorten <=-0.2cm](0,1)--node[above left]{$E_4$}(0.5,2)node[sloped,pos=0.3,allow upside down]{\arrowIIn};

\draw[shorten >=-0.3cm,shorten <=-0.3cm](0.5,2)--node[above left]{$X$}(1.5,2.5)node[sloped,pos=0.3,allow upside down]{\arrowIIn};

\node[left=0.2cm] at (0,1){$[-1,2]$};
\node[below left=0.2cm] at (0,0){$[1,0]$};
\node[below right=0.2cm] at (3,0){$[0,1]$};

\end{tikzpicture}
$$
\end{minipage}

\begin{minipage}{0.5\textwidth}

\paragraph{\underline{\underline{Case 2D}}}
\label{2D}

\begin{itemize}
\item The orbifold group is $A_3$.
\item The degree of 
$$\widetilde{M}=Bl_{X\cap Z,p_2,Y\cap Z,X\cap E_1,Y\cap E_3}\mathbb{P}^2$$ is 4 and the picard number of $\widehat{M}$ is 3. Here $p_2$ is a point in $Z$ different from $E_1\cap Z,Y\cap Z$. 
\item $K_{\widetilde{M}}=-3Z-2E_1-2E_2-2E_3-E_4-E_5$.
\end{itemize}

\end{minipage}%
\begin{minipage}{0.4\textwidth}
$$
\begin{tikzpicture}
[decoration={markings, 
    mark= at position 0.5 with {\arrow{stealth }}}
]

\draw[postaction={decorate}][line width=0.65mm][shorten >=-0.3cm,shorten <=-0.3cm](0,0)--node[left]{$E_1$}(0,1);

\draw[shorten >=-0.3cm,shorten <=-0.3cm](3.5,2)--node[above right]{$Y$}(1.5,3)node[sloped,pos=0.3,allow upside down]{\arrowIIn};

\draw[shorten >=-0.2cm,shorten <=-0.2cm](4,1)--node[right]{$E_5$}(3.5,2)node[sloped,pos=0.3,allow upside down]{\arrowIIn};

\draw[line width=0.65mm][postaction={decorate}][shorten >=-0.3cm,shorten <=-0.3cm](4,0)--node[right]{$E_3$}(4,1);

\draw[shorten >=-0.3cm,shorten <=-0.3cm][postaction={decorate}](2.5,1)--node[right]{$F$}(1.5,3);

\draw[shorten >=-0.3cm,shorten <=-0.3cm](2.5,0)--node[right]{$E_2$}(2.5,1)node[sloped,pos=0.3,allow upside down]{\arrowIIn};

\draw[line width=0.65mm][shorten >=-0.3cm,shorten <=-0.3cm](0,0)--(4,0)node[
    sloped,
    pos=0.35,
    allow upside down]{\arrowBox};

\draw[shorten >=-0.2cm,shorten <=-0.2cm](0,1)--node[above left]{$E_4$}(0.5,2)node[sloped,pos=0.3,allow upside down]{\arrowIIn};

\draw[shorten >=-0.3cm,shorten <=-0.3cm](0.5,2)--node[above left]{$X$}(1.5,3)node[sloped,pos=0.3,allow upside down]{\arrowIIn};

\node at (-0.5,0) {$Z$};

\node[left=0.2cm] at (0,1){$[-1,2]$};
\node[below left=0.2cm] at (0,0){$[1,0]$};
\node[below right=0.2cm] at (4,0){$[0,1]$};
\node[above left=0.35cm] at (0.5,1.5){$[-2,1]$};
\node[above left=0.35cm] at (2,3){$[-1,-1]$};
\node[left=0.1cm] at (2.5,1){$[1,-1]$};

\end{tikzpicture}
$$
\end{minipage}

\begin{minipage}{0.5\textwidth}

\paragraph{\underline{\underline{Case 2E}}}
\label{2E}

\begin{itemize}
\item The orbifold group is $D_4$.
\item The degree of 
$$\widetilde{M}=Bl_{X\cap Z,p_2,Y\cap Z,X\cap E_1,Y\cap E_3,F\cap E_2}\mathbb{P}^2$$ is 3 and the picard number of $\widehat{M}$ is 3. Here $p_2$ is a point in $Z$ different from $E_1\cap Z,Y\cap Z$. 
\item $K_{\widetilde{M}}=-3Z-2E_1-2E_2-2E_3-E_4-E_5-E_6$.
\end{itemize}

\end{minipage}%
\begin{minipage}{0.4\textwidth}
$$
\begin{tikzpicture}
[decoration={markings, 
    mark= at position 0.5 with {\arrow{stealth }}}
]

\draw[postaction={decorate}][line width=0.65mm][shorten >=-0.3cm,shorten <=-0.3cm](0,0)--node[left]{$E_1$}(0,1);

\draw[shorten >=-0.3cm,shorten <=-0.3cm](3.5,2)--node[above right]{$Y$}(1.5,3);
\draw[line width=0.0001mm](3.5,2)--node[sloped,pos=0.5,allow upside down]{\arrowIIn}(1.5,3);

\draw[shorten >=-0.2cm,shorten <=-0.2cm](4,1)--node[right]{$E_5$}(3.5,2);
\draw[line width=0.0001mm](4,1)--node[sloped,pos=0.5,allow upside down]{\arrowIIn}(3.5,2);

\draw[line width=0.65mm][postaction={decorate}][shorten >=-0.3cm,shorten <=-0.3cm](4,0)--node[right]{$E_3$}(4,1);

\draw[shorten >=-0.3cm,shorten <=-0.3cm](2.5,2)--node[below left]{$F$}(1.5,3);
\draw[line width=0.0001mm](2.5,2)--node[sloped,pos=0.5,allow upside down]{\arrowIIn}(1.5,3);

\draw[shorten >=-0.2cm,shorten <=-0.2cm](3,1)--node[below left]{$E_6$}(2.5,2);
\draw[line width=0.0001mm](3,1)--node[sloped,pos=0.5,allow upside down]{\arrowIIn}(2.5,2);

\draw[postaction={decorate}][line width=0.65mm][shorten >=-0.3cm,shorten <=-0.3cm](3,0)--node[right]{$E_2$}(3,1);

\draw[line width=0.65mm][shorten >=-0.3cm,shorten <=-0.3cm](0,0)--(4,0);
\draw[line width=0.001mm](0,0)--(4,0)node[
    sloped,
    pos=0.35,
    allow upside down]{\arrowBox};

\draw[shorten >=-0.2cm,shorten <=-0.2cm](0,1)--node[above left]{$E_4$}(0.5,2);
\draw[line width=0.01mm](0,1)--node[sloped,pos=0.5,allow upside down]{\arrowIIn}(0.5,2);

\draw[shorten >=-0.3cm,shorten <=-0.3cm](0.5,2)--node[above left]{$X$}(1.5,3);
\draw[line width=0.01mm](0.5,2)--node[sloped,pos=0.5,allow upside down]{\arrowIIn}(1.5,3);

\node at (-0.5,0) {$Z$};

\node[left=0.2cm] at (0,1){$[-1,2 ]$};
\node[below left=0.2cm] at (0,0){$[1,0]$};
\node[below right=0.2cm] at (4,0){$[0,1]$};
\node[above left=0.35cm] at (0.5,1.5){$[-2 ,1]$};
\node[above left=0.35cm] at (2,3){$[-1,-1]$};

\end{tikzpicture}
$$
\end{minipage}

\section{Case \ref{case3}}
\label{sec:case3}

This section focuses on case \ref{case3}. 
Recall fixed points and flows of the action on $\mathbb{P}^2$ are the following with $\alpha>\beta$.
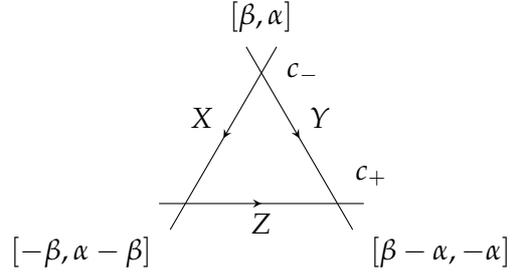
$$
\begin{tikzpicture}
[decoration={markings, 
    mark= at position 0.5 with {\arrow{stealth }}}
]
\draw[postaction={decorate}](1.2,1.732+1.732*0.2)--node[above left]{$X$}(-0.2,-0.2*1.732);
\draw[postaction={decorate}](0.8,1.732+1.732*0.2)--node[above right]{$Y$}(2.2,-1.732*0.2);
\draw[postaction={decorate}](-0.2*1.732,0)--node[below]{$Z$}(2+0.2*1.732,0);

\node[above=0.4cm] at (1,1.732){$[\beta,\alpha]$};
\node[right=0.2cm] at (1,1.732){$c_-$};
\node[below left=0.3cm] at (0,0){$[-\beta,\alpha-\beta]$};
\node[below right=0.3cm] at (2,0){$[\beta-\alpha,-\alpha]$};
\node[above right=0.1cm] at (2,0){$c_+$};
\end{tikzpicture}.
$$

\subsection{Cyclic $\Gamma$}\label{cycliccase3}




The first observation is that there must be a blow-up at $X\cap Y$.
\begin{lemma}\label{lem:m1}
There must be a blow-up at $X\cap Y$. Therefore, $\widetilde{M}$ is a blow-up of $M_1=Bl_{X\cap Y}\mathbb{P}^2$.
\end{lemma}
\begin{proof}
If there were no blow-up at $X\cap Y$, then $c_-=X\cap Y$ in $\widetilde{M}$ and $E$ would either contain the proper transform of $X$ or the proper transform of $Y$.
If the proper transform of $X$ were contained in $E$, then it would be a $(-2)$-curve. Blow-up must be performed at $p_1=X\cap Z$ in $\mathbb{P}^2$. But Corollary \ref{cor:blowup} prevents further blow-up at $X\cap E_1$. Therefore this case cannot happen.
The situation is exactly the same if we assume the proper transform of $Y$ is contained in $E$. 
\end{proof}

Let $M_1$ be $Bl_{X\cap Y}\mathbb{P}^2$. 
Weights and direction of flows on $M_1$ are 
$$
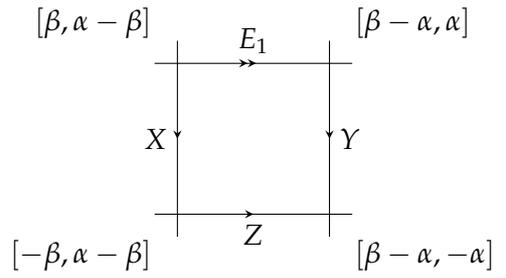
\begin{tikzpicture}
[decoration={markings, 
    mark= at position 0.5 with {\arrow{stealth }}}
]
\draw[postaction={decorate}](0,2.3)--node[left]{$X$}(0,-0.3);
\draw[postaction={decorate}](2,2.3)--node[right]{$Y$}(2,-0.3);
\draw[postaction={decorate}](-0.3,0)--node[below]{$Z$}(2+0.3,0);
\draw(-0.3,2)--node[above]{$E_1$}(2+0.3,2)node[sloped,pos=0.5,allow upside down]{\arrowIIn};

\node[above left=0.2cm] at (0,2){$[\beta,\alpha-\beta]$};
\node[below left=0.2cm] at (0,0){$[-\beta,\alpha-\beta]$};
\node[below right=0.2cm] at (2,0){$[\beta-\alpha,-\alpha]$};
\node[above right=0.2cm] at (2,2){$[\beta-\alpha,\alpha]$};
\end{tikzpicture}.
$$
Further blow-ups must be performed. 
Because of Corollary \ref{cor:weightsdeterminepq}, weights at either of the two ending fixed points in $E$ determine $p$ and $q$. 
Because $E$ must contain $c_-$ in $\widetilde{M}$, if there is no blow-up at $X\cap E_1$, then either the proper transform of $X$, or the proper transform of $E_1$, is contained in $E$. Therefore, if there is no blow-up at $X\cap E_1$, then either $X\cap Z$, or $E_1\cap Y$, must be blown up. It follows that
the next blow-up $\pi_2$ can be chosen to be one of the following
\begin{itemize}
\item $\pi_2$ is the blow-up at $p_2=E_1\cap Y$.
\item $\pi_2$ is the blow-up at $p_2=X\cap E_1$.
\item $\pi_2$ is the blow-up at $p_2=X\cap Z$, and there is no blow-up at $X\cap E_1,E_1\cap Y$ in further blow-ups.
\end{itemize}
These three cases are studied in the following Section \ref{subsubsec:1}, Section \ref{subsubsec:2}, and Section \ref{subsubsec:3}.

\subsubsection{$p_2=E_1\cap Y$}
\label{subsubsec:1}

Then $M_2$ is
$$
\begin{tikzpicture}
[decoration={markings, 
    mark= at position 0.5 with {\arrow{stealth }}}
]
\draw[postaction={decorate}](0,2.3)--node[left]{$X$}(0,-0.3);
\draw(2,1.3)--node[right]{$Y$}(2,-0.3)node[sloped,pos=0.5,allow upside down]{\arrowIIn};
\draw[postaction={decorate}](-0.3,0)--node[below]{$Z$}(2+0.3,0);
\draw[line width=0.65mm][postaction={decorate}](-0.3,2)--node[above]{$E_1$}(1+0.3,2);
\draw(1-0.2,2.2)--node[above right]{$E_2$}(2+0.2,1-0.2)node[sloped,pos=0.5,allow upside down]{\arrowIIn};

\node[above left=0.2cm] at (0,2){$[\beta,\alpha-\beta]$};
\node[below left=0.2cm] at (0,0){$[-\beta,\alpha-\beta]$};
\node[below right=0.2cm] at (2,0){$[\beta-\alpha,-\alpha]$};
\node[above right=0.2cm] at (1,2){$[\beta-\alpha,2\alpha-\beta]$};
\node[above right=0.2cm] at (2,1){$[\beta-2\alpha,\alpha]$};
\end{tikzpicture}
$$

Now:
\begin{itemize}
\item $X\cap E_1$ and $E_1\cap E_2$ cannot be blown up because of Lemma \ref{lem:blowup}.
\item $E_2\cap Y$ cannot be blown up because of Corollary \ref{cor:blowup}.
\item $Z\cap Y$ cannot be blown up because this will make $E$ in $\widetilde{M}$ disconnected.
\end{itemize}
Additional blow-ups are necessary since the weights on $E_1$ does not satisfy $\beta=2\alpha-\beta$. If we contract $E_1$ now, the action by $\mathfrak{E}$ at the orbifold point $q$ does not have same weights.
Indeed, contracting $E_1$ gives an $A_1$ singularity, and the  weights are $[\alpha-\frac\beta2,\frac\beta2]$ at the orbifold point.
Hence there must be a blow-up at $X\cap Z$, and $p_3$ can be taken as $X\cap Z$.  So $M_3$ is
$$
\begin{tikzpicture}
[decoration={markings, 
    mark= at position 0.5 with {\arrow{stealth }}}
]

\draw(-0.2,0.2)--node[below left]{$E_3$}(1.2,-1.2)node[sloped,pos=0.5,allow upside down]{\arrowIIn};

\draw[postaction={decorate}](0.7,-1)--node[below]{$Z$}(2.6,-1);

\draw(2.3,0.6)--node[right]{$Y$}(2.3,-1.3)node[sloped,pos=0.5,allow upside down]{\arrowIIn};

\draw(1.1,1.5)--node[above right]{$E_2$}(2.5,0.1)node[sloped,pos=0.5,allow upside down]{\arrowIIn};

\draw[line width=0.65mm][postaction={decorate}](-0.3,1.3)--node[above]{$E_1$}(1.6,1.3);

\draw(0,1.6)--node[left]{$X$}(0,-0.3)node[sloped,pos=0.5,allow upside down]{\arrowIIn};

\node[below left=0.1cm] at (0,0){$[-\beta,\alpha]$};
\node[below left=0.1cm] at (1,-1){$[-\alpha,\alpha-\beta]$};
\node[below right=0.1cm] at (2.3,-1){$[\beta-\alpha,-\alpha]$};
\node[above right=0.1cm] at (2.3,0.3){$[\beta-2\alpha,\alpha]$};
\node[above right=0.1cm] at (1.3,1.3){$[\beta-\alpha,2\alpha-\beta]$};
\node[above left=0.1cm] at (0,1.3){$[\beta,\alpha-\beta]$};

\end{tikzpicture}.
$$

Now:
\begin{itemize}
\item $E_3\cap X, X\cap E_1, E_1\cap E_2,E_2\cap Y$ all cannot be blown up because of Corollary \ref{cor:blowup}.
\item $E_3\cap Z,Y\cap Z$ cannot be blown up because such blow-ups will make $E$ in $\widetilde{M}$ disconnected.
\end{itemize}
No further blow-ups can be taken in $M_3$. This leads to a contradiction, as after contracting $E_1$ the weights at the orbifold point are still $[\alpha-\frac{\beta}{2},\frac{\beta}{2}]$.
In conclusion, $E_1\cap Y$ cannot be blown up in $M_1$.

\subsubsection{$p_2=X\cap E_1$}
\label{subsubsec:2}

The weights and direction of flows on $M_2$ are as follows:
$$
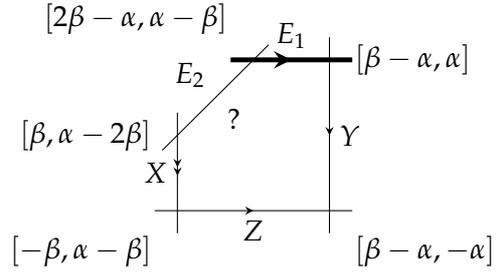
\begin{tikzpicture}
[decoration={markings, 
    mark= at position 0.5 with {\arrow{stealth }}}
]

\draw(0,1.3)--node[left]{$X$}(0,-0.3)node[sloped,pos=0.5,allow upside down]{\arrowIIn};

\draw[postaction={decorate}](2,2.3)--node[right]{$Y$}(2,-0.3);
\draw[postaction={decorate}](-0.3,0)--node[below]{$Z$}(2+0.3,0);
\draw[line width=0.65mm][postaction={decorate}](0.7,2)--node[above]{$E_1$}(2+0.3,2);
\draw(0-0.2,0.8)--node[above left]{$E_2$}node[below right=0.02cm]{$?$}(1+0.2,2.2);

\node[left=0.2cm] at (0,1){$[\beta,\alpha-2\beta]$};
\node[below left=0.2cm] at (0,0){$[-\beta,\alpha-\beta]$};
\node[above left=0.2cm] at (1,2){$[2\beta-\alpha,\alpha-\beta]$};
\node[right=0.2cm] at (2,2){$[\beta-\alpha,\alpha]$};
\node[below right=0.2cm] at (2,0){$[\beta-\alpha,-\alpha]$};
\end{tikzpicture}.
$$
The direction of the flow on $E_2$ is determined by the sign of $\alpha-2\beta$.
\begin{lemma}\label{lem:propery}
The proper transform of $Y$ in $\widetilde{M}$ cannot be contained in $E$.
\end{lemma}
\begin{proof}
If the proper transform of $Y$ in $\widetilde{M}$ were contained in $E$, then it would have self-intersection $-2$. Blow-up must be performed at $Y\cap Z$ in $M_2$. Take $p_3=Y\cap Z$. Then $M_3$ is
$$
\begin{tikzpicture}
[decoration={markings, 
    mark= at position 0.5 with {\arrow{stealth }}}
]

\draw(0,1.3)--node[left]{$X$}(0,-0.3-0.3)node[sloped,pos=0.5,allow upside down]{\arrowIIn};

\draw(2.3,2.3)--node[right]{$Y$}(2.3,0.4)node[sloped,pos=0.5,allow upside down]{\arrowIIn};

\draw[postaction={decorate}](-0.3,-0.3)--node[below]{$Z$}(1+0.6,-0.3);
\draw[line width=0.65mm][postaction={decorate}](0.7,2)--node[above]{$E_1$}(2+0.3+0.3,2);
\draw(0-0.2,0.8)--node[above left]{$E_2$}node[below right=0.02cm]{$?$}(1+0.2,2.2);

\draw(2.5,0.9)--node[below right=0.1cm]{$E_3$}node[sloped,pos=0.5,allow upside down]{\arrowIIn}(1+0.1,-0.5);

\node[left=0.2cm] at (0,1){$[\beta,\alpha-2\beta]$};
\node[below left=0.2cm] at (0,-0.3){$[-\beta,\alpha-\beta]$};
\node[above left=0.2cm] at (1,2){$[2\beta-\alpha,\alpha-\beta]$};
\node[right=0.2cm] at (2.3,2){$[\beta-\alpha,\alpha]$};
\node[below right=0.2cm] at (1.1,-0.3){$[\beta-\alpha,-\beta]$};
\node[right=0.2cm] at (2.3,0.7){$[\beta,-\alpha]$};

\end{tikzpicture}.
$$
No further blow-ups can be taken to make $Y$ a $(-2)$-curve. Contradiction.
\end{proof}

It follows that $E_1\cap Y$ must be one of the ending fixed points of $E$ in $\widetilde{M}$. Since the holomorphic $\mathbb{C}^*$-action must have same weights at the orbifold point, we must have  
$$w_0=\theta p=\alpha,\ w_1=\theta(p-1)-\theta=\alpha-\beta,$$
or
$$w_{k}=\theta-\theta(p-1)=\beta-\alpha,\ w_{k+1}=-\theta p=-\alpha.$$
It follows that $\alpha=\theta p$ and $\beta=2\theta$.

\begin{lemma}\label{lem:a=2b}
$E_2$ in $M_2$ is a fixed curve and it is the repulsive set $c_-$ in $M_2$.
\end{lemma}
\begin{proof}
The curve $E_2$ in $M_2$ is a fixed curve if and only if $\alpha=2\beta$. Hence, if the lemma were not true, we would have two cases: either $\alpha>2\beta$, or $\alpha<2\beta$. In the following we will show that neither of these cases is possible.

\paragraph{\underline{\underline{If $\boldsymbol{\alpha>2\beta}$}}}
Flows on $M_2$ are
$$
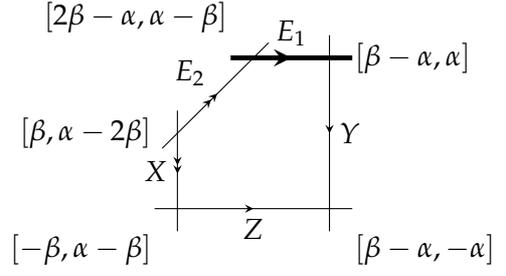
\begin{tikzpicture}
[decoration={markings, 
    mark= at position 0.5 with {\arrow{stealth }}}
]

\draw(0,1.3)--node[left]{$X$}(0,-0.3)node[sloped,pos=0.5,allow upside down]{\arrowIIn};

\draw[postaction={decorate}](2,2.3)--node[right]{$Y$}(2,-0.3);
\draw[postaction={decorate}](-0.3,0)--node[below]{$Z$}(2+0.3,0);
\draw[line width=0.65mm][postaction={decorate}](0.7,2)--node[above]{$E_1$}(2+0.3,2);

\draw(0-0.2,0.8)--node[above left]{$E_2$}(1+0.2,2.2)node[sloped,pos=0.5,allow upside down]{\arrowIIn};

\node[left=0.2cm] at (0,1){$[\beta,\alpha-2\beta]$};
\node[below left=0.2cm] at (0,0){$[-\beta,\alpha-\beta]$};
\node[above left=0.2cm] at (1,2){$[2\beta-\alpha,\alpha-\beta]$};
\node[right=0.2cm] at (2,2){$[\beta-\alpha,\alpha]$};
\node[below right=0.2cm] at (2,0){$[\beta-\alpha,-\alpha]$};
\end{tikzpicture}.
$$

Now:
\begin{itemize}
\item $E_2\cap E_1$ and $E_1\cap Y$ can not be blown up because of Lemma \ref{lem:blowup}.
\item There must be a blow-up at $X\cap E_2$ because $E$ in $\widetilde{M}$ must contain $c_-$.
\end{itemize}
Take $p_3=X\cap E_2$. Weights and direction of flows on $M_3$ are
$$
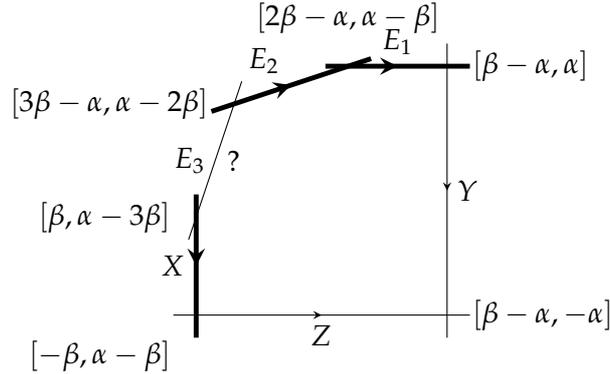
\begin{tikzpicture}
[decoration={markings, 
    mark= at position 0.5 with {\arrow{stealth }}}
]
\draw[line width=0.65mm][postaction={decorate}](0,1.3)--node[left]{$X$}(0,-0.6);
\draw[postaction={decorate}](2.3+1,2.3+1)--node[right]{$Y$}(2.3+1,-0.6);
\draw[postaction={decorate}](-0.3,-0.3)--node[below]{$Z$}(3.6,-0.3);
\draw[line width=0.65mm][postaction={decorate}](0.7+1,2+1)--node[above]{$E_1$}(2.3+0.3+1,2+1);
\draw[line width=0.65mm][postaction={decorate}](0.5-0.3,2.5-0.1)--node[above left]{$E_2$}(1+1+0.3,2+1+0.1);
\draw(0-0.1,1-0.3)--node[left]{$E_3$}node[right=0.02cm]{$?$}(0.5+0.1,2.5+0.3);

\node[left=0.2cm] at (0,1){$[\beta,\alpha-3\beta]$};
\node[below left=0.2cm] at (0,-0.3){$[-\beta,\alpha-\beta]$};
\node[above=0.3cm] at (1+1,2+1){$[2\beta-\alpha,\alpha-\beta]$};
\node[right=0.2cm] at (2.3+1,-0.3){$[\beta-\alpha,-\alpha]$};
\node[right=0.2cm] at (2.3+1,3){$[\beta-\alpha,\alpha]$};
\node[left=0.2cm] at (0.5,2.5){$[3\beta-\alpha,\alpha-2\beta]$};

\end{tikzpicture}.
$$
The direction of the flow on $E_3$ depends on the sign of $\alpha-3\beta$.
Now to ensure $E$ is connected, the only possibility is that $E_3$ is a fixed curve, which implies $\alpha=3\beta$.
To make $E_3$ a $(-2)$-curve, another blow-up at a point in $E_3$ other than $E_3\cap X,E_3\cap E_2$ is necessary.

But $X\cap Z$ can not serve as the other ending fixed point in $E$: If $X\cap Z$ were the other ending fixed point, by Corollary \ref{cor:su2weights}, we would have $\alpha-\beta=\alpha$, which is not possible. 
Thus, the proper transform of $Z$ in $\widetilde{M}$ should also be contained in $E$.
Three further blow-ups in $Z$ must be taken at $Y\cap Z$ to make $Z$ a $(-2)$-curve since $Z$ has self-intersection 1 in $M_3$. This is not possible, because the exceptional curves of these blow-ups contain $(-2)$-curves, which will disconnect $E$.

\paragraph{\underline{\underline{If $\boldsymbol{\alpha<2\beta}$}}}
Then flows on $M_2$ are
$$
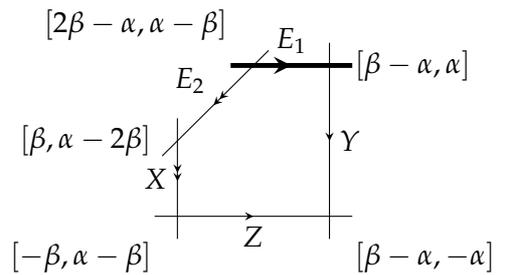
\begin{tikzpicture}
[decoration={markings, 
    mark= at position 0.5 with {\arrow{stealth }}}
]

\draw(0,1.3)--node[left]{$X$}(0,-0.3)node[sloped,pos=0.5,allow upside down]{\arrowIIn};

\draw[postaction={decorate}](2,2.3)--node[right]{$Y$}(2,-0.3);
\draw[postaction={decorate}](-0.3,0)--node[below]{$Z$}(2+0.3,0);
\draw[line width=0.65mm][postaction={decorate}](0.7,2)--node[above]{$E_1$}(2+0.3,2);

\draw(1+0.2,2.2)--node[above left]{$E_2$}(0-0.2,0.8)node[sloped,pos=0.5,allow upside down]{\arrowIIn};

\node[left=0.2cm] at (0,1){$[\beta,\alpha-2\beta]$};
\node[below left=0.2cm] at (0,0){$[-\beta,\alpha-\beta]$};
\node[above left=0.2cm] at (1,2){$[2\beta-\alpha,\alpha-\beta]$};
\node[right=0.2cm] at (2,2){$[\beta-\alpha,\alpha]$};
\node[below right=0.2cm] at (2,0){$[\beta-\alpha,-\alpha]$};
\end{tikzpicture}.
$$

Now $X\cap E_2$ cannot be blown up because of Corollary \ref{cor:blowup}. The exceptional set $E$ in $\widetilde{M}$ must entirely consist of $E_1$, which leads to a contradiction, as $\mathfrak{E}$ would not have same weights at the orbifold point.
\end{proof}

With Lemma \ref{lem:a=2b} being proved, $E_2$ in $M_2$ is a fixed curve. We have $\alpha=2\beta$, and $p=4$. The orbifold group must be $L(3,4)$. As the $\mathbb{C}^*$-action is assumed to be primitive, we should take $\alpha=2$ and $\beta=1$. So $M_2$ is
$$
\begin{tikzpicture}
[decoration={markings, 
    mark= at position 0.5 with {\arrow{stealth }}}
]

\draw(0,1.3)--node[left]{$X$}(0,-0.3)node[sloped,pos=0.5,allow upside down]{\arrowIIn};

\draw[postaction={decorate}](2,2.3)--node[right]{$Y$}(2,-0.3);
\draw[postaction={decorate}](-0.3,0)--node[below]{$Z$}(2+0.3,0);
\draw[line width=0.65mm][postaction={decorate}](0.7,2)--node[above]{$E_1$}(2+0.3,2);

\draw(0-0.2,0.8)--node[above left]{$E_2$}(1+0.2,2.2)node[
    sloped,
    pos=0.5,
    allow upside down]{\arrowBox};

\node[left=0.2cm] at (0,1){$[1,0]$};
\node[below left=0.2cm] at (0,0){$[-1,1]$};
\node[above left=0.2cm] at (1,2){$[0,1]$};
\node[right=0.2cm] at (2,2){$[-1,2]$};
\node[below right=0.2cm] at (2,0){$[-1,-2]$};
\end{tikzpicture}.
$$
Now $X\cap E_2$ cannot be blown up because of Corollary \ref{cor:blowup}. Because $E_2$ is the repulsive set, it has to be contained in $E$. Consequently $E_2$ has to be a $(-2)$-curve.
We need to perform one more blow-up in $E_2$ at a point $p_3$ other than $X\cap E_2$ and $E_2\cap E_1$.  Weights and flows in $M_3$ now are
$$
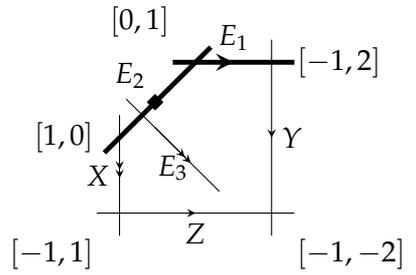
\begin{tikzpicture}
[decoration={markings, 
    mark= at position 0.5 with {\arrow{stealth }}}
]

\draw(0,1.3)--node[left]{$X$}(0,-0.3)node[sloped,pos=0.5,allow upside down]{\arrowIIn};

\draw[postaction={decorate}](2,2.3)--node[right]{$Y$}(2,-0.3);
\draw[postaction={decorate}](-0.3,0)--node[below]{$Z$}(2+0.3,0);
\draw[line width=0.65mm][postaction={decorate}](0.7,2)--node[above]{$E_1$}(2+0.3,2);

\draw[shorten >=-0.3cm,shorten <=-0.3cm](0.3,1.3)--node[below]{$E_3$}(1.1,0.5)node[sloped,pos=0.5,allow upside down]{\arrowIIn};

\draw[line width=0.65mm](0-0.2,0.8)--node[above left]{$E_2$}(1+0.2,2.2)node[
    sloped,
    pos=0.5,
    allow upside down]{\arrowBox};

\node[left=0.2cm] at (0,1){$[1,0]$};
\node[below left=0.2cm] at (0,0){$[-1,1]$};
\node[above left=0.2cm] at (1,2){$[0,1]$};
\node[right=0.2cm] at (2,2){$[-1,2]$};
\node[below right=0.2cm] at (2,0){$[-1,-2]$};
\end{tikzpicture}.
$$
Similar as before, $X\cap E_2$ cannot serve as the other ending fixed point of $E$ because of the weight issue. Blow-up must be performed at $p_4=X\cap Z$ in $M_3$. Weights and flows on $M_4$ are
$$
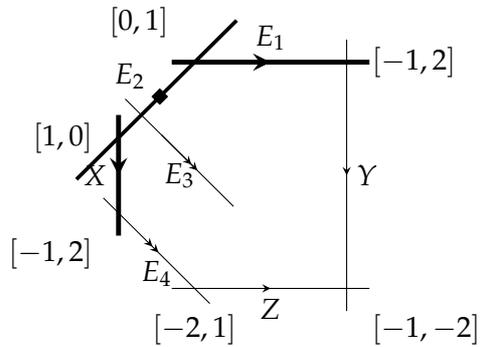
\begin{tikzpicture}
[decoration={markings, 
    mark= at position 0.5 with {\arrow{stealth }}}
]
\draw[line width=0.65mm][postaction={decorate}](0,1.3)--node[left]{$X$}(0,-0.3);
\draw[postaction={decorate}](2+1,2.3)--node[right]{$Y$}(2+1,-0.3-1);
\draw[postaction={decorate}](-0.3+1,0-1)--node[below]{$Z$}(2+0.3+1,0-1);
\draw[line width=0.65mm][postaction={decorate}](0.7,2)--node[above]{$E_1$}(2+0.3+1,2);

\draw[shorten >=-0.5cm,shorten <=-0.5cm][line width=0.5mm](0-0.2,0.8)--node[above left]{$E_2$}(1+0.2,2.2)node[
    sloped,
    pos=0.3,
    allow upside down]{\arrowBox};

\draw(0-0.2,0+0.2)--node[below]{$E_4$}(1+0.2,-1-0.2)node[
    sloped,
    pos=0.5,
    allow upside down]{\arrowIIn};

\draw[shorten >=-0.3cm,shorten <=-0.3cm](0.3,1.3)--node[below]{$E_3$}(1.3,0.3)node[sloped,pos=0.5,allow upside down]{\arrowIIn};

\node[left=0.2cm] at (0,1){$[1,0]$};
\node[below left=0.2cm] at (0,0){$[-1,2]$};
\node[above left=0.2cm] at (1,2){$[0,1]$};
\node[right=0.2cm] at (2+1,2){$[-1,2]$};
\node[below right=0.2cm] at (2+1,-1){$[-1,-2]$};
\node[below=0.2cm] at (1,-1){$[-2,1]$};
\end{tikzpicture}.
$$
Due to the weight constraints (see Corollary \ref{cor:su2weights}), $E_4\cap X$ must be the other ending fixed point, and $E=X\cup E_2\cup E_1$. 
By analyzing all further blow-ups in $M_4$, the following proposition holds.
\begin{proposition}
If $p_2=X\cap E_1$, $\widetilde{M}$ can only be the $M_4$ above, or $M_5=Bl_{Z\cap Y}M_4$, whose configuration is
$$
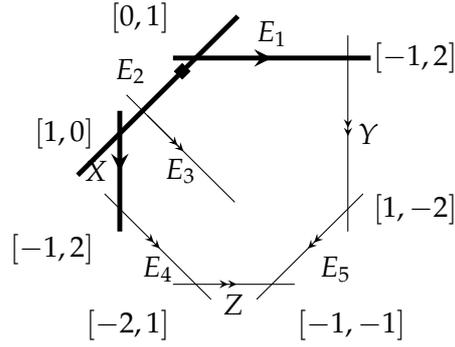
\begin{tikzpicture}
[decoration={markings, 
    mark= at position 0.5 with {\arrow{stealth }}}
]
\draw[line width=0.65mm][postaction={decorate}](0,1.3)--node[left]{$X$}(0,-0.3);

\draw(2+1,2.3)--node[right]{$Y$}(2+1,-0.3-1+1)node[
    sloped,
    pos=0.5,
    allow upside down]{\arrowIIn};

\draw(-0.3+1,0-1)--node[below]{$Z$}(2+0.3,0-1)node[
    sloped,
    pos=0.5,
    allow upside down]{\arrowIIn};

\draw[line width=0.65mm][postaction={decorate}](0.7,2)--node[above]{$E_1$}(2+0.3+1,2);

\draw(3+0.2,0+0.2)--node[below right]{$E_5$}(2-0.2,-1-0.2)node[
    sloped,
    pos=0.5,
    allow upside down]{\arrowIIn};

\draw[shorten >=-0.5cm,shorten <=-0.5cm][line width=0.65mm](0-0.2,0.8)--node[above left]{$E_2$}(1+0.2,2.2)node[
    sloped,
    pos=0.5,
    allow upside down]{\arrowBox};

\draw(0-0.2,0+0.2)--node[below]{$E_4$}(1+0.2,-1-0.2)node[
    sloped,
    pos=0.5,
    allow upside down]{\arrowIIn};

\draw[shorten >=-0.3cm,shorten <=-0.3cm](0.3,1.3)--node[below]{$E_3$}(1.3,0.3)node[sloped,pos=0.3,allow upside down]{\arrowIIn};

\node[left=0.2cm] at (0,1){$[1,0]$};
\node[below left=0.2cm] at (0,0){$[-1,2]$};
\node[above left=0.2cm] at (1,2){$[0,1]$};
\node[right=0.2cm] at (2+1,2){$[-1,2]$};
\node[below right=0.2cm] at (2+1-1,-1){$[-1,-1]$};
\node[right=0.2cm] at (2+1,-1+1){$[1,-2]$};
\node[below left=0.2cm] at (1,-1){$[-2,1]$};
\end{tikzpicture}
$$
\end{proposition}

The first case in the above proposition will be labeled as Case 3A, and the second case will be Case 3B.

\begin{proof}
The only possible blow-up in $M_4$ is at $p_5=Z\cap Y$. If such a blow-up occurs, then no more fixed points can be blown up in $M_5$. Therefore, the proposition is proved.
\end{proof}

\subsubsection{$p_2=X\cap Z$, and there is no blow-up at $X\cap E_1,E_1\cap Y$ in further blow-ups}
\label{subsubsec:3}

Flows on $M_2$ are
$$
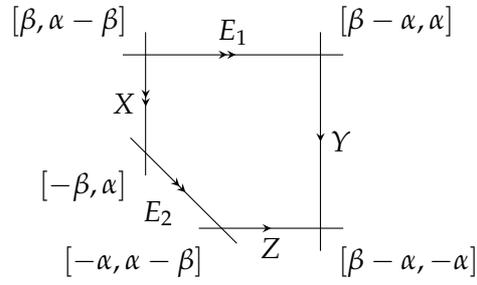
\begin{tikzpicture}
[decoration={markings, 
    mark= at position 0.5 with {\arrow{stealth }}}
]

\draw(-0.2,0.2)--node[below left]{$E_2$}(1.2,-1.2)node[sloped,pos=0.5,allow upside down]{\arrowIIn};

\draw[postaction={decorate}](0.7,-1)--node[below]{$Z$}(2.6,-1);
\draw[postaction={decorate}](2.3,0.6+1)--node[right]{$Y$}(2.3,-1.3);

\draw(-0.3,1.3)--node[above]{$E_1$}(1.6+1,1.3)node[sloped,pos=0.5,allow upside down]{\arrowIIn};

\draw(0,1.6)--node[left]{$X$}(0,-0.3)node[sloped,pos=0.5,allow upside down]{\arrowIIn};

\node[below left=0.1cm] at (0,0){$[-\beta,\alpha]$};
\node[below left=0.1cm] at (1,-1){$[-\alpha,\alpha-\beta]$};
\node[below right=0.1cm] at (2.3,-1){$[\beta-\alpha,-\alpha]$};
\node[above right=0.1cm] at (1.3+1,1.3){$[\beta-\alpha,\alpha]$};
\node[above left=0.1cm] at (0,1.3){$[\beta,\alpha-\beta]$};

\end{tikzpicture}.
$$
The repulsive set $c_-$ must be contained in $E$. However, blow-ups cannot be taken at $X\cap E_2,X\cap E_1,E_1\cap Y$. It follows this case cannot happen.

Section \ref{subsubsec:1}-\ref{subsubsec:3} together conclude the cyclic case \ref{case3}.
\begin{proposition}\label{prop:case31}
In case \ref{case3}, if the orbifold group $\Gamma$ is cyclic, then there are only Case 3A and Case 3B in the following.
\end{proposition}

Note that the $\mathfrak{E}$ action is lifted from $t\curvearrowright[x:y:z]=[t^{2}x:ty:z]$ on $\mathbb{P}^2$.

\begin{minipage}{0.5\textwidth}

\paragraph{\underline{\underline{Case 3A}}}
\label{3Aproblem}

\begin{itemize}
\item The orbifold group is $A_3$.
\item The degree of 
$$\widetilde{M}=Bl_{X\cap Y,X\cap E_1,p_3,X\cap Z}\mathbb{P}^2$$ is 5 and the picard number of $\widehat{M}$ is 2. Here $p_3$ is a point in $E_2$ different from $X\cap E_2,E_1\cap E_2$. 
\item $K_{\widetilde{M}}=-3Z+E_1+2E_2+3E_3-2E_4$. 
\end{itemize}

\end{minipage}%
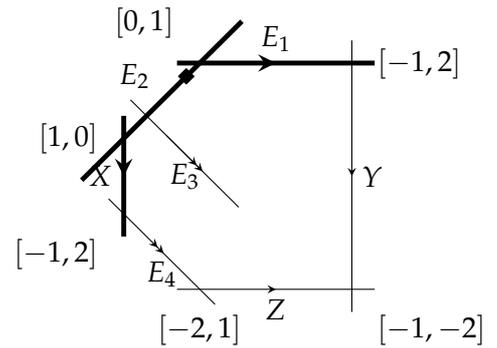
\begin{minipage}{0.4\textwidth}
$$
\begin{tikzpicture}
[decoration={markings, 
    mark= at position 0.5 with {\arrow{stealth }}}
]
\draw[line width=0.65mm][postaction={decorate}](0,1.3)--node[left]{$X$}(0,-0.3);
\draw[postaction={decorate}](2+1,2.3)--node[right]{$Y$}(2+1,-0.3-1);
\draw[postaction={decorate}](-0.3+1,0-1)--node[below]{$Z$}(2+0.3+1,0-1);
\draw[line width=0.65mm][postaction={decorate}](0.7,2)--node[above]{$E_1$}(2+0.3+1,2);

\draw[shorten >=-0.5cm,shorten <=-0.5cm][line width=0.65mm](0-0.2,0.8)--node[above left]{$E_2$}(1+0.2,2.2)node[
    sloped,
    pos=0.5,
    allow upside down]{\arrowBox};

\draw(0-0.2,0+0.2)--node[below]{$E_4$}(1+0.2,-1-0.2)node[
    sloped,
    pos=0.5,
    allow upside down]{\arrowIIn};

\draw[shorten >=-0.3cm,shorten <=-0.3cm](0.3,1.3)--node[below]{$E_3$}(1.3,0.3)node[sloped,pos=0.5,allow upside down]{\arrowIIn};

\node[left=0.2cm] at (0,1){$[1,0]$};
\node[below left=0.2cm] at (0,0){$[-1,2]$};
\node[above left=0.2cm] at (1,2){$[0,1]$};
\node[right=0.2cm] at (2+1,2){$[-1,2]$};
\node[below right=0.2cm] at (2+1,-1){$[-1,-2]$};
\node[below=0.2cm] at (1,-1){$[-2,1]$};

\end{tikzpicture}
$$
\end{minipage}

\begin{minipage}{0.5\textwidth}

\paragraph{\underline{\underline{Case 3B}}}
\label{3B}
\begin{itemize}
    \item The pair $(\widetilde{M},\mathfrak{E})$ is biholomorphic to the pair 
    of Case \hyperref[2D]{2D}.
\end{itemize}

\end{minipage}%
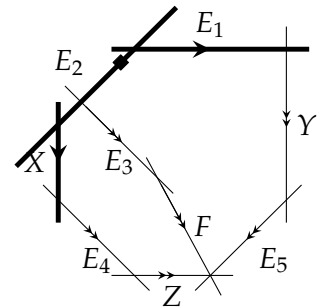
\begin{minipage}{0.4\textwidth}
$$
\begin{tikzpicture}
[decoration={markings, 
    mark= at position 0.5 with {\arrow{stealth }}}
]
\draw[line width=0.65mm][postaction={decorate}](0,1.3)--node[left]{$X$}(0,-0.3);

\draw(2+1,2.3)--node[right]{$Y$}(2+1,-0.3-1+1)node[
    sloped,
    pos=0.5,
    allow upside down]{\arrowIIn};

\draw(-0.3+1,0-1)--node[below]{$Z$}(2+0.3,0-1)node[
    sloped,
    pos=0.5,
    allow upside down]{\arrowIIn};

\draw[line width=0.65mm][postaction={decorate}](0.7,2)--node[above]{$E_1$}(2+0.3+1,2);

\draw(3+0.2,0+0.2)--node[below right]{$E_5$}(2-0.2,-1-0.2)node[
    sloped,
    pos=0.5,
    allow upside down]{\arrowIIn};

\draw[shorten >=-0.5cm,shorten <=-0.5cm][line width=0.65mm](0-0.2,0.8)--node[above left]{$E_2$}(1+0.2,2.2)node[
    sloped,
    pos=0.5,
    allow upside down]{\arrowBox};

\draw(0-0.2,0+0.2)--node[below]{$E_4$}(1+0.2,-1-0.2)node[
    sloped,
    pos=0.5,
    allow upside down]{\arrowIIn};

\draw[shorten >=-0.3cm,shorten <=-0.3cm](0.3,1.3)--node[below]{$E_3$}(1.3,0.3)node[sloped,pos=0.3,allow upside down]{\arrowIIn};

\draw[shorten >=-0.3cm,shorten <=-0.3cm](1.3,0.3)--node[right]{$F$}(2,-1)node[sloped,pos=0.3,allow upside down]{\arrowIIn};

\end{tikzpicture}
$$
\end{minipage}

\subsection{Non-cyclic $\Gamma$}\label{noncycliccase3}
In the non-cyclic case, the orbifold group can only be of $D_n$ or $E_n$ type. An important feature is the central curve in $E$ is a fixed curve, which is also the repulsive set $c_-$.
Lemma \ref{lem:m1} also holds by the same proof in this case and $M_1=Bl_{X\cap Y}\mathbb{P}^2$.
$$
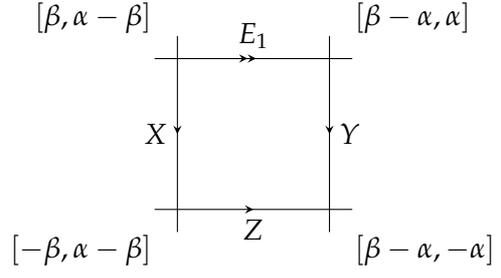
\begin{tikzpicture}
[decoration={markings, 
    mark= at position 0.5 with {\arrow{stealth }}}
]
\draw[postaction={decorate}](0,2.3)--node[left]{$X$}(0,-0.3);
\draw[postaction={decorate}](2,2.3)--node[right]{$Y$}(2,-0.3);
\draw[postaction={decorate}](-0.3,0)--node[below]{$Z$}(2+0.3,0);
\draw(-0.3,2)--node[above]{$E_1$}(2+0.3,2)node[sloped,pos=0.5,allow upside down]{\arrowIIn};

\node[above left=0.2cm] at (0,2){$[\beta,\alpha-\beta]$};
\node[below left=0.2cm] at (0,0){$[-\beta,\alpha-\beta]$};
\node[below right=0.2cm] at (2,0){$[\beta-\alpha,-\alpha]$};
\node[above right=0.2cm] at (2,2){$[\beta-\alpha,\alpha]$};
\end{tikzpicture}.
$$
To create the fixed central curve, blow-up must be taken at $p_2=X\cap E_1$. 
Now $M_2$ is
$$
\begin{tikzpicture}
[decoration={markings, 
    mark= at position 0.5 with {\arrow{stealth }}}
]

\draw(0,1.3)--node[left]{$X$}(0,-0.3)node[sloped,pos=0.65,allow upside down]{\arrowIIn};

\draw[postaction={decorate}](2,2.3)--node[right]{$Y$}(2,-0.3);
\draw[postaction={decorate}](-0.3,0)--node[below]{$Z$}(2+0.3,0);
\draw[line width=0.65mm][postaction={decorate}](0.7,2)--node[above]{$E_1$}(2+0.3,2);
\draw(0-0.2,0.8)--node[above left]{$E_2$}node[below right]{?}(1+0.2,2.2);

\node[left=0.2cm] at (0,1){$[\beta,\alpha-2\beta]$};
\node[below left=0.2cm] at (0,0){$[-\beta,\alpha-\beta]$};
\node[above left=0.2cm] at (1,2){$[2\beta-\alpha,\alpha-\beta]$};
\node[right=0.2cm] at (2,2){$[\beta-\alpha,\alpha]$};
\node[below right=0.2cm] at (2,0){$[\beta-\alpha,-\alpha]$};
\end{tikzpicture}.
$$
The direction of the flow on $E_2$ is determined by the sign of $\alpha-2\beta$, resulting in three possible cases.
\begin{itemize}
\item $\alpha=2\beta$. In this case, $E_2$ is already a fixed curve.
\item $\alpha<2\beta$. To get the fixed curve, a further blow-up should be taken at $p_3=E_2\cap E_1$ in $M_2$. However, this is not possible as $E_1$ is already a $(-2)$-curve.
\item $\alpha>2\beta$. In this case, to get the fixed curve,  a further blow-up should be taken at $p_3=X\cap E_2$ in $M_2$.
\end{itemize}
In the following Section \ref{subsubsec:a=2b} and Section  \ref{subsubsec:a>2b}, we analyze the remaining cases in detail.

\subsubsection{$\boldsymbol{\alpha=2\beta}$}
\label{subsubsec:a=2b}

Weights and flows on $M_2$ are
$$
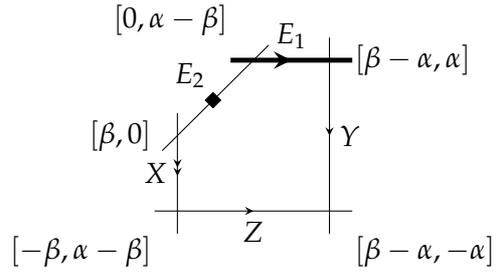
\begin{tikzpicture}
[decoration={markings, 
    mark= at position 0.5 with {\arrow{stealth }}}
]

\draw(0,1.3)--node[left]{$X$}(0,-0.3)node[sloped,pos=0.5,allow upside down]{\arrowIIn};

\draw[postaction={decorate}](2,2.3)--node[right]{$Y$}(2,-0.3);
\draw[postaction={decorate}](-0.3,0)--node[below]{$Z$}(2+0.3,0);
\draw[line width=0.65mm][postaction={decorate}](0.7,2)--node[above]{$E_1$}(2+0.3,2);

\draw(0-0.2,0.8)--node[above left]{$E_2$}(1+0.2,2.2)node[
    sloped,
    pos=0.5,
    allow upside down]{\arrowBox};

\node[left=0.2cm] at (0,1){$[\beta,0]$};
\node[below left=0.2cm] at (0,0){$[-\beta,\alpha-\beta]$};
\node[above left=0.2cm] at (1,2){$[0,\alpha-\beta]$};
\node[right=0.2cm] at (2,2){$[\beta-\alpha,\alpha]$};
\node[below right=0.2cm] at (2,0){$[\beta-\alpha,-\alpha]$};
\end{tikzpicture}.
$$
It is clear that the proper transform of $E_2$ in $\widetilde{M}$ is the central curve in $E$, since $E_2$ is already the fixed curve.  
Now
\begin{itemize}
\item Similar to Lemma \ref{lem:propery},  $Y$ cannot be contained in $E$.
\item $X\cap E_2,E_2\cap E_1$ can not be blown up because of Corollary \ref{cor:blowup}.
\end{itemize}
Therefore, $E_1$ must be one of the three chains of $(-2)$-curves intersecting the central curve in $E$, which is a chain of type $A_1=L(1,2)$. According to Theorem \ref{thm:noncyclicweights}, the weights at the ending fixed point of this chain must be $[p,-q]=[2,-1]$. Hence, 
$$\alpha=p=2,\beta-\alpha=-q=-1,$$ 
which shows
$\alpha=2,\beta=1$.
Another blow-up is necessary in $E_2$ to make it a $(-2)$-curve. As there are three chains of $(-2)$-curves intersecting $E_2$, there also should be a blow-up at $X\cap Z$. Hence the next lemma is clear.

\begin{lemma}
There must be a blow-up at $p_3=X\cap Z$ in $M_2$ and a blow-up at a  point $p_4\in E_2$ other than $X\cap E_2,E_2\cap E_1$.
\end{lemma}
The resulting surface is $M_4$.
$$
\begin{tikzpicture}
[decoration={markings, 
    mark= at position 0.5 with {\arrow{stealth }}}
]

\draw[line width=0.5mm][postaction={decorate}](0,1.3)--node[left]{$X$}(0,-0.3);

\draw[postaction={decorate}](3,2.3)--node[right]{$Y$}(2+1,-0.3-1);
\draw[postaction={decorate}](-0.3+1,0-1)--node[below]{$Z$}(2+0.3+1,0-1);
\draw[postaction={decorate}][line width=0.65mm](0.7,2)--node[above]{$E_1$}(2+0.3+1,2);

\draw[line width=0.65mm][shorten >=-0.4cm,shorten <=-0.4cm](0-0.2,0.8)--node[above left]{$E_2$}(1+0.2,2.2)node[
    sloped,
    pos=0.5,
    allow upside down]{\arrowBox};

\draw(0-0.2,0+0.2)--node[below left]{$E_3$}(1+0.2,-1-0.2)node[sloped,pos=0.5,allow upside down]{\arrowIIn};

\draw(0.5-0.1,1.5+0.1)--node[below left]{$E_4$}(1.5+0.2,0.5-0.2)node[sloped,pos=0.5,allow upside down]{\arrowIIn};

\draw[postaction={decorate}] (1.5,0.5) to[out=-10,in=180] node[below]{$F$}(3,-1);

\node[above right] at (1.5,0.5){$[-1,1]$};
\node[above left=0.2cm] at (1,2){$[0,1]$};
\node[above right=0.2cm] at (3,2){$[-1,2]$};
\node[below right=0.2cm] at (3,-1){$[-1,-2]$};
\node[below left=0.2cm] at (1,-1){$[-2,1]$};
\node[left=0.2cm] at (0,0){$[-1,2]$};
\node[left=0.2cm] at (0,1){$[1,0]$};

\end{tikzpicture}.
$$
Here, $F$ is the proper transform of the curve from $p_4$ to $Y\cap Z$. To guarantee the presence of three chains of $(-2)$-curves intersecting $E_2$, an extra blow-up at the intersection point $p_5=E_4\cap F$ is required. We have the following proposition.

\begin{proposition}
The surface $\widetilde{M}$ must be a blow-up of $M_5=Bl_{E_4\cap F}M_4$.
\end{proposition}

Flows on $M_5$ are given by
$$
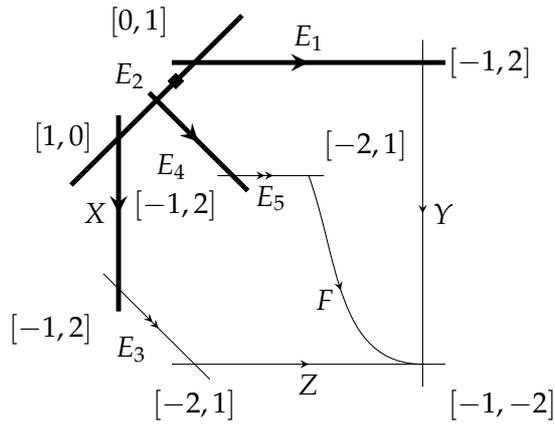
\begin{tikzpicture}
[decoration={markings, 
    mark= at position 0.5 with {\arrow{stealth }}}
]
\draw[line width=0.65mm][postaction={decorate}](0,1.3)--node[left]{$X$}(0,-0.3-1);
\draw[postaction={decorate}](2+1+1,2.3)--node[right]{$Y$}(2+1+1,-0.3-1-1);
\draw[postaction={decorate}](-0.3+1,0-1-1)--node[below]{$Z$}(2+0.3+1+1,0-1-1);
\draw[line width=0.65mm][postaction={decorate}](0.7,2)--node[above]{$E_1$}(2+0.3+1+1,2);

\draw[line width=0.65mm][shorten >=-0.6cm,shorten <=-0.6cm](0-0.2,0.8)--node[above left]{$E_2$}(1+0.2,2.2)node[
    sloped,
    pos=0.4,
    allow upside down]{\arrowBox};

\draw(0-0.2,0+0.2-1)--node[below left]{$E_3$}(1+0.2,-1-0.2-1)node[sloped,pos=0.5,allow upside down]{\arrowIIn};

\draw[line width=0.65mm][postaction={decorate}](0.5-0.1,1.5+0.1)--node[below left]{$E_4$}(1.5+0.2,0.5-0.2);

\draw(1.5-0.2,0.5)--node[below]{$E_5$}(1.5+1+0.2,0.5)node[sloped,pos=0.5,allow upside down]{\arrowIIn};

\draw[postaction={decorate}] (2.5,0.5) to[out=-70,in=180] node[left]{$F$}(3+1,-1-1);

\node[left=0.2cm] at (0,1){$[1,0]$};
\node[below left=0.2cm] at (0,0-1){$[-1,2]$};
\node[above left=0.2cm] at (1,2){$[0,1]$};
\node[right=0.2cm] at (2+1+1,2){$[-1,2]$};
\node[below right=0.2cm] at (2+1+1,-1-1){$[-1,-2]$};
\node[below=0.2cm] at (1,-1-1){$[-2,1]$};
\node[below left=0.06cm] at (1.5,0.5){$[-1,2]$};
\node[above right=0.06cm] at (1.5+1,0.5) {$[-2,1]$};
\end{tikzpicture}.
$$
Note that here $F$ intersects $Z$ non-transversely, and a simple calculation shows $F^2=Z^2=0$. Blow up $Y\cap Z$ lowers the self intersection of $F$ by 1.
By analyzing all further possible blow-ups case-by-case, we can classify all $(\widetilde{M},\mathfrak{E})$ in the case that $\alpha=2\beta$.
\begin{proposition}\label{prop:case32}
In case \ref{case3}, when $\alpha=2\beta$ and the orbifold group is non-cyclic, there are five possible $(\widetilde{M},\mathfrak{E})$.
\end{proposition}

The structures of these $(\widetilde{M},\mathfrak{E})$ are given below. Note that the $\mathfrak{E}$ action is lifted from  $t\curvearrowright[x:y:z]=[t^{2}x:ty:z]$ on $\mathbb{P}^2$.

\begin{minipage}{0.47\textwidth}
\paragraph{\underline{\underline{Case 3C}}}
\label{3C}

\begin{itemize} 
\item The orbifold group is $D_4$. 
\item The degree of 
$$\widetilde{M}=Bl_{X\cap Y,X\cap E_1,X\cap Z,p_4,E_4\cap F}\mathbb{P}^2$$ is 4 and the picard number of $\widehat{M}$ is 2. Here $p_4$ is a point in $E_2$ different from $E_1\cap E_2,X\cap E_2$.
\item $K_{\widetilde{M}}=-3Z+E_1+2E_2-2E_3+3E_4+4E_5$.
\end{itemize}
\end{minipage}%
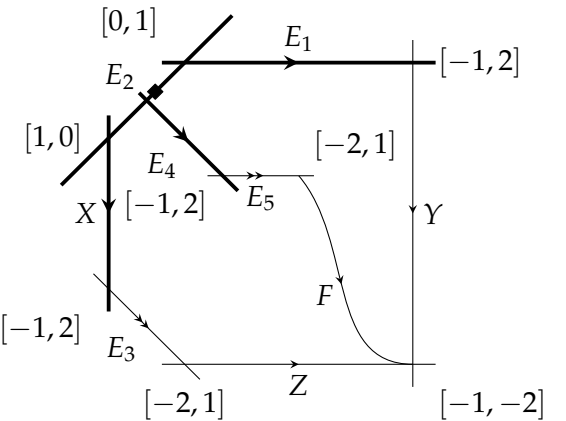
\begin{minipage}{0.5\textwidth}
$$
\scalebox{0.9}{
\begin{tikzpicture}
[decoration={markings, 
    mark= at position 0.5 with {\arrow{stealth }}}
]
\draw[line width=0.5mm][postaction={decorate}](0,1.3)--node[left]{$X$}(0,-0.3-1);
\draw[postaction={decorate}](2+1+1,2.3)--node[right]{$Y$}(2+1+1,-0.3-1-1);
\draw[postaction={decorate}](-0.3+1,0-1-1)--node[below]{$Z$}(2+0.3+1+1,0-1-1);
\draw[line width=0.5mm][postaction={decorate}](0.7,2)--node[above]{$E_1$}(2+0.3+1+1,2);

\draw[line width=0.5mm][shorten >=-0.6cm,shorten <=-0.6cm](0-0.2,0.8)--node[above left]{$E_2$}(1+0.2,2.2)node[
    sloped,
    pos=0.3,
    allow upside down]{\arrowBox};

\draw(0-0.2,0+0.2-1)--node[below left]{$E_3$}(1+0.2,-1-0.2-1)node[sloped,pos=0.5,allow upside down]{\arrowIIn};

\draw[line width=0.5mm][postaction={decorate}](0.5-0.1,1.5+0.1)--node[below left]{$E_4$}(1.5+0.2,0.5-0.2);

\draw(1.5-0.2,0.5)--node[below]{$E_5$}(1.5+1+0.2,0.5)node[sloped,pos=0.5,allow upside down]{\arrowIIn};

\draw[postaction={decorate}](2.5,0.5) to[out=-50,in=180] node[left]{$F$}(3+1,-1-1);

\node[left=0.2cm] at (0,1){$[1,0]$};
\node[below left=0.2cm] at (0,0-1){$[-1,2]$};
\node[above left=0.2cm] at (1,2){$[0,1]$};
\node[right=0.2cm] at (2+1+1,2){$[-1,2]$};
\node[below right=0.2cm] at (2+1+1,-1-1){$[-1,-2]$};
\node[below=0.2cm] at (1,-1-1){$[-2,1]$};
\node[below left=0.06cm] at (1.5,0.5){$[-1,2]$};
\node[above right=0.06cm] at (1.5+1,0.5) {$[-2,1]$};

\end{tikzpicture}
}
$$
\end{minipage}

\begin{minipage}{0.47\textwidth}
\paragraph{\underline{\underline{Case 3D}}}
\label{3D}

\begin{itemize}
    \item The pair $(\widetilde{M},\mathfrak{E})$ is biholomorphic to the pair of Case \hyperref[2E]{2E}.
\end{itemize}

\end{minipage}%
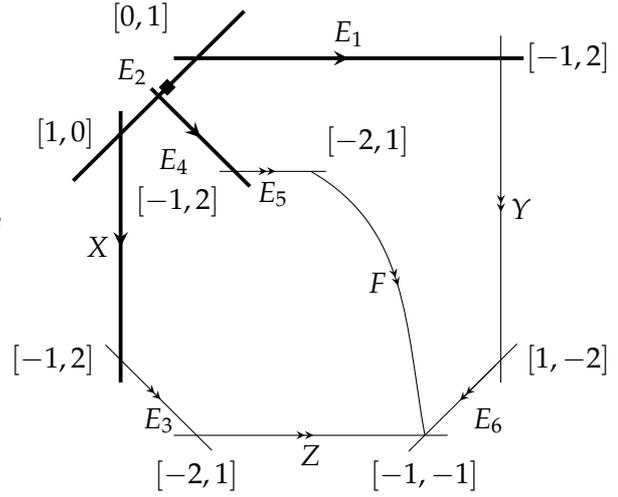
\begin{minipage}{0.5\textwidth}
$$
\scalebox{0.9}{
\begin{tikzpicture}
[decoration={markings, 
    mark= at position 0.5 with {\arrow{stealth }}}
]
\draw[line width=0.5mm][postaction={decorate}](0,1.3)--node[left]{$X$}(0,-0.3-1-1);

\draw(2+1+1+1,2.3)--node[right]{$Y$}(2+1+1+1,-0.3-1-1)node[sloped,pos=0.5,allow upside down]{\arrowIIn};

\draw(-0.3+1,0-1-1-1)--node[below]{$Z$}(2+0.3+1+1,0-1-1-1)node[sloped,pos=0.5,allow upside down]{\arrowIIn};

\draw[line width=0.5mm][postaction={decorate}](0.7,2)--node[above]{$E_1$}(2+0.3+1+1+1,2);
\draw[line width=0.5mm][shorten >=-0.6cm,shorten <=-0.6cm](0-0.2,0.8)--node[above left]{$E_2$}(1+0.2,2.2)node[
    sloped,
    pos=0.3,
    allow upside down]{\arrowBox};

\draw(0-0.2,0+0.2-1-1)--node[below]{$E_3$}(1+0.2,-1-0.2-1-1)node[sloped,pos=0.5,allow upside down]{\arrowIIn};

\draw[line width=0.5mm][postaction={decorate}](0.5-0.1,1.5+0.1)--node[below left]{$E_4$}(1.5+0.2,0.5-0.2);

\draw(1.5-0.2,0.5)--node[below]{$E_5$}(1.5+1+0.2,0.5)node[sloped,pos=0.5,allow upside down]{\arrowIIn};

\draw[shorten >=-0.3cm,shorten <=-0.3cm](2+1+1+1,-1-1)--node[below right]{$E_6$}(2+1+1,0-1-1-1)node[sloped,pos=0.3,allow upside down]{\arrowIIn};

\draw(2.5,0.5) to[out=-30,in=100] node[left]{$F$}node[sloped,pos=0.5,allow upside down]{\arrowIIn}(3+1,-1-1-1);

\node[left=0.2cm] at (0,1){$[1,0]$};
\node[left=0.2cm] at (0,0-1-1){$[-1,2]$};
\node[above left=0.2cm] at (1,2){$[0,1]$};
\node[right=0.2cm] at (2+1+1+1,2){$[-1,2]$};
\node[right=0.2cm] at (2+1+1+1,-1-1){$[1,-2]$};
\node[below=0.2cm] at (1,-1-1-1){$[-2,1]$};
\node[below left=0.06cm] at (1.5,0.5){$[-1,2]$};
\node[above right=0.06cm] at (1.5+1,0.5) {$[-2,1]$};
\node[below=0.2cm] at (2+1+1,-1-1-1){$[-1,-1]$};

\end{tikzpicture}
}
$$
\end{minipage}

\begin{minipage}{0.47\textwidth}
\paragraph{\underline{\underline{Case 3E}}}
\label{3E}

\begin{itemize} 
\item The orbifold group is $D_5$.
\item The degree of $$\widetilde{M}=Bl_{X\cap Y,X\cap E_1,X\cap Z,p_4,E_4\cap F,E_5\cap F}\mathbb{P}^2$$ is 3 and the picard number of $\widehat{M}$ is 2. Here $p_4$ is a point in $E_2$ different from $X\cap E_2,E_1\cap E_2$.
\item $K_{\widetilde{M}}=-3Z+E_1+2E_2-2E_3+3E_4+4E_5+5E_6$.
\end{itemize} 
\end{minipage}%
\begin{minipage}{0.5\textwidth}
$$
\scalebox{0.9}{
\begin{tikzpicture}
[decoration={markings, 
    mark= at position 0.5 with {\arrow{stealth }}}
]
\draw[line width=0.65mm][postaction={decorate}](0,1.3)--node[left]{$X$}(0,-0.3-1-1);
\draw[postaction={decorate}](2+1+1+1,2.3)--node[right]{$Y$}(2+1+1+1,-0.3-1-1-1);
\draw[postaction={decorate}](-0.3+1,0-1-1-1)--node[below]{$Z$}(2+0.3+1+1+1,0-1-1-1);
\draw[line width=0.65mm][postaction={decorate}](0.7,2)--node[above]{$E_1$}(2+0.3+1+1+1,2);
\draw[line width=0.65mm][shorten >=-0.6cm,shorten <=-0.6cm](0-0.2,0.8)--node[above left]{$E_2$}(1+0.2,2.2)node[
    sloped,
    pos=0.3,
    allow upside down]{\arrowBox};

\draw(0-0.2,0+0.2-1-1)--node[below]{$E_3$}(1+0.2,-1-0.2-1-1)node[sloped,pos=0.5,allow upside down]{\arrowIIn};

\draw[line width=0.65mm][postaction={decorate}](0.5-0.1,1.5+0.1)--node[below left]{$E_4$}(1.5+0.2,0.5-0.2);
\draw[line width=0.65mm][postaction={decorate}](1.5-0.2,0.5)--node[below]{$E_5$}(1.5+1+0.2,0.5);

\draw(1.5+1,0.5+0.2)--node[right]{$E_6$}(1.5+1,0.5-1-0.2)node[sloped,pos=0.5,allow upside down]{\arrowIIn};


\draw(2.5,0.5-1) to[out=-10,in=180] node[below left]{$F$}node[sloped,pos=0.5,allow upside down]{\arrowIIn}(3+1+1,-1-1-1);

\node[left=0.2cm] at (0,1){$[1,0]$};
\node[left=0.2cm] at (0,0-1-1){$[-1,2]$};
\node[above left=0.2cm] at (1,2){$[0,1]$};
\node[right=0.2cm] at (2+1+1+1,2){$[-1,2]$};
\node[below=0.2cm] at (1,-1-1-1){$[-2,1]$};
\node[below left=0.06cm] at (1.5,0.5){$[-1,2]$};
\node[above right=0.06cm] at (1.5+1,0.5) {$[-2,3]$};
\node[below left=0.06cm] at (1.5+1,0.5-1) {$[-3,1]$};
\node[below=0.2cm] at (2+1+1+1,-1-1-1){$[-1,-2]$};

\end{tikzpicture}
}
$$
\end{minipage}

\begin{minipage}{0.47\textwidth}

\paragraph{\underline{\underline{Case 3F}}}
\label{3F}

\begin{itemize}
    \item The pair $(\widetilde{M},\mathfrak{E})$ is biholomorphic to the pair of Case \hyperref[3E]{3E}.
\end{itemize}

\end{minipage}%
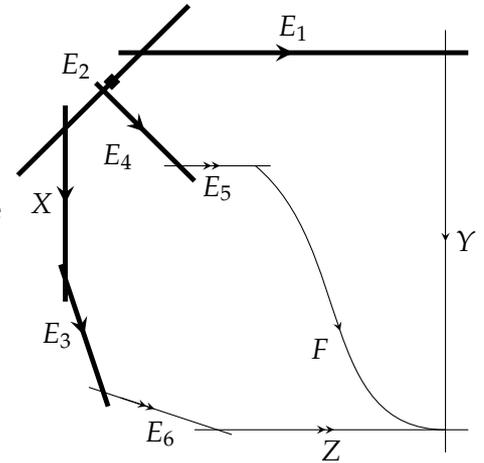
\begin{minipage}{0.5\textwidth}
$$
\scalebox{0.9}{
\begin{tikzpicture}
[decoration={markings, 
    mark= at position 0.5 with {\arrow{stealth }}}
]
\draw[line width=0.65mm][postaction={decorate}](0,1.3)--node[left]{$X$}(0,-0.3-1);
\draw[postaction={decorate}](2+1+1+1,2.3)--node[right]{$Y$}(2+1+1+1,-0.3-1-1-1);

\draw(-0.3+1+1,0-1-1-1)--node[below]{$Z$}(2+0.3+1+1+1,0-1-1-1)node[sloped,pos=0.5,allow upside down]{\arrowIIn};

\draw[line width=0.65mm][postaction={decorate}](0.7,2)--node[above]{$E_1$}(2+0.3+1+1+1,2);
\draw[line width=0.65mm][shorten >=-0.6cm,shorten <=-0.6cm](0-0.2,0.8)--node[above left]{$E_2$}(1+0.2,2.2)node[
    sloped,
    pos=0.3,
    allow upside down]{\arrowBox};

\draw[line width=0.65mm][shorten >=-0.2cm,shorten <=-0.2cm][postaction={decorate}](0,0-1)--node[left]{$E_3$}(1-0.5,-2-0.5);

\draw[shorten >=-0.2cm,shorten <=-0.2cm](0.5,-2.5)--node[below]{$E_6$}(2,-3)node[sloped,pos=0.3,allow upside down]{\arrowIIn};

\draw[line width=0.65mm][postaction={decorate}](0.5-0.1,1.5+0.1)--node[below left]{$E_4$}(1.5+0.2,0.5-0.2);

\draw(1.5-0.2,0.5)--node[below]{$E_5$}(1.5+1+0.2,0.5)node[sloped,pos=0.5,allow upside down]{\arrowIIn};


\draw(2.5,0.5) to[out=-40,in=180] node[below left]{$F$}node[sloped,pos=0.5,allow upside down]{\arrowIn}(3+1+1,-1-1-1);

\end{tikzpicture}
}
$$
\end{minipage}

\begin{minipage}{0.47\textwidth}

\paragraph{\underline{\underline{Case 3G}}}
\label{3G}

\begin{itemize} 
\item The orbifold group is $E_6$.
\item The degree of $$\widetilde{M}=Bl_{X\cap Y,X\cap E_1,X\cap Z,p_4,E_4\cap F,Z\cap E_3,E_5\cap F}\mathbb{P}^2$$ is 2 and the picard number of $\widehat{M}$ is 2. Here $p_4$ is a point in $E_2$ different from $X\cap E_2,E_1\cap E_2$. 
\item $K_{\widetilde{M}}=-3Z+E_1+2E_2-2E_3+3E_4+4E_5+5E_7-4E_6$.
\end{itemize}

\end{minipage}%
\begin{minipage}{0.5\textwidth}
$$
\scalebox{0.9}{
\begin{tikzpicture}
[decoration={markings, 
    mark= at position 0.5 with {\arrow{stealth }}}
]
\draw[line width=0.65mm][postaction={decorate}](0,1.3)--node[left]{$X$}(0,-0.3-1);
\draw[postaction={decorate}](2+1+1+1,2.3)--node[right]{$Y$}(2+1+1+1,-0.3-1-1-1);

\draw(-0.3+1+1,0-1-1-1)--node[below]{$Z$}(2+0.3+1+1+1,0-1-1-1)node[sloped,pos=0.5,allow upside down]{\arrowIIn};

\draw[line width=0.65mm][postaction={decorate}](0.7,2)--node[above]{$E_1$}(2+0.3+1+1+1,2);
\draw[line width=0.65mm][shorten >=-0.6cm,shorten <=-0.6cm](0-0.2,0.8)--node[above left]{$E_2$}(1+0.2,2.2)node[
    sloped,
    pos=0.3,
    allow upside down]{\arrowBox};

\draw[line width=0.5mm][shorten >=-0.2cm,shorten <=-0.2cm][postaction={decorate}](0,0-1)--node[left]{$E_3$}(1-0.5,-2-0.5);

\draw[shorten >=-0.2cm,shorten <=-0.2cm](0.5,-2.5)--node[below]{$E_6$}(2,-3)node[sloped,pos=0.3,allow upside down]{\arrowIIn};

\draw[line width=0.5mm][postaction={decorate}](0.5-0.1,1.5+0.1)--node[below left]{$E_4$}(1.5+0.2,0.5-0.2);

\draw[line width=0.5mm][postaction={decorate}](1.5-0.2,0.5)--node[below]{$E_5$}(1.5+1+0.2,0.5);


\draw[shorten >=-0.2cm,shorten <=-0.2cm](2.5,0.5)--node[right]{$E_7$}(2.5,0.5-1)node[sloped,pos=0.3,allow upside down]{\arrowIIn};

\draw(2.5,0.5-1) to[out=10,in=180] node[below left]{$F$}node[sloped,pos=0.5,allow upside down]{\arrowIIn}(3+1+1,-1-1-1);

\node[left=0.2cm] at (0,1){$[1,0]$};
\node[left=0.2cm] at (0,0-1){$[-1,2]$};
\node[above left=0.2cm] at (1,2){$[0,1]$};
\node[right=0.2cm] at (2+1+1+1,2){$[-1,2]$};
\node[below=0.2cm] at (1-1,-1-1-0.5){$[-2,3]$};
\node[below left=0.06cm] at (1.5,0.5){$[-1,2]$};
\node[above right=0.06cm] at (1.5+1,0.5) {$[-2,3]$};
\node[below left=0.06cm] at (1.5+1,0.5-1) {$[-3,1]$};
\node[below=0.2cm] at (2+1+1+1,-1-1-1){$[-1,-2]$};
\node[below=0.2cm] at (2,-1-1-1){$[-3,1]$};

\end{tikzpicture}
}
$$
\end{minipage}

\subsubsection{$\boldsymbol{\alpha>2\beta}$}
\label{subsubsec:a>2b}

Weights and flows on $M_2$ are
$$
\begin{tikzpicture}
[decoration={markings, 
    mark= at position 0.5 with {\arrow{stealth }}}
]

\draw(0,1.3)--node[left]{$X$}(0,-0.3)node[sloped,pos=0.5,allow upside down]{\arrowIIn};

\draw[postaction={decorate}](2,2.3)--node[right]{$Y$}(2,-0.3);
\draw[postaction={decorate}](-0.3,0)--node[below]{$Z$}(2+0.3,0);
\draw[line width=0.65mm][postaction={decorate}](0.7,2)--node[above]{$E_1$}(2+0.3,2);

\draw(0-0.2,0.8)--node[above left]{$E_2$}(1+0.2,2.2)node[sloped,pos=0.5,allow upside down]{\arrowIIn};

\node[left=0.2cm] at (0,1){$[\beta,\alpha-2\beta]$};
\node[below left=0.2cm] at (0,0){$[-\beta,\alpha-\beta]$};
\node[above left=0.2cm] at (1,2){$[2\beta-\alpha,\alpha-\beta]$};
\node[right=0.2cm] at (2,2){$[\beta-\alpha,\alpha]$};
\node[below right=0.2cm] at (2,0){$[\beta-\alpha,-\alpha]$};
\end{tikzpicture}.
$$
To get the fixed curve, blow-up should be taken at $p_3=X\cap E_2$. Weights and flows on  $M_3$ are
$$
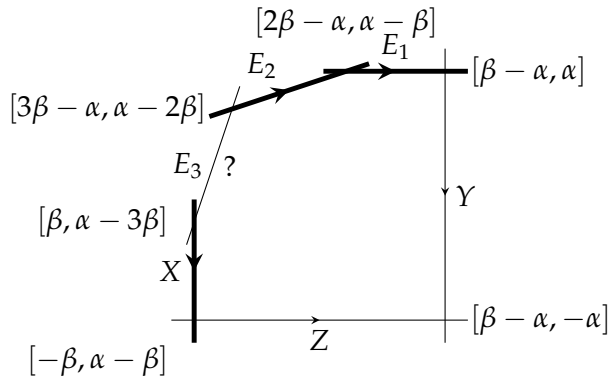
\begin{tikzpicture}
[decoration={markings, 
    mark= at position 0.5 with {\arrow{stealth }}}
]
\draw[line width=0.65mm][postaction={decorate}](0,1.3)--node[left]{$X$}(0,-0.6);
\draw[postaction={decorate}](2.3+1,2.3+1)--node[right]{$Y$}(2.3+1,-0.6);
\draw[postaction={decorate}](-0.3,-0.3)--node[below]{$Z$}(3.6,-0.3);
\draw[line width=0.65mm][postaction={decorate}](0.7+1,2+1)--node[above]{$E_1$}(2.3+0.3+1,2+1);
\draw[line width=0.65mm][postaction={decorate}](0.5-0.3,2.5-0.1)--node[above left]{$E_2$}(1+1+0.3,2+1+0.1);
\draw(0-0.1,1-0.3)--node[left]{$E_3$}node[right]{?}(0.5+0.1,2.5+0.3);

\node[left=0.2cm] at (0,1){$[\beta,\alpha-3\beta]$};
\node[below left=0.2cm] at (0,-0.3){$[-\beta,\alpha-\beta]$};
\node[above=0.3cm] at (1+1,2+1){$[2\beta-\alpha,\alpha-\beta]$};
\node[right=0.2cm] at (2.3+1,-0.3){$[\beta-\alpha,-\alpha]$};
\node[right=0.2cm] at (2.3+1,3){$[\beta-\alpha,\alpha]$};
\node[left=0.2cm] at (0.5,2.5){$[3\beta-\alpha,\alpha-2\beta]$};

\end{tikzpicture}.
$$
Now $E_3$ has to be a fixed curve to ensure $E$ is connected.
This means $\alpha=3\beta$ and $E_3$ is the central curve in $E$. Similar proof of Lemma \ref{lem:propery} shows $Y$ is not contained in $E$.
Thus $E_1\cup E_2$ is one of the three chains of $(-2)$-curves in $E$, and this chain is of type $A_2$. The ending fixed point is $E_1\cap Y$, therefore it has weights $[p,q]=[2,3]$, which gives
$\alpha=3,\alpha-\beta=2$.
This shows
$\alpha=3,\beta=1$.
There must be one more blow-up at a point $p_4$ different from $X\cap E_3, E_3\cap E_2$ in $E_3$ to make it a $(-2)$-curve. Weights and flows on $M_4$ are
$$
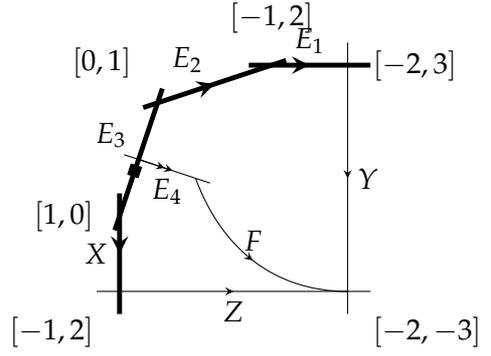
\begin{tikzpicture}
[decoration={markings, 
    mark= at position 0.5 with {\arrow{stealth }}}
]
\draw[line width=0.65mm][postaction={decorate}](0,1.3-1)--node[left]{$X$}(0,-0.3-1);
\draw[postaction={decorate}](2+1,2.3)--node[right]{$Y$}(2+1,-0.3-1);
\draw[postaction={decorate}](-0.3,0-1)--node[below]{$Z$}(2+0.3+1,0-1);
\draw[line width=0.65mm][postaction={decorate}](0.7+1,2)--node[above]{$E_1$}(2+0.3+1,2);
\draw[line width=0.65mm][shorten >=-0.2cm,shorten <=-0.2cm][postaction={decorate}](1-0.5,1+0.5)--node[above left]{$E_2$}(2,2);
\draw[line width=0.65mm][shorten >=-0.2cm,shorten <=-0.2cm](0,0)--node[above left]{$E_3$}(0.5,1.5)node[
    sloped,
    pos=0.3,
    allow upside down]{\arrowBox};

\draw[shorten >=-0.2cm,shorten <=-0.2cm](0.25,0.75)--node[below]{$E_4$}(1,0.5)node[sloped,pos=0.3,allow upside down]{\arrowIIn};

\draw[postaction={decorate}] (1,0.5)  to[out=-70,in=180] node[above]{$F$}(3,-1);

\node[left=0.2cm] at (0,1-1){$[1,0]$};
\node[below left=0.2cm] at (0,-1){$[-1,2]$};
\node[above=0.3cm] at (1+1,2){$[-1,2]$};
\node[right=0.2cm] at (2+1,2){$[-2,3]$};
\node[below right=0.2cm] at (3,-1){$[-2,-3]$};
\node[above left=0.2cm] at (0.5,1.5){$[0,1]$};
\node at (1.8,0.5) {$[-1,1]$};
\end{tikzpicture}.
$$
In $M_4$, there is only one chain of $(-2)$-curves. Therefore, we need to perform one more blow-up at the intersection $p_5=E_4\cap F$ to make $E_4$ a $(-2)$-curve. Weights and flows on $M_5$ are
$$
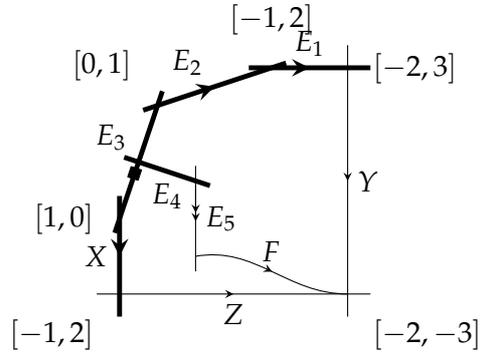
\begin{tikzpicture}
[decoration={markings, 
    mark= at position 0.5 with {\arrow{stealth }}}
]
\draw[line width=0.65mm][postaction={decorate}](0,1.3-1)--node[left]{$X$}(0,-0.3-2);
\draw[postaction={decorate}](2+2,2.3)--node[right]{$Y$}(2+2,-0.3-2);
\draw[postaction={decorate}](-0.3,0-2)--node[below]{$Z$}(2+0.3+2,0-2);
\draw[line width=0.65mm][postaction={decorate}](0.7+1,2)--node[above]{$E_1$}(2+0.3+2,2);
\draw[line width=0.65mm][shorten >=-0.2cm,shorten <=-0.2cm][postaction={decorate}](1-0.5,1+0.5)--node[above left]{$E_2$}(2,2);
\draw[line width=0.65mm][shorten >=-0.2cm,shorten <=-0.2cm](0,0)--node[above left]{$E_3$}(0.5,1.5)node[
    sloped,
    pos=0.3,
    allow upside down]{\arrowBox};

\draw[line width=0.65mm][shorten >=-0.2cm,shorten <=-0.2cm](0.25,0.75)--node[below]{$E_4$}(1.4,0.4);

\draw[shorten >=-0.2cm,shorten <=-0.2cm](1.3,0.5)--node[right]{$E_5$}(1.3,0.5-1)node[sloped,pos=0.3,allow upside down]{\arrowIIn};

\draw[postaction={decorate}] (1,0.5-1)  to[out=10,in=180] node[above]{$F$}(4,-2);

\node[left=0.2cm] at (0,1-1){$[1,0]$};
\node[below left=0.2cm] at (0,-2){$[-1,2]$};
\node[above=0.3cm] at (1+1,1.8){$[-1,2]$};
\node[right=0.2cm] at (4,2.3){$[-2,3]$};
\node[below right=0.2cm] at (4,-1.8){$[-2,-3]$};
\node[above left=0.2cm] at (0.5,1.5){$[0,1]$};
\node at (2.3,0.4) {$[-1,2]$};
\node at (1,-1) {$[-2,1]$};
\end{tikzpicture}.
$$

It is clear that the following proposition holds.
\begin{proposition}
In case \ref{case3}, when $\alpha>2\beta$ and the orbifold group is non-cyclic, the pair $(\widetilde{M},\mathfrak{E})$ must be a blow-up of $M_5$ with the $\mathbb{C}^*$-action described above. 
\end{proposition}

Note that the intersection of $F$ and $Z$ is non-transverse. It is direct to check that $F^2=4$ in $M_5$ and blowing up $Y\cap Z$ will decrease the intersection number of $F$ by $4$. By examining all further blow-ups, we have 
\begin{proposition}
In case \ref{case3}, when $\alpha>2\beta$ and the orbifold group is non-cyclic, there are six possible $(\widetilde{M},\mathfrak{E})$.
\end{proposition}

In the following we describe the six $(\widetilde{M},\mathfrak{E})$ in details. Note that the $\mathfrak{E}$ action is lifted from the action $t\curvearrowright[x:y:z]=[t^3x:ty:z]$ on $\mathbb{P}^2$.

\begin{minipage}{0.47\textwidth}

\paragraph{\underline{\underline{Case 3H}}}
\label{3H}

\begin{itemize} 
\item The orbifold group is $D_5$.
\item The degree of 
         $$\widetilde{M}=Bl_{X\cap Y,X\cap E_1, X\cap E_2,p_4,E_4\cap F}\mathbb{P}^2$$
      is 4 and the picard number of $\widehat{M}$ is 1. Here $p_4$ is a point in $E_2$ different from $X\cap E_3,E_2\cap E_3$. 
\item $K_{\widetilde{M}}=-3Z+E_1+2E_2+3E_3+4E_4+5E_5$.
\end{itemize}
\end{minipage}%
\begin{minipage}{0.5\textwidth}
$$
\scalebox{0.9}{
\begin{tikzpicture}
[decoration={markings, 
    mark= at position 0.5 with {\arrow{stealth }}}
]
\draw[line width=0.65mm][postaction={decorate}](0,1.3-1)--node[left]{$X$}(0,-0.3-1-2);
\draw[postaction={decorate}](2+1+2,2.3)--node[right]{$Y$}(2+1+2,-0.3-1-2);
\draw[postaction={decorate}](-0.3,0-1-2)--node[below]{$Z$}(2+0.3+1+2,0-1-2);
\draw[line width=0.65mm][postaction={decorate}](0.7+1,2)--node[above]{$E_1$}(2+0.3+1+2,2);
\draw[line width=0.65mm][shorten >=-0.2cm,shorten <=-0.2cm][postaction={decorate}](1-0.5,1+0.5)--node[above left]{$E_2$}(2,2);
\draw[line width=0.65mm][shorten >=-0.6cm,shorten <=-0.6cm](0,0)--node[above left]{$E_3$}(0.5,1.5)node[
    sloped,
    pos=0.3,
    allow upside down]{\arrowBox};

\draw[line width=0.65mm][postaction={decorate}][shorten >=-0.2cm,shorten <=-0.2cm](0.25,0.75)--node[below]{$E_4$}(1,0.5);

\draw[shorten >=-0.2cm,shorten <=-0.2cm](1,0.5)--node[right]{$E_5$}(1,0.5-1)node[sloped,pos=0.3,allow upside down]{\arrowIIn};

\draw[postaction={decorate}](1,0.5-1) to[out=10,in=180] node[below]{$F$}(3+2,-1-2);

\node[left=0.2cm] at (0,1-1){$[1,0]$};
\node[below left=0.2cm] at (0,-3){$[-1,2]$};
\node[above=0.3cm] at (1+1,2){$[-1,2]$};
\node[right=0.2cm] at (5,2.2){$[-2,3]$};
\node[below right=0.2cm] at (5,-3){$[-2,-3]$};
\node[above left=0.2cm] at (0.5,1.5){$[0,1]$};
\node at (1.7,1) {$[-1,2]$};
\node at (1.2,-1) {$[-2,1]$};

\end{tikzpicture}
}
$$
\end{minipage}

\begin{minipage}{0.47\textwidth}

\paragraph{\underline{\underline{Case 3I}}}

\begin{itemize}
    \item The pair $(\widetilde{M},\mathfrak{E})$ is biholomorphic to the pair of
    Case \hyperref[3E]{3E}.
\end{itemize}

\end{minipage}%
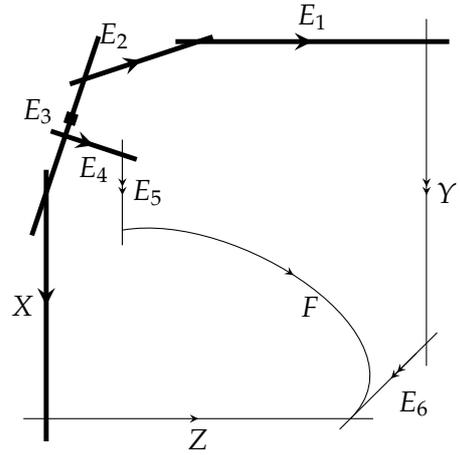
\begin{minipage}{0.5\textwidth}
$$
\scalebox{0.9}{
\begin{tikzpicture}
[decoration={markings, 
    mark= at position 0.5 with {\arrow{stealth }}}
]
\draw[line width=0.65mm][postaction={decorate}](0,1.3-1)--node[left]{$X$}(0,-0.3-1-2);

\draw(2+1+2,2.3)--node[right]{$Y$}(2+1+2,-0.3-1-1)node[sloped,pos=0.5,allow upside down]{\arrowIIn};

\draw[postaction={decorate}](-0.3,0-1-2)--node[below]{$Z$}(2+0.3+1+1,0-1-2);
\draw[line width=0.65mm][postaction={decorate}](0.7+1,2)--node[above]{$E_1$}(2+0.3+1+2,2);
\draw[line width=0.65mm][shorten >=-0.2cm,shorten <=-0.2cm][postaction={decorate}](1-0.5,1+0.5)--node[above left]{$E_2$}(2,2);
\draw[line width=0.65mm][shorten >=-0.6cm,shorten <=-0.6cm](0,0)--node[above left]{$E_3$}(0.5,1.5)node[
    sloped,
    pos=0.3,
    allow upside down]{\arrowBox};

\draw[line width=0.65mm][postaction={decorate}][shorten >=-0.2cm,shorten <=-0.2cm](0.25,0.75)--node[below]{$E_4$}(1,0.5);

\draw[shorten >=-0.2cm,shorten <=-0.2cm](1,0.5)--node[right]{$E_5$}(1,0.5-1)node[sloped,pos=0.3,allow upside down]{\arrowIIn};

\draw[shorten >=-0.2cm,shorten <=-0.2cm](5,-2)--node[below right]{$E_6$}(4,-3)node[sloped,pos=0.3,allow upside down]{\arrowIIn};

\draw[postaction={decorate}](1,0.5-1) to[out=10,in=45] node[below]{$F$}(3+1,-1-2);

\node[left=0.2cm] at (0,1-1){$[1,0]$};
\node[below left=0.2cm] at (0,-3){$[-1,2]$};
\node[above=0.3cm] at (1+1,2){$[-1,2]$};
\node[right=0.2cm] at (5,2.2){$[-2,3]$};
\node[above left=0.2cm] at (0.5,1.5){$[0,1]$};
\node at (1.7,1) {$[-1,2]$};
\node at (1.2,-1) {$[-2,1]$};
\node at (5.8,-2.2){$[1,-3]$};
\node at (4,-3.4){$[-2,-1]$};

\end{tikzpicture}
}
$$
\end{minipage}

\begin{minipage}{0.47\textwidth}

\paragraph{\underline{\underline{Case 3J}}}
\label{3J}

\begin{itemize} 
\item The orbifold group is $E_6$.
\item The degree of 
    $$\widetilde{M}=Bl_{X\cap Y,X\cap E_1, X\cap E_2,p_4,E_4\cap F,E_5\cap F}\mathbb{P}^2$$
    is 3 and the picard number of $\widehat{M}$ is 1. Here $p_4$ is a point in $E_2$ different from $X\cap E_3,E_2\cap E_3$.  
\item $K_{\widetilde{M}}=-3Z+E_1+2E_2+3E_3+4E_4+5E_5+6E_6$.
\end{itemize}

\end{minipage}%
\begin{minipage}{0.5\textwidth}
$$
\scalebox{0.9}{
\begin{tikzpicture}
[decoration={markings, 
    mark= at position 0.5 with {\arrow{stealth }}}
]
\draw[line width=0.65mm][postaction={decorate}](0,1.3-1)--node[left]{$X$}(0,-0.3-1-2);
\draw[postaction={decorate}](2+1+2,2.3)--node[right]{$Y$}(2+1+2,-0.3-1-2);
\draw[postaction={decorate}](-0.3,0-1-2)--node[below]{$Z$}(2+0.3+1+2,0-1-2);
\draw[line width=0.65mm][postaction={decorate}](0.7+1,2)--node[above]{$E_1$}(2+0.3+1+2,2);
\draw[line width=0.65mm][shorten >=-0.2cm,shorten <=-0.2cm][postaction={decorate}](1-0.5,1+0.5)--node[above left]{$E_2$}(2,2);
\draw[line width=0.65mm][shorten >=-0.6cm,shorten <=-0.6cm](0,0)--node[above left]{$E_3$}(0.5,1.5)node[
    sloped,
    pos=0.3,
    allow upside down]{\arrowBox};

\draw[line width=0.65mm][postaction={decorate}][shorten >=-0.2cm,shorten <=-0.2cm](0.25,0.75)--node[below]{$E_4$}(1,0.5);
\draw[line width=0.65mm][postaction={decorate}][shorten >=-0.2cm,shorten <=-0.2cm](1,0.5)--node[right]{$E_5$}(1,0.5-1);

\draw[shorten >=-0.2cm,shorten <=-0.2cm](1,0.5-1)--node[below]{$E_6$}(1+1,0.5-1)node[sloped,pos=0.3,allow upside down]{\arrowIIn};

\draw[postaction={decorate}](1+1,0.5-1) to[out=-90,in=180] node[below]{$F$}(3+2,-1-2);

\node[left=0.2cm] at (0,1-1){$[1,0]$};
\node[below left=0.2cm] at (0,-3){$[-1,2]$};
\node[above=0.3cm] at (1+1,2){$[-1,2]$};
\node[right=0.2cm] at (5,2.2){$[-2,3]$};
\node[below right=0.2cm] at (5,-3){$[-2,-3]$};
\node[above left=0.2cm] at (0.5,1.5){$[0,1]$};
\node at (1.7,1) {$[-1,2]$};
\node at (1.1,-1.2) {$[-2,3]$};
\node at (2.7,-0.4) {$[-3,1]$};

\end{tikzpicture}
}
$$
\end{minipage}

\begin{minipage}{0.47\textwidth}
\paragraph{\underline{\underline{Case 3K}}}
\begin{itemize}
    \item The pair $(\widetilde{M},\mathfrak{E})$ is biholomorphic to the pair of Case \hyperref[3G]{3G}.
\end{itemize}

\end{minipage}%
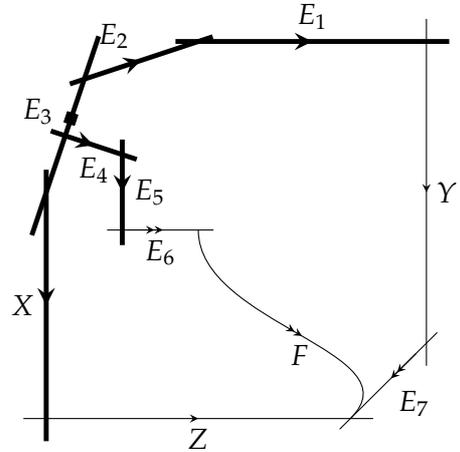
\begin{minipage}{0.5\textwidth}
$$
\scalebox{0.9}{
\begin{tikzpicture}
[decoration={markings, 
    mark= at position 0.5 with {\arrow{stealth }}}
]
\draw[line width=0.65mm][postaction={decorate}](0,1.3-1)--node[left]{$X$}(0,-0.3-1-2);
\draw[postaction={decorate}](2+1+2,2.3)--node[right]{$Y$}(2+1+2,-0.3-1-1);
\draw[postaction={decorate}](-0.3,0-1-2)--node[below]{$Z$}(2+0.3+1+1,0-1-2);
\draw[line width=0.65mm][postaction={decorate}](0.7+1,2)--node[above]{$E_1$}(2+0.3+1+2,2);
\draw[line width=0.65mm][shorten >=-0.2cm,shorten <=-0.2cm][postaction={decorate}](1-0.5,1+0.5)--node[above left]{$E_2$}(2,2);
\draw[line width=0.65mm][shorten >=-0.6cm,shorten <=-0.6cm](0,0)--node[above left]{$E_3$}(0.5,1.5)node[
    sloped,
    pos=0.3,
    allow upside down]{\arrowBox};

\draw[line width=0.65mm][postaction={decorate}][shorten >=-0.2cm,shorten <=-0.2cm](0.25,0.75)--node[below]{$E_4$}(1,0.5);
\draw[line width=0.65mm][shorten >=-0.2cm,shorten <=-0.2cm][postaction={decorate}](1,0.5)--node[right]{$E_5$}(1,0.5-1);

\draw[shorten >=-0.2cm,shorten <=-0.2cm](1,0.5-1)--node[below]{$E_6$}(1+1,0.5-1)node[sloped,pos=0.3,allow upside down]{\arrowIIn};

\draw[shorten >=-0.2cm,shorten <=-0.2cm](5,-2)--node[below right]{$E_7$}(4,-3)node[sloped,pos=0.3,allow upside down]{\arrowIIn};

\draw(1+1,0.5-1) to[out=-90,in=45] node[below]{$F$}node[sloped,pos=0.5,allow upside down]{\arrowIIn}(3+1,-1-2);

\end{tikzpicture}
}
$$
\end{minipage}

\begin{minipage}{0.47\textwidth}
\paragraph{\underline{\underline{Case 3L}}}
\label{3L}

\begin{itemize} 
\item The orbifold group is $E_7$.
\item The degree of 
        $$\widetilde{M}=Bl_{X\cap Y,X\cap E_1, X\cap E_2,p_4,E_4\cap F,E_5\cap F,E_6\cap F}\mathbb{P}^2$$
    is 2 and the picard number of $\widehat{M}$ is 1. Here $p_4$ is a point in $E_2$ different from $X\cap E_3,E_2\cap E_3$.  
\item $K_{\widetilde{M}}=-3Z+E_1+2E_2+3E_3+4E_4+5E_5+6E_6+7E_7$.
\end{itemize}

\end{minipage}%
\begin{minipage}{0.5\textwidth}
$$
\scalebox{0.9}{
\begin{tikzpicture}
[decoration={markings, 
    mark= at position 0.5 with {\arrow{stealth }}}
]
\draw[line width=0.65mm][postaction={decorate}](0,1.3-1)--node[left]{$X$}(0,-0.3-1-2);
\draw[postaction={decorate}](2+1+2,2.3)--node[right]{$Y$}(2+1+2,-0.3-1-2);
\draw[postaction={decorate}](-0.3,0-1-2)--node[below]{$Z$}(2+0.3+1+2,0-1-2);
\draw[line width=0.65mm][postaction={decorate}](0.7+1,2)--node[above]{$E_1$}(2+0.3+1+2,2);
\draw[line width=0.65mm][shorten >=-0.2cm,shorten <=-0.2cm][postaction={decorate}](1-0.5,1+0.5)--node[above left]{$E_2$}(2,2);
\draw[line width=0.65mm][shorten >=-0.6cm,shorten <=-0.6cm](0,0)--node[above left]{$E_3$}(0.5,1.5)node[
    sloped,
    pos=0.3,
    allow upside down]{\arrowBox};

\draw[line width=0.65mm][postaction={decorate}][shorten >=-0.2cm,shorten <=-0.2cm](0.25,0.75)--node[below]{$E_4$}(1,0.5);
\draw[line width=0.65mm][shorten >=-0.2cm,shorten <=-0.2cm][postaction={decorate}](1,0.5)--node[right]{$E_5$}(1,0.5-1);
\draw[line width=0.65mm][shorten >=-0.2cm,shorten <=-0.2cm][postaction={decorate}](1,0.5-1)--node[below]{$E_6$}(1+1,0.5-1);

\draw[shorten >=-0.2cm,shorten <=-0.2cm](1+1,0.5-1)--node[right]{$E_7$}(1+1,0.5-1-1)node[sloped,pos=0.3,allow upside down]{\arrowIIn};

\draw[postaction={decorate}] (1+1,0.5-1-1) to[out=10,in=180] node[below]{$F$}(3+2,-1-2);

\node[left=0.2cm] at (0,1-1){$[1,0]$};
\node[below left=0.2cm] at (0,-3){$[-1,2]$};
\node[above=0.3cm] at (1+1,2){$[-1,2]$};
\node[right=0.2cm] at (5,2.2){$[-2,3]$};
\node[below right=0.2cm] at (5,-3){$[-2,-3]$};
\node[above left=0.2cm] at (0.5,1.5){$[0,1]$};
\node at (1.7,1) {$[-1,2]$};
\node at (1.1,-1.2) {$[-2,3]$};
\node at (2.7,-0.4) {$[-3,4]$};
\node at (2.1,-2) {$[-4,1]$};

\end{tikzpicture}
}
$$
\end{minipage}

\begin{minipage}{0.47\textwidth}

\paragraph{\underline{\underline{Case 3M}}}
\label{3M}

\begin{itemize}
\item The orbifold group is $E_8$.
\item The degree of 
        $$\widetilde{M}=Bl_{X\cap Y,X\cap E_1, X\cap E_2,p_4,E_4\cap F,E_5\cap F,E_6\cap F,E_7\cap F}\mathbb{P}^2$$
    is 1 and the picard number of $\widehat{M}$ is 1. Here $p_4$ is a point in $E_2$ different from $X\cap E_3,E_2\cap E_3$.  
\item $K_{\widetilde{M}}=-3Z+E_1+2E_2+3E_3+4E_4+5E_5+6E_6+7E_7+8E_8$.
\end{itemize}

\end{minipage}%
\begin{minipage}{0.5\textwidth}
$$
\scalebox{0.9}{
\begin{tikzpicture}
[decoration={markings, 
    mark= at position 0.5 with {\arrow{stealth }}}
]
\draw[line width=0.65mm][postaction={decorate}](0,1.3-1)--node[left]{$X$}(0,-0.3-1-2);
\draw[postaction={decorate}](2+1+2,2.3)--node[right]{$Y$}(2+1+2,-0.3-1-2);
\draw[postaction={decorate}](-0.3,0-1-2)--node[below]{$Z$}(2+0.3+1+2,0-1-2);
\draw[line width=0.65mm][postaction={decorate}](0.7+1,2)--node[above]{$E_1$}(2+0.3+1+2,2);
\draw[line width=0.65mm][shorten >=-0.2cm,shorten <=-0.2cm][postaction={decorate}](1-0.5,1+0.5)--node[above left]{$E_2$}(2,2);
\draw[line width=0.65mm][shorten >=-0.6cm,shorten <=-0.6cm](0,0)--node[above left]{$E_3$}(0.5,1.5)node[
    sloped,
    pos=0.3,
    allow upside down]{\arrowBox};

\draw[line width=0.65mm][postaction={decorate}][shorten >=-0.2cm,shorten <=-0.2cm](0.25,0.75)--node[below]{$E_4$}(1,0.5);
\draw[line width=0.65mm][shorten >=-0.2cm,shorten <=-0.2cm][postaction={decorate}](1,0.5)--node[right]{$E_5$}(1,0.5-1);
\draw[line width=0.65mm][shorten >=-0.2cm,shorten <=-0.2cm][postaction={decorate}](1,0.5-1)--node[below]{$E_6$}(1+1,0.5-1);
\draw[line width=0.65mm][shorten >=-0.2cm,shorten <=-0.2cm][postaction={decorate}](1+1,0.5-1)--node[right]{$E_7$}(1+1,0.5-1-1);

\draw[shorten >=-0.2cm,shorten <=-0.2cm](1+1,0.5-1-1)--node[below]{$E_8$}(1+1+1,0.5-1-1)node[sloped,pos=0.5,allow upside down]{\arrowIIn};

\draw[postaction={decorate}] (1+1+1,0.5-1-1) to[out=-90,in=180] node[below]{$F$}(3+2,-1-2);

\end{tikzpicture}
}
$$
\end{minipage}

\subsection{Conclusion}

To summarize, concluding Propositions \ref{prop:case1}, \ref{prop:case2}, \ref{prop:case31}, \ref{prop:case32}, we obtain a list of candidates of pairs $(\widehat{M},\mathfrak{E})$, except the reversed Eguchi-Hanson, where $M = \widehat{M}\setminus E$ could potentially support Hermitian non-K\"ahler ALE gravitational instantons with structure group in $SU(2)$, and $\mathfrak{E}$ being the holomorphic extremal vector field up to constant multiples. These surfaces are listed in Table \ref{cases}.

\begin{table}[h]
\centering
\caption{Possible $(\widehat{M},\mathfrak{E})$.}
\label{cases}
\begin{tabular}{ c c c c }
\hline
\specialrule{0em}{1pt}{1pt}
 Case &  $\Gamma\subset SU(2)$ & Picard rank of $\widehat{M}$ & The $\mathfrak{E}$ action\\ 
\specialrule{0em}{1pt}{1pt}
\hline\hline
\specialrule{0em}{2pt}{2pt}
\hyperref[1A]{1A} & $A_1$ & 2& $t\curvearrowright[x:y:z]=[tx:ty:z]$ on $\mathbb{P}^2$\\  
\specialrule{0em}{2pt}{2pt}
\hyperref[1B]{1B}  &$A_1$ & 3& $t\curvearrowright[x:y:z]=[tx:ty:z]$ on $\mathbb{P}^2$\\ 
\specialrule{0em}{2pt}{2pt}
\hyperref[1C]{1C}  & $A_1$ & 4& $t\curvearrowright[x:y:z]=[tx:ty:z]$ on $\mathbb{P}^2$\\
\specialrule{0em}{2pt}{2pt}
\hyperref[2A]{2A} &  $A_1$ &3& $t\curvearrowright[x:y:z]=[t^{-1}x:t^{-1}y:z]$ on $\mathbb{P}^2$\\
\specialrule{0em}{1pt}{1pt}
\hyperref[2D]{2D} &  $A_3$ &3& $t\curvearrowright[x:y:z]=[t^{-1}x:t^{-1}y:z]$ on $\mathbb{P}^2$\\
\specialrule{0em}{1pt}{1pt}
\hyperref[2E]{2E} &  $D_4$ &3& $t\curvearrowright[x:y:z]=[t^{-1}x:t^{-1}y:z]$ on $\mathbb{P}^2$\\
\specialrule{0em}{1pt}{1pt}
\hyperref[3Aproblem]{3A} &  $A_3$ &2& $t\curvearrowright[x:y:z]=[t^{2}x:ty:z]$ on $\mathbb{P}^2$\\
\specialrule{0em}{1pt}{1pt}
\hyperref[3C]{3C} &  $D_4$ &2& $t\curvearrowright[x:y:z]=[t^{2}x:ty:z]$ on $\mathbb{P}^2$\\
\specialrule{0em}{1pt}{1pt}
\hyperref[3E]{3E} &  $D_5$ &2& $t\curvearrowright[x:y:z]=[t^{2}x:ty:z]$ on $\mathbb{P}^2$\\
\specialrule{0em}{1pt}{1pt}
\hyperref[3G]{3G} &  $E_6$ &2& $t\curvearrowright[x:y:z]=[t^{2}x:ty:z]$ on $\mathbb{P}^2$\\
\specialrule{0em}{1pt}{1pt}
\hyperref[3H]{3H} &  $D_5$ &1& $t\curvearrowright[x:y:z]=[t^{3}x:ty:z]$ on $\mathbb{P}^2$\\
\specialrule{0em}{1pt}{1pt}
\hyperref[3J]{3J} &  $E_6$ &1& $t\curvearrowright[x:y:z]=[t^{3}x:ty:z]$ on $\mathbb{P}^2$\\
\specialrule{0em}{1pt}{1pt}
\hyperref[3L]{3L} &  $E_7$ &1& $t\curvearrowright[x:y:z]=[t^{3}x:ty:z]$ on $\mathbb{P}^2$\\
\specialrule{0em}{1pt}{1pt}
\hyperref[3M]{3M} &  $E_8$ &1& $t\curvearrowright[x:y:z]=[t^{3}x:ty:z]$ on $\mathbb{P}^2$\\
\specialrule{0em}{1pt}{1pt}
\hline
\end{tabular}
\end{table}

\begin{proposition}\label{prop:complextype}
If a Hermitian non-K\"ahler ALE gravitational instanton $(M,h)$ with structure group in $SU(2)$ other than the reversed Eguchi-Hanson exists, then $M$ would be one of the surfaces listed in Table \ref{cases}. The corresponding holomorphic extremal vector field would be $\mathfrak{E}$ up to scaling. 
\end{proposition}

\section{The \texorpdfstring{$\mathcal{A}$}--functional and  Bach-flat K\"ahler metrics}
\label{sec:nonexistence}

We need to use finer techniques to detect whether $(\widehat{M},\mathfrak{E})$ in Table \ref{cases} admits special Bach-flat K\"ahler metrics. For extremal K\"ahler metrics whose holomorphic extremal vector field induces a $\mathbb{C}^*$-action, explicitly knowing what the action is, we can calculate the minimum of its scalar curvature by combining symplectic and complex geometry techniques, following LeBrun-Simanca \cite{ls}. This will be done in this section.

\subsection{Computing the minimum of scalar curvature}\label{computefutaki}

We calculate the minimum of $s_g$ for an extremal K\"ahler $g$ following the idea of \cite{ls}. It turns out that $\min s_g$ only depends on the K\"ahler class $[\omega]$ and the holomorphic extremal vector field. Recall $\mathfrak{K}=\nabla^{1,0}s_{{g}}=\frac12(\nabla s_{{g}}-\sqrt{-1}J\nabla s_{{g}})$ is the holomorphic extremal vector field. We only consider the situation that the complex surface (orbifold) $X$ is rational since this is enough for our application.

On a complex surface (or orbifold) $M$, given a K\"ahler metric $g$ in a K\"ahler class $[\omega]$, for a holomorphic field $\xi=\nabla^{1,0} f$ that admits a potential function $f:M\to\mathbb{C}$, where $\nabla^{1,0}f$ refers to the vector field dual to $\overline{\partial}f$, the Futaki invariant is defined as 
\begin{equation}
    \mathcal{F}(\xi,[\omega])=-\int f(s_g-s_0)d\mu.
\end{equation}
Here, $s_g$ is the scalar curvature of the metric $g$ and $s_0=8\pi{c_1[\omega]}/{[\omega]^2}$ is the average of the scalar curvature. The volume form is $d\mu=\omega^2/2$. It turns out that the Futaki invariant does not depend on the choice of the K\"ahler metric $g$ in the K\"ahler class $[\omega]$ and the potential function $f$, so it makes sense to write it as $\mathcal{F}(\xi,[\omega])$.  Note that the Futaki invariant can also be defined even if the holomorphic field $\xi$ does not admit a potential function.
The Calabi functional is defined as 
$$\mathcal{C}(g)=\int s_g^2d\mu$$
on the space of K\"ahler metrics $g$. Restricted to a fixed K\"ahler class $[\omega]$, a K\"ahler metric $g\in[\omega]$ is a critical point of $\mathcal{C}$ over $[\omega]$ if and only if $\nabla^{1,0}s_g$ is a holomorphic vector field. Such metrics are the \textit{extremal K\"ahler metrics} in the sense of Calabi. 
If a K\"ahler metric $g$ is an extremal metric, the following equality holds
$$\mathcal{C}(g)=s_0^2\int d\mu+\int(s_g-s_0)^2d\mu=32\pi^2\frac{(c_1[\omega])^2}{[\omega]^2}-\mathcal{F}(\mathfrak{E},[\omega]),$$
with $\mathfrak{E}=\nabla^{1,0}(s_g-s_0)$, which is a holomorphic vector field.

\begin{assumption}\label{assumption}
The K\"ahler metric $g$ in the K\"ahler class $[\omega]$ is an extremal K\"ahler metric on a rational complex 2-dimensional orbifold $X$, whose scalar curvature is nonconstant. The holomorphic extremal vector field $\mathfrak{K}$ induces a holomorphic $\mathbb{C}^*$-action. Orbifold points only appear in $c_\pm$. 
\end{assumption}

To compute $\mathcal{F}(\mathfrak{K},[\omega])$, assuming the existence of extremal metrics, it suffices to consider
$\mathcal{F}(\mathfrak{K},[\omega])=-\int_X (s_g-s_0)^2d\mu$. The function
$s_g-s_0$ is a Hamiltonian function of the real holomorphic vector field $-2\Im\mathfrak{K}=J\nabla s_g$, in the sense that 
$$d(s_g-s_0)=-\omega\left(-2\Im\mathfrak{K},\cdot\right).$$
It is more convenient to work with the generator $\mathfrak{E}$ of the $\mathbb{C}^*$-action instead of $\mathfrak{K}$. Hence, let $\xi$ be the infinitesimal generator of the $S^1$-action associated to the primitive holomorphic $\mathbb{C}^*$-action, and $t$ be the moment map associated to this $S^1$-action with $\max t=-\min t=a$. That is,
$$dt=-\omega(\xi,\cdot).$$
The average of $t$ over $X$ is denoted by $t_0$. There exists a nonzero positive constant $h$ such that 
\begin{equation}\label{h}
t-t_0=h(s_g-s_0).
\end{equation}

Initially in $X$, $c_+,c_-$ might be fixed points. But as in \cite{ls}, we can consider $\sigma:X'\to X$, where $\sigma$ comes from blowing up $c_\pm$ suitably, so that the attractive set and repulsive set in $X'$ are both fixed curves, and if there are orbifold points in $c_\pm$, $\sigma$ resolves these orbifold points at the same time.
The fact that repulsive and attractive sets in $X'$ are fixed curves implies generic $\mathbb{C}^*$-orbits are rational curves with trivial self-intersection. This indicates that $X'$ is fibered over a rational curve $\Sigma$, in the sense that after contracting some curves  in some  fibers, it becomes a ruled surface over $\Sigma\simeq\mathbb{P}^1$ .
Because blow-ups are only taken in $c_\pm$, pulling back the K\"ahler form, $\sigma^*\omega$ is still a K\"ahler form defined on $X'\setminus\sigma^{-1}(c_\pm)$. Generic flows of the $\mathbb{C}^*$-action are flowing out of $\sigma^{-1}(c_-)$ and flowing into $\sigma^{-1}(c_+)$. In $X'$, $\sigma^{-1}(c_\pm)$ may be union of curves.

\begin{example}
Consider the quotient $X=\mathbb{P}^2/\mathbb{\mathbb{Z}}_3$, where the $\mathbb{Z}_3$-action is defined by $\xi_3\cdot[x:y:z]=[\xi_3 ^2x:\xi_3 y:z]$. Here, $\xi_3$ is the cubic unit root. The orbifold $X$ has three $A_2$ orbifold points located at $[1:0:0]$, $[0:1:0]$, and $[0:0:1]$.
We consider the holomorphic $\mathbb{C}^*$-action on $\mathbb{P}^2$ given by $t\cdot[x:y:z]=[tx:ty:z]$. This action has weights $[1,1]$ at $[0:0:1]$ and weights $[-1,0]$ at $[1:0:0]$ and $[0:0:1]$. Note that this action commutes with the $\mathbb{Z}_3$-action, therefore descends to a $\mathbb{C}^*$-action on $X$.

In $X$, the attractive set $c_+$ is the quotient $\{[x:y:0]\}/\mathbb{Z}_3$, which is a $\mathbb{P}^1$. The repulsive set $c_-$ is the point $[0:0:1]$. Denote $\{[x:y:0]\}/\mathbb{Z}_3$ by $L$.
Now we take the blow-up $\sigma$:
\begin{itemize}
\item $c_-=[0:0:1]$ is an orbifold point, so a blow-up needs to be taken at $[0:0:1]$. The resulting exceptional set is a union of two $(-2)$-curves, $E_1^z$ and $E_2^z$.
\item $c_+=\{[x:y:0]\}/\mathbb{Z}_3$ is a curve with two orbifold points, so blow-ups need to be taken at $[1:0:0],[0:1:0]$. They are $A_2$ singularities so the exceptional set of these two blow-ups are both unions of two $(-2)$-curves, $E^x_1,E^x_2,E^y_1,E^y_2$.
\end{itemize}
This way we get the minimal resolution of $X$. 
The following picture illustrates the situation.
$$
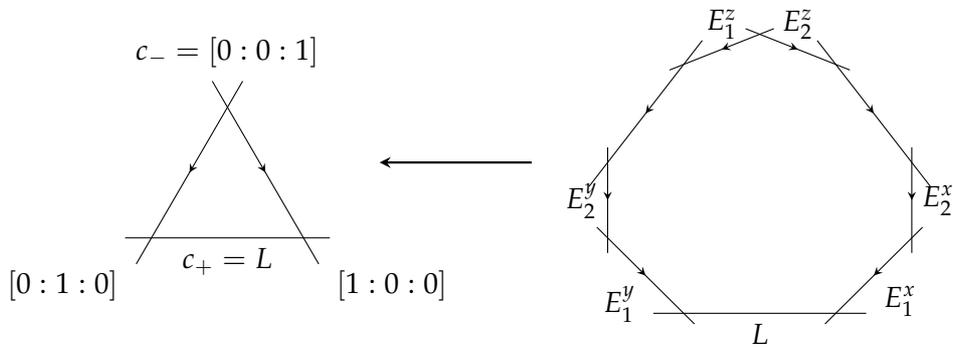
\begin{tikzpicture}
[decoration={markings, 
    mark= at position 0.5 with {\arrow{stealth }}}
]
\draw[shorten >=-0.4cm,shorten <=-0.4cm][postaction={decorate}](1,1.732)--(0,0);
\draw[postaction={decorate}](0.8,1.732+1.732*0.2)--(2.2,-1.732*0.2);
\draw(-0.2*1.732,0)--node[below]{$c_+=L$}(2+0.2*1.732,0);

\node[above=0.4cm] at (1,1.732){$c_-=[0:0:1]$};
\node[below left=0.3cm] at (0,0){$[0:1:0]$};
\node[below right=0.3cm] at (2,0){$[1:0:0]$};

\draw[shorten >=-0.4cm,shorten <=-0.4cm][postaction={decorate}](7,1.3+1)--(6,1);
\draw[shorten >=-0.4cm,shorten <=-0.4cm][postaction={decorate}](9,1.3+1)--(10,1);
\draw[shorten >=-0.4cm,shorten <=-0.4cm](7,-1)--node[below]{$L$}(9,-1);

\draw[shorten >=-0.2cm,shorten <=-0.2cm][postaction={decorate}](6,1)--node[left]{$E^y_2$}(6,0);
\draw[shorten >=-0.2cm,shorten <=-0.2cm][postaction={decorate}](6,0)--node[below left]{$E^y_1$}(7,-1);

\draw[shorten >=-0.2cm,shorten <=-0.2cm][postaction={decorate}](10,1)--node[right]{$E^x_2$}(10,0);
\draw[shorten >=-0.2cm,shorten <=-0.2cm][postaction={decorate}](10,0)--node[below right]{$E^x_1$}(9,-1);

\draw[shorten >=-0.2cm,shorten <=-0.2cm][postaction={decorate}](8,1.3+1+0.4)--node[above]{$E^z_1$}(7,1.3+1);
\draw[shorten >=-0.2cm,shorten <=-0.2cm][postaction={decorate}](8,1.3+1+0.4)--node[above]{$E^z_2$}(9,1.3+1);

\draw[-stealth][line width=0.3mm](5,1)--(3,1);

\end{tikzpicture}
$$
However, the repulsive set is still a fixed point $E_1^z\cap E_2^z$. To resolve this, we need to perform one more blow-up at $E_1^z\cap E_2^z$. The resulting exceptional curve of this blow-up is denoted by $E_3^z$, and the resulting surface is the desired $X'$. Specifically, $E_1^z\cup E_3^z\cup E_2^z$ is $\sigma^{-1}(c_-)$, while $E_2^y\cup E_1^y\cup L\cup E_1^x\cup E_2^x$ is $\sigma^{-1}(c_+)$.
$$
\begin{tikzpicture}
[decoration={markings, 
    mark= at position 0.5 with {\arrow{stealth }}}
]

\draw[shorten >=-0.2cm,shorten <=-0.2cm][postaction={decorate}](7-0.4,1.3+1)--(6-0.2,1);
\draw[shorten >=-0.2cm,shorten <=-0.2cm][postaction={decorate}](9+0.4,1.3+1)--(10+0.2,1);
\draw[line width=0.5mm][shorten >=-0.4cm,shorten <=-0.4cm](7-0.2,-1)--node[below]{$L$}(9+0.2,-1);

\draw[line width=0.5mm][shorten >=-0.2cm,shorten <=-0.2cm][postaction={decorate}](6-0.2,1)--node[left]{$E^y_2$}(6-0.2,0);
\draw[line width=0.5mm][shorten >=-0.2cm,shorten <=-0.2cm][postaction={decorate}](6-0.2,0)--node[below left]{$E^y_1$}(7-0.2,-1);

\draw[line width=0.5mm][shorten >=-0.2cm,shorten <=-0.2cm][postaction={decorate}](10.2,1)--node[right]{$E^x_2$}(10.2,0);
\draw[line width=0.5mm][shorten >=-0.2cm,shorten <=-0.2cm][postaction={decorate}](10.2,0)--node[below right]{$E^x_1$}(9.2,-1);

\draw[line width=0.5mm][shorten >=-0.2cm,shorten <=-0.2cm][postaction={decorate}](8-0.6,1.3+1+0.4)--node[above]{$E^z_1$}(7-0.4,1.3+1);
\draw[line width=0.5mm][shorten >=-0.2cm,shorten <=-0.2cm][postaction={decorate}](8+0.6,1.3+1+0.4)--node[above]{$E^z_2$}(9+0.4,1.3+1);

\draw[postaction={decorate}](8-0.6+0.4,1.3+1+0.4+0.15) to[out=-70,in=60] node[left]{generic $\mathbb{C}^*$}(7-0.2+0.8,-1-0.15);
\node at (7-0.2+0.5,0.53) {orbit};

\draw[line width=0.5mm][shorten >=-0.2cm,shorten <=-0.2cm](8-0.6,1.3+1+0.4)--node[above]{$E^z_3$}(8+0.6,1.3+1+0.4);

\end{tikzpicture}
$$
Generic $\mathbb{C}^*$-orbits are flowing from $E_3^z$ to $L$.
If we consider the K\"ahler form $\omega$ on $X$ descended from the Fubini-Study metric $\omega_{FS}={\sqrt{-1}}\partial\overline{\partial}\log(|x|^2+|y|^2+|z|^2)$ on $\mathbb{P}^2$, after passing to $X'$, the pulled back K\"ahler form $\sigma^*(\omega)$ lives on $X'\setminus\sigma^{-1}(c_\pm)$.
The moment map for the $S^1$-action on $X=\mathbb{P}^2/\mathbb{Z}_3$ is given by 
$$[x:y:z]\mapsto\frac{|x|^2+|y|^2}{|x|^2+|y|^2+|z|^2},$$
and this is also the moment map for the $S^1$-action on $X'\setminus\sigma^{-1}(c_\pm)$. The symplectic quotient at level $t$ gives a symplectic form with area $\frac132\pi t$ on $\mathbb{P}^1$ (because of the $\mathbb{Z}_3$ quotient). 
\end{example}

As in \cite{ls}, there is a projection map
$$p: X'\setminus\sigma^{-1}(c_\pm)\to\Sigma\times(-a,a).$$
The first factor is given by the fibration to $\Sigma$ while the second factor is the moment map $t$.
Fibers of the map $p$  are the $S^1$-orbits. Fixed points of this $S^1$-action inside $X'$ are denoted by $\alpha_1,\ldots\alpha_m$. 
For $t\notin\{t(\alpha_1),\ldots,t(\alpha_m)\}$,
the restriction of $p$ over $\Sigma\times t$ has the structure of orbifold $S^1$-bundle over the orbifold $\Sigma\times t$. 
Globally the map $p$ provides a structure of orbifold $S^1$-bundle
$$p:Y\to \Sigma\times(-a,a)\setminus\{p(\alpha_1),\ldots,p(\alpha_m)\},$$ 
where $Y$ denotes the preimage of $\Sigma\times(-a,a)\setminus\{p(\alpha_1),\ldots,p(\alpha_m)\}$ under the projection map $p$, excluding the fixed points of the $S^1$-action.

On each $\Sigma\times t$ with $t\notin\{t(\alpha_1),\ldots,t(\alpha_m)\}$ 
there is the metric $g(t)$ induced by the symplectic reduction. The orthogonal complement of the $S^1$-orbits in $Y$ with respect to the K\"ahler metric on $Y$ defines a connection 1-form for the orbifold $S^1$-bundle. Denote the curvature two-form of this $U(1)$-connection  by $\Omega$.
Let  $t_j=t(\alpha_j)$ be critical values of the moment map. 
If we set $A(t)$ as the area of $(\Sigma,g(t))$ and $t$ is not a critical value, then same calculation in page 314 of \cite{ls} gives
$$\frac{dA}{dt}(t)=\int_{\Sigma\times\{t\}}-\sqrt{-1}\Omega=-2\pi c_1(Y)[\Sigma_t].$$
Here $[\Sigma_t]$ is the homology class of $\Sigma\times{t}$ and $c_1(Y)$ is the Chern class of the orbifold $S^1$-bundle. 
For any $b,c\in(-a,a)$, if we let $S_j\subset\Sigma\times(-a,a)\setminus\{p(\alpha_1),\ldots,p(\alpha_m)\}$ be small 2-spheres around points $p(\alpha_j)$, we then have
$$[\Sigma_c]-[\Sigma_b]=\sum_{t_j\in[b,c]}[S_j].$$
Here $S_j$ is assigned with the outward pointing orientation. Then,
\begin{equation}\label{eq:cross}
-2\pi c_1(Y)[\Sigma_c]+2\pi c_1(Y)[\Sigma_b]=-2\pi\sum_{t_j\in[b,c]}c_1(Y)[S_j].
\end{equation}
Generally, if the weights of a fixed point $\alpha$ are $[r,s]$, then the Chern class of the orbifold $S^1$-bundle, evaluated on the class $[S]$ of 2-sphere around $p(\alpha)$,  can be calculated as 
\begin{equation}\label{eq:Chern number}
    c_1(Y)[S]=-\frac{1}{rs}.
\end{equation}
In particular, if weights of the fixed point $\alpha_j$ is $[r_j,s_j]$ (as $\alpha_j$ is not in $c_\pm$, we can always assume $r_j>0,s_j<0$), we then have $c_1(Y)[S_j]=-\frac{1}{r_js_j}$. In conclusion, with equation (\ref{eq:cross}), we have the formula
\begin{equation}\label{eq:xxxx}
-2\pi c_1(Y)[\Sigma_c]+2\pi c_1(Y)[\Sigma_b]=2\pi\sum_{t_j\in[b,c]}\frac{1}{r_js_j},
\end{equation}
from which it follows that
\begin{equation}\label{secondorderderivative}
\frac{d^2A}{dt^2}(t)=2\pi\sum\frac{1}{r_js_j}\delta_{t_j}.
\end{equation}
Here $\delta_{t_j}$ is the Dirac measure at $t=t_j$. 
We adopt the notations $-2\pi c_1(Y)[\Sigma_a]$ and $-2\pi c_1(Y)[\Sigma_{-a}]$ to denote the limits of $-2\pi c_1(Y)[\Sigma_t]$ as $t$ approaches $a$ and $-a$, respectively. Equation
(\ref{eq:xxxx}) gives
\begin{equation}\label{eq:crossformula}
c_1(Y)[\Sigma_a]-c_1(Y)[\Sigma_{-a}]=\sum_{t_j\in(-a,a)}-\frac{1}{r_js_j}.
\end{equation} 

Consider the chains of rational curves $E_j'$ and $E_j$, obtained by tracing the $\mathbb{C}^*$-action forward from $\alpha_j$ to $c_+$ and backward from $\alpha_j$ to $c_{-}$, respectively. For a generic $\mathbb{C}^*$-orbit $F$ in $M$, its area is $\omega(F)=4\pi a$. As approaching to singular orbits, generic orbits break into several orbits with multiplicities.
If $\alpha_j$ is a fixed point with weights $[r_j,s_j]$, then generic orbits $F$ break into $r_jE_j'-s_jE_j$. In particular,
$\omega(r_jE_j'-s_jE_j)=\omega(F)=4\pi a$.
Computation in Theorem 3 in \cite{ls} now completely carries over:
{ 
\begin{equation}\label{ts}
\int_Mts_gd\mu=\omega(F)\left(\omega(c_+)-\omega(c_-)\right),
\end{equation}
\begin{align}\label{t0}
&\int_Md\mu=2\pi\int_{-a}^aA(t)dt\\
=\quad&\frac{\omega(F)^2}{8}\left[4\frac{\omega(c_+)+\omega(c_-)}{\omega(F)}+\sum_j\frac{1}{r_js_j}\left(\frac{\omega(r_jE_j')-\omega(-s_jE_j)}{\omega(F)}\right)^2+c_1(Y)[\Sigma_a]-c_1(Y)[\Sigma_{-a}]\right],\notag
\end{align}
\begin{align}\label{t1}
&\int_Mtd\mu=2\pi\int_{-a}^atA(t)dt\\
=\quad &\frac{\omega(F)^3}{96\pi}\left[6\frac{\omega(c_+)-\omega(c_-)}{\omega(F)}-\sum_j\frac{1}{r_js_j}\left(\frac{\omega(r_jE_j')-\omega(-s_jE_j)}{\omega(F)}\right)^3+c_1(Y)[\Sigma_{-a}]+c_1(Y)[\Sigma_a]\right],\notag
\end{align}
}
{ 
\begin{align}\label{t2}
&\int_Mt^2d\mu=2\pi\int_{-a}^at^2A(t)dt\\
=\quad &\frac{\omega(F)^4}{768\pi^2}\left[8\frac{\omega(c_+)+\omega(c_-)}{\omega(F)}+\sum_j\frac{1}{r_js_j}\left(\frac{\omega(r_jE_j')-\omega(-s_jE_j)}{\omega(F)}\right)^4+c_1(Y)[\Sigma_a]-c_1(Y)[\Sigma_{-a}]\right].\notag
\end{align}
}
For simplicity, let 
\begin{equation}
T=(T_s,T_0,T_1,T_2):=\left(\int_Mts_gd\mu,\int_Md\mu,\int_Mtd\mu,\int_Mt^2d\mu\right).
\end{equation}
Then
\begin{equation}\label{s0}
s_0=\frac{8\pi c_1[\omega]}{[\omega]^2}=\frac{4\pi c_1[\omega]}{T_0}.
\end{equation}

\begin{example}\label{exam:8.2}
Again we take the above example $\mathbb{P}^2/\mathbb{Z}_3$. The region $X'\setminus\sigma^{-1}\{c_\pm\}$ is biholomorphic to $X\setminus\{[0:0:1],[x:y:0]/\mathbb{Z}_3\}$.

To compute $c_1(Y)[\Sigma_{-a}]$, note that the orbifold point $[0:0:1]$ has weights $[1,1]$ and is an $A_2$ orbifold point. By \eqref{eq:Chern number}, it is direct to see
$$c_1(Y)[\Sigma_{-a}]=-\frac13.$$

To compute $c_1(Y)[\Sigma_{a}]$, the $S^1$-bundle over $\Sigma_{t}$ with $t$ close to $a$, after passing to the orbifold cover, is the dual of the normal bundle of $\Sigma_t\simeq c_+$, as described by
equations (3.12), (3.13) 
in \cite{ls}. 
The curve $c_+$ in $X$ lifted to the orbifold cover has self-intersection $1$.
Since the order of the orbifold structure group is $3$, we have
$$c_1(Y)[\Sigma_{a}]=-\frac13.$$

There is no fixed point when $t\in(-a,a)$ and equation (\ref{eq:crossformula}) 
holds, as expected. Moreover, for the metric $\omega$ on $\mathbb{P}^2/\mathbb{Z}_3$ descended from the standard Fubini-Study metric, $\omega(c_+)=\frac{\pi}{3},\omega(c_-)=0,\omega(F)=2\pi$. As there is no fixed point when $t\in (-a,a)$, equations (\ref{ts})-(\ref{t2}) give
$$T=(4\pi^2/3,2\pi^2/3,\pi^2/9,\pi^2/18).$$
The area of $\omega$ is exactly $2\pi^2/3$ because of the $\mathbb{Z}_3$ quotient. And
the scalar curvature of $\omega_{FS}={\sqrt{-1}}\partial\overline{\partial}\log(|x|^2+|y|^2+|z|^2)$ is 12 because of our choice of $\omega_{FS}$ on $\mathbb{P}^2$, which is compatible with $\int_{X}ts_gd\mu=12\int_Xtd\mu=4\pi^2/3$ by our computation.
\end{example}

\begin{example}\label{exam:8.3}
Take Case \hyperref[2D]{2D} as another example. 

To compute $c_1(Y)[\Sigma_a]$, note $c_+$ as a fixed point has weights $[-1,-1]$. The orbifold $S^1$-bundle over small spheres around the point $c_+$ has Chern number $-1$. But we should use the opposite orientation here (the orientation given by the symplectic form on $\Sigma_t$) when we compute $c_1(Y)[\Sigma_a]$.
Thus,
$$c_1(Y)[\Sigma_a]=1.$$

To compute $c_1(Y)[\Sigma_{-a}]$, observe that $c_-$ in $\widehat{M}$ is the orbifold point resulting from contracting $E\subset\widetilde{M}$. Since this orbifold point is an $A_3$ singularity, Theorem \ref{thm:cyclicweights} tells us that the  weights of the point $c_-$ are $[\frac12,\frac12]$. So we have
$$c_1(Y)[\Sigma_{-a}]=\frac14(-\frac{1}{\frac12\times\frac12})=-1.$$

There are three fixed points when $t\in(-a,a)$, with weights $[-2,1],[-1,1],[-2,1]$, and equation (\ref{eq:crossformula}) holds as expected
$$c_1(Y)[\Sigma_a]-c_1(Y)[\Sigma_{-a}]=2=\frac12+1+\frac12.$$
\end{example}

Now the extremal Futaki invariant can be computed as 
\begin{align}\label{futaki}
\mathcal{F}(\mathfrak{E},[\omega])=&-\int_M (s_g-s_0)^2d\mu\notag\\
=&-\frac1h\int_Mts_gd\mu+s_0\frac1h\int_Mtd\mu\notag\\
=&-\frac1hT_s+s_0\frac1hT_1,
\end{align}
with $h$ defined by (\ref{h}).
Multiplying (\ref{h}) by $t$ and take the integration one obtains 
\begin{align}\label{hh}
h=&\left.\left(-T_2+\frac{1}{T_0}T_1^2\right)\right/\left(-T_s+s_0T_1\right).
\end{align}
The minimum of $s_g$ is
\begin{align}\label{mins}
\min s_g&=\min\left\{\frac{1}{h}(t-t_0)\right\}+s_0\notag\\
&=\frac1h\min t-\frac{2}{h[\omega]^2}\int_Mtd\mu+\frac{8\pi c_1[\omega]}{[\omega]^2}\notag\\
&=-\frac{\omega(F)}{4\pi h}-\frac{1}{hT_0}T_1+\frac{4\pi c_1[\omega]}{T_0}.
\end{align}

\subsection{Bach-flat K\"ahler  metrics and its scalar curvature}
\label{sec:computation}
The computation we just performed gives a function 
\begin{equation}
{\min_{\mathfrak{E}}}\ s_g([\omega]):=-\frac{\omega(F)}{4\pi h}-\frac{1}{hT_0}T_1+\frac{4\pi c_1[\omega]}{T_0}
\end{equation} 
defined on the K\"ahler cone $\mathcal{K}$. When there is an extremal K\"ahler metric in the K\"ahler class $[\omega]$ with extremal vector field proportional to the vector field $\mathfrak{E}$, the function exactly coincides with the minimum of its scalar curvature. 
We added the lower index $\mathfrak{E}$ to emphasize the dependence of the function on the $\mathbb{C}^*$-action induced by $\mathfrak{E}$. The function $\min_{\mathfrak{E}} s_g$ defined on $\mathcal{K}$ is clearly homogeneous. Note that it is possible that $-T_s+s_0T_1=0$ in \eqref{h}, but in this case, as a consequence, $\int t(t-t_0)=\int(t-t_0)^2=0$, which is absurd. Hence when $-T_s+s_0T_1=0$, the corresponding K\"ahler class can never contain an extremal K\"ahler metric with $\mathfrak{E}$ as the extremal vector field up to scaling.

\begin{example}[Eguchi-Hanson]
Recall that the reversed Eguchi-Hanson metric is a Hermitian non-K\"ahler ALE gravitational instanton as in Example \ref{example:eguchihanson1}. The compactified surface $\widehat{M}$ is $H_2=\mathbb{P}(\mathcal{O}\oplus\mathcal{O}(2))$ contracting the curve $C_\infty$. The Picard number of $\widehat{M}$ is 1. Choosing $\mathfrak{K}$ as the holomorphic vector field which  preserves the curve $C_0$ and the curve $C_\infty$, while flows points from $C_\infty$ to $C_0$, our calculation in Section \ref{computefutaki} can be applied. If we assume that $[\omega]$ is the K\"ahler class such that $[\omega](C_0)=1$, it is direct to check by our above formulas that
$${\min_{\mathfrak{E}}}\ s_g([\omega])=0.$$
This shows that, if there is an extremal metric that does not have constant scalar curvature with $\mathfrak{E}$ as the holomorphic extremal vector field,  then the minimum of the scalar curvature is 0 and is achieved at the orbifold point. And indeed, by our correspondence Theorem \ref{main:correspondence}, there is the extremal metric on $\widehat{M}$ whose scalar curvature is nonnegative, and vanishes exactly at the orbifold point.
\end{example}

We will see that the function $\min_\mathfrak{E} s_g$ turns out to be always nonzero over the region of the K\"ahler cone $\mathcal{K}$ where special Bach-flat K\"ahler metrics could exist, for pairs $(\widehat{M},\mathfrak{E})$ in Table \ref{cases}.   It follows that there is no Hermitian non-K\"ahler ALE gravitational instanton with structure group in $SU(2)$, except the reversed Eguchi-Hanson.
To compute $\min_\mathfrak{E} s_g$, we only need to 
\begin{enumerate}
\item use (\ref{ts})-(\ref{t2}) to compute $T$;
\item use (\ref{hh}) to compute $h$;
\item compute $\min_\mathfrak{E} s_g$ using (\ref{mins}).
\end{enumerate}
It suffices to know $\omega(c_{\pm})$, $\omega(E_i)$, $\omega(E_i')$, $\omega(F)$, $c_1[\omega]$, and $c_1(Y)[\Sigma_{\pm a}]$ in each case.
The number of variables one needs to parametrize $\mathcal{K}(\widehat{M})/\mathbb{R}_+$ is the picard number minus one. 
Mathematical software Mathematica will be used to simplify the expressions since the computation is complicated. 
The notations in each figure are adapted to be compatible with our discussion in Section \ref{computefutaki}.
The notation $\sum$ in the following denotes cyclic sum. For example, if there are variables $a,b,c$ and $S$ is the set of permutations $\sigma$ of $a,b,c$, then $\sum a^rb^sc^t:=\sum_{\sigma\in S}\sigma(a)^r\sigma(b)^s\sigma(c)^t$. For example, $\sum a=2(a+b+c)$. The Chern numbers $c_1(Y)[\Sigma_{\pm a}]$ can be computed as in Example \ref{exam:8.2}-\ref{exam:8.3} and we shall list them for each case without details.

\begin{minipage}{0.55\textwidth}

\paragraph{\underline{\underline{Case \hyperref[1A]{1A}}}}

\begin{itemize} 
\item There is only one fixed point $E_1\cap E_1'$ with weights $[-1,1]$.
\item $c_1(Y)[\Sigma_{a}]=-1,c_1(Y)[\Sigma_{-a}]=-2$.
\item Set $\omega(c_+)=1,\omega(E_1')=a$. Then $$\omega(E_1)=(1-a)/2,\ \omega(F)=(1+a)/2,$$  and $\omega(c_-)=0$.
\item $c_1[\omega]=2+a$.
\end{itemize}

\end{minipage}%
\begin{minipage}{0.4\textwidth}
$$
\begin{tikzpicture}
[decoration={markings, 
    mark= at position 0.5 with {\arrow{stealth }}}
]

\draw(0,1.3)--node[left]{$E_1'$}(0,-0.3)node[sloped,pos=0.5,allow upside down]{\arrowIIn};

\draw[postaction={decorate}](2,2.3)--node[right]{$F$}(2,-0.3);

\draw(-0.3,0)--node[below]{$c_+$}(2+0.3,0)node[
    sloped,
    pos=0.5,
    allow upside down]{\arrowBox};

\draw[line width=0.65mm]
(0.7,2)--node[above]{$c_-$}(2+0.3,2)node[
    sloped,
    pos=0.5,
    allow upside down]{\arrowBox};

\draw(1+0.2,2.2)--node[above left]{$E_1$}(0-0.2,0.8)node[sloped,pos=0.5,allow upside down]{\arrowIIn};

\node[left=0.2cm] at (0,1){$[-1,1]$};
\node[below left=0.2cm] at (0,0){$[-1,0]$};
\node[above left=0.2cm] at (1,2){$[1,0]$};
\node[right=0.2cm] at (2,2){$[0,1]$};
\node[below right=0.2cm] at (2,0){$[0,-1]$};

\end{tikzpicture}.
$$
\end{minipage}

Here $a$ must satisfy the bound $0<a<1$.
The function $\min_{\mathfrak{E}} s_g$ is given by
$$\min_{\mathfrak{E}} s_g=-\frac{48 \pi  a \left(a^4-2 a^3-8 a^2+2
   a-1\right)}{3 a^6-18 a^5+3 a^4+12 a^3+9
   a^2+6 a+1},$$
which is positive when $0<a<1$.

\begin{minipage}{0.55\textwidth}

\paragraph{\underline{\underline{Case \hyperref[1B]{1B}}}}

\begin{itemize} 
\item There are two fixed points $E_1\cap E_1',E_2\cap E_2'$, with weights $[-1,1],[-1,1]$.
\item $c_1(Y)[\Sigma_{a}]=0,c_1(Y)[\Sigma_{-a}]=-2$.
\item Set $\omega(c_+)=1,\omega(E_1')=a,\omega(E_2')=b$. Then 
\begin{align*}
    \omega(E_1)=(1+b-a)/2,\ \omega(E_2)=(1+a-b)/2,
\end{align*}
$$\omega(F)=(1+a+b)/2.$$ Moreover $\omega(c_-)=0$.
\item $c_1[\omega]=2+a+b$.
\end{itemize}

\end{minipage}%
\begin{minipage}{0.4\textwidth}
$$
\begin{tikzpicture}
[decoration={markings, 
    mark= at position 0.5 with {\arrow{stealth }}}
]

\draw(0,1.3)--node[left]{$E_1'$}(0,-0.3-0.3)node[sloped,pos=0.5,allow upside down]{\arrowIIn};

\draw(2.3,2.3)--node[right]{$E_2$}(2.3,0.4)node[sloped,pos=0.5,allow upside down]{\arrowIIn};

\draw(-0.3,-0.3)--node[below]{$c_+$}(1+0.6,-0.3)node[
    sloped,
    pos=0.5,
    allow upside down]{\arrowBox};

\draw[line width=0.65mm](0.7,2)--node[above]{$c_-$}(2+0.3+0.3,2)node[
    sloped,
    pos=0.5,
    allow upside down]{\arrowBox};

\draw(1+0.2,2.2)--node[above left]{$E_1$}(0-0.2,0.8)node[sloped,pos=0.5,allow upside down]{\arrowIIn};

\draw(2.5,0.9)--node[below right=0.1cm]{$E_2'$}(1+0.1,-0.5)node[sloped,pos=0.5,allow upside down]{\arrowIIn};

\node[left=0.2cm] at (0,1){$[1,-1]$};
\node[below left=0.2cm] at (0,-0.3){$[-1,0]$};
\node[above left=0.2cm] at (1,2){$[1,0]$};
\node[right=0.2cm] at (2.3,2){$[0,1]$};
\node[below right=0.2cm] at (1.1,-0.3){$[0,-1]$};
\node[right=0.2cm] at (2.3,0.7){$[1,-1]$};

\end{tikzpicture}.
$$
\end{minipage}

Here we have $|a-b|<1$ and $a,b$ are positive. Hence,
\begin{align*}
\quad\min_{\mathfrak{E}} s_g
=&48\pi\left(\sum a-2\sum a^2+8\sum a^3+2\sum a^4-\sum a^5+4\sum ab\right.+12\sum a^3b+3\sum a^4b\\[0.1cm]
&\left.-6\sum a^2b^2-2\sum a^3b^2\right)\Big/\\[0.1cm]
&\left(1+6\sum a+9\sum a^2+12\sum a^3+3\sum a^4-18\sum a^5+3\sum a^6+15\sum ab+36\sum a^2b\right.\\[0.1cm]
&+36 \sum a^3b+6\sum a^4b-18\sum a^5b\left.+9\sum a^2b^2+45\sum a^4b^2+12\sum a^3b^2-30\sum a^3b^3\right)
\end{align*}
Using some elementary inequalities, we have
\begin{itemize}
\item $6\sum a^3b-6\sum a^2b^2\geq 0$.
\item $2\sum a^4b-2\sum a^3b^2\geq0$.
\item $\sum a^4b-\sum a^5=-(a-b)^2(\sum a^3+\sum a^2b)>-\sum a^3-\sum a^2b\geq-2\sum a^3$.
\item $\sum a+\sum a^3\geq2\sum a^2$.
\end{itemize}
Add these together and note the extra term $8\sum a^3$ in the numerator, we get that the numerator is positive. Therefore, $\min_{\mathfrak{E}}s_g\neq0$ on the entire K\"ahler cone $\mathcal{K}$.

\begin{minipage}{0.55\textwidth}

\paragraph{\underline{\underline{Case \hyperref[1C]{1C}}}}

\begin{itemize} 
\item There are three fixed points $E_1\cap E_1',E_2\cap E_2',E_3\cap E_3'$, with weights $[-1,1],[-1,1],[-1,1]$.
\item $c_1(Y)[\Sigma_{a}]=1,c_1(Y)[\Sigma_{-a}]=-2$.
\item Set $\omega(c_+)=1,\omega(E_1')=a,\omega(E_2')=b,\omega(E_3')=c$. Then 
$$\omega(E_1)=(1-a+b+c)/2,\ \omega(E_2)=(1+a-b+c)/2,$$
$$\omega(E_3)=(1+a+b-c)/2,\ \omega(F)=(1+a+b+c)/2.$$ Moreover $\omega(c_-)=0$.
\item $c_1[\omega]=2+a+b+c$.
\end{itemize}

\end{minipage}%
\begin{minipage}{0.4\textwidth}
$$
\begin{tikzpicture}
[decoration={markings, 
    mark= at position 0.5 with {\arrow{stealth }}}
]

\draw(0,1.3)--node[left]{$E_1'$}(0,-0.3-0.3)node[sloped,pos=0.5,allow upside down]{\arrowIIn};

\draw(0+1.3,1.3)--node[left]{$E_2'$}(0+1.3,-0.3-0.3)node[sloped,pos=0.5,allow upside down]{\arrowIIn};

\draw(2.3+1,2.3)--node[right]{$E_3$}(2.3+1,0.4)node[sloped,pos=0.5,allow upside down]{\arrowIIn};

\draw(-0.3,-0.3)--node[below]{$c_+$}(1+0.6+1,-0.3)node[
    sloped,
    pos=0.7,
    allow upside down]{\arrowBox};

\draw[line width=0.65mm](0.7,2)--node[above]{$c_-$}(2+0.3+0.3+1,2)node[
    sloped,
    pos=0.7,
    allow upside down]{\arrowBox};

\draw(1+0.2,2.2)--node[left]{$E_1$}(0-0.2,0.8)node[sloped,pos=0.5,allow upside down]{\arrowIIn};

\draw(1+0.2+1.3,2.2)--node[left]{$E_2$}(0-0.2+1.3,0.8)node[sloped,pos=0.5,allow upside down]{\arrowIIn};

\draw(2.5+1,0.9)--node[below right=0.1cm]{$E_3'$}(1+0.1+1,-0.5)node[sloped,pos=0.5,allow upside down]{\arrowIIn};

\node[left=0.2cm] at (0,1){$[1,-1]$};
\node[below left=0.2cm] at (0,-0.3){$[-1,0]$};
\node[above left=0.2cm] at (1,2){$[1,0]$};
\node[right=0.2cm] at (2.3+1,2){$[0,1]$};
\node[below right=0.2cm] at (1.1+1,-0.3){$[0,-1]$};
\node[right=0.2cm] at (2.3+1,0.7){$[1,-1]$};

\end{tikzpicture}
$$
\end{minipage}

Here we have $a-b-c,-a+b-c,-a-b+c<1$ and $a,b,c$ are positive. Hence,
{ 
\begin{align*}
\min_{\mathfrak{E}} s_g
=&48\pi\Big(\frac12\sum a-\sum a^2+4\sum a^3+\sum a^4-\frac12\sum a^5+4\sum ab+12\sum a^3b+3\sum a^4b\\[0.1cm]
&-6\sum a^2b^2-2\sum a^3b^2+6\sum abc
+6\sum a^2bc-3\sum a^2b^2c+4\sum a^3bc\Big)\Big/\\[0.1cm]
&\Big(1+3\sum a+\frac92\sum a^2+6\sum a^3+\frac32\sum a^4-9\sum a^5+\frac32\sum a^6+15\sum ab+36\sum a^2b\\[0.1cm]
&+36\sum a^3b
+6\sum a^4b-18\sum a^5b+9\sum a^2b^2
+45\sum a^4b^2+12\sum a^3b^2-30\sum a^3b^3\\[0.1cm]
&+18\sum a^2b^2c+20\sum abc
+54\sum a^2bc+36\sum a^3bc+3\sum a^4bc+12\sum a^3b^2c-3\sum a^2b^2c^2\Big).
\end{align*}
}
Elementary inequalities give
\begin{itemize}
\item $3\sum a^3bc-3\sum a^2b^2c\geq0$.
\item $2\sum a^4b-2\sum a^3b^2\geq0$.
\item $\sum a^4b-\frac12\sum a^5=a^4(b+c-a)+b^4(a+c-b)+c^4(a+b-c)>-\frac12\sum a^4$.
\item $6\sum a^3b-6\sum a^2b^2\geq0$.
\item $\frac12\sum a+2\sum a^3-\sum a^2\geq0$.
\end{itemize}
Therefore the numerator is positive. The function $\min_{\mathfrak{E}}s_g\neq0$ on the K\"ahler cone $\mathcal{K}$.

\begin{minipage}{0.55\textwidth}

\paragraph{\underline{\underline{Case \hyperref[2A]{2A}}}}

\begin{itemize} 
\item There are three fixed points, $E_1\cap E_1',E_2\cap E_2',E_3\cap E_3'$, with weights $[-1,1],[-1,1],[-1,1]$.
\item $c_1(Y)[\Sigma_{a}]=1,c_1(Y)[\Sigma_{-a}]=-2$.
\item Set $\omega(E_1')=a,\omega(E_2')=b,\omega(E_3')=1$. Then 
$$\omega(E_1)=(-a+b+1)/2,\ \omega(E_2)=(a-b+1)/2,$$
$$\omega(E_3)=(a+b-1)/2,\ \omega(F)=(a+b+1)/2.$$ 
Moreover $\omega(c_{\pm})=0$.
\item $c_1[\omega]=a+b+1$.
\end{itemize}

\end{minipage}%
\begin{minipage}{0.4\textwidth}
$$
\begin{tikzpicture}
[decoration={markings, 
    mark= at position 0.5 with {\arrow{stealth }}}
]

\draw[shorten >=-0.3cm,shorten <=-0.3cm](0,0)--node[left]{$E_1$}(0,1)node[sloped,pos=0.5,allow upside down]{\arrowIIn};

\draw[shorten >=-0.3cm,shorten <=-0.3cm][postaction={decorate}](3,1)--node[above right]{$E_3'$}(1.5,2.5);

\draw[shorten >=-0.3cm,shorten <=-0.3cm](3,0)--node[right]{$E_3$}(3,1)node[sloped,pos=0.5,allow upside down]{\arrowIIn};

\draw[shorten >=-0.3cm,shorten <=-0.3cm][postaction={decorate}](1,1)--node[right]{$E_2'$}(1.5,2.5);

\draw[shorten >=-0.3cm,shorten <=-0.3cm](1,0)--node[right]{$E_2$}(1,1)node[sloped,pos=0.3,allow upside down]{\arrowIIn};

\draw[line width=0.65mm][shorten >=-0.3cm,shorten <=-0.3cm](0,0)--node[below]{$c_-$}(3,0)node[
    sloped,
    pos=0.5,
    allow upside down]{\arrowBox};

\draw[shorten >=-0.3cm,shorten <=-0.3cm][postaction={decorate}](0,1)--node[above left]{$E_1'$}(1.5,2.5);

\node[left=0.2cm] at (0,1){$[-1,1]$};
\node[below left=0.2cm] at (0,0){$[1,0]$};
\node[above right=0.2cm] at (1.5,2.5){$[-1,-1]$};

\node[above left=0.2cm] at (1.5,2.5){$c_+$};

\end{tikzpicture}
$$
\end{minipage}

Here we have $|a-b|<1$, $a+b>1$, and $a,b>0$.
Hence,
\begin{align*}
\min_{\mathfrak{E}} s_g
=& 16\pi\left(1+\sum a\right)\left(-1+4\sum a-6\sum a^2+4\sum a^3-\sum a^4+4\sum a^3b-3\sum a^2b^2\right)\Big/\\[0.1cm]
&\left(1-6\sum a+15\sum a^2+\sum ab-20\sum a^3+4\sum a^2b+15\sum a^4+4\sum a^3b-3\sum a^2b^2\right.\\[0.1cm]
&\left.-6\sum a^5+2\sum a^4b+4\sum a^3b^2+\sum a^6-6\sum a^5b+15\sum a^4b^2-10\sum a^3b^3\right).
\end{align*}
The numerator can be rewritten as
$$16\pi(1+\sum a)\left(-1+4\sum a-6\sum a^2+4\sum a^3-(a-b)^4\right).$$ 
If either $a$ or $b$ is greater than or equal to one, then we have the following inequality
$$-1+4\sum a-6\sum a^2+4\sum a^3-(a-b)^4\geq -1+2\sum a^2-(a-b)^4>2\sum a^2-2>0.$$
If $a$ and $b$ are both less than one, then
\begin{align*}
    &\quad-1+4\sum a-6\sum a^2+4\sum a^3-\sum a^4+4\sum a^3b-3\sum a^2b^2\\
    &>3\sum a-6\sum a^2+3\sum a^3+4\sum a^3b-3\sum a^2b^2>0.
\end{align*}
Thus, we conclude that $\min_{\mathfrak{E}} s_g\neq0$ on the entire K\"ahler cone $\mathcal{K}$.

\begin{minipage}{0.55\textwidth}

\paragraph{\underline{\underline{Case \hyperref[2D]{2D}}}}

\begin{itemize} 
\item There are three fixed points, $E_1\cap E_1',E_2\cap E_2',E_3\cap E_3'$, with weights $[-2,1],[-1,1],[-2,1]$.
\item $c_1(Y)[\Sigma_{a}]=1,c_1(Y)[\Sigma_{-a}]=-1$.
\item Set $\omega(E_2')=1,\omega(E_1')=a,\omega(E_3')=b$. Then 
$$\omega(E_1)=(2+b-a)/4,\ \omega(E_2)=(a+b)/2,$$
$$\omega(E_3)=(2+a-b)/4,\ \omega(F)=(2+a+b)/2.$$ 
Moreover $\omega(c_{\pm})=0$.
\item $c_1[\omega]=1+a+b$.
\end{itemize}

\end{minipage}%
\begin{minipage}{0.4\textwidth}
$$
\begin{tikzpicture}
[decoration={markings, 
    mark= at position 0.5 with {\arrow{stealth }}}
]

\draw[postaction={decorate}][line width=0.65mm][shorten >=-0.3cm,shorten <=-0.3cm](0,0)--node[left]{$ $}(0,1);

\draw[shorten >=-0.3cm,shorten <=-0.3cm](3.5,2)--node[above right]{$E_3'$}(1.5,3)node[sloped,pos=0.5,allow upside down]{\arrowIIn};

\draw[shorten >=-0.2cm,shorten <=-0.2cm](4,1)--node[right]{$E_3$}(3.5,2)node[sloped,pos=0.3,allow upside down]{\arrowIIn};

\draw[line width=0.65mm][postaction={decorate}][shorten >=-0.3cm,shorten <=-0.3cm](4,0)--node[right]{$ $}(4,1);

\draw[shorten >=-0.3cm,shorten <=-0.3cm][postaction={decorate}](2.5,1)--node[right]{$E_2'$}(1.5,3);

\draw[shorten >=-0.3cm,shorten <=-0.3cm](2.5,0)--node[right]{$E_2$}(2.5,1)node[sloped,pos=0.3,allow upside down]{\arrowIIn};

\draw[line width=0.65mm][shorten >=-0.3cm,shorten <=-0.3cm](0,0)--(4,0)node[
    sloped,
    pos=0.35,
    allow upside down]{\arrowBox};

\draw[shorten >=-0.2cm,shorten <=-0.2cm](0,1)--node[above left]{$E_1$}(0.5,2)node[sloped,pos=0.3,allow upside down]{\arrowIIn};

\draw[shorten >=-0.3cm,shorten <=-0.3cm](0.5,2)--node[above left]{$E_1'$}(1.5,3)node[sloped,pos=0.3,allow upside down]{\arrowIIn};

\node at (-0.5,0) {$c_-$};

\node[left=0.2cm] at (0,1){$[-1,2]$};
\node[below left=0.2cm] at (0,0){$[1,0]$};
\node[below right=0.2cm] at (4,0){$[0,1]$};
\node[above left=0.35cm] at (0.5,1.5){$[-2,1]$};
\node[above left=0.35cm] at (2,3){$[-1,-1]$};
\node[left=0.1cm] at (2.5,1){$[1,-1]$};
\node[above right=0.3cm] at (1.3,3){$c_+$};

\end{tikzpicture}.
$$
\end{minipage}

Here we have the bound $|a-b|<2$ and $a,b>0$. Hence,
\begin{align*}
\min_{\mathfrak{E}} s_g
=&32\pi\left(1+\sum a\right)\left(8\sum a-12\sum a^2+8\sum a^3-\sum a^4+12\sum ab+4\sum a^3b-3\sum a^2b^2\right)\Big/\\[0.1cm]
&\left(32\sum a^2-32\sum a^3+32\sum a^4-12\sum a^5\right.
+\sum a^6+32\sum ab+32\sum a^2b+32\sum a^3b+4\sum a^4b\\[0.1cm]
&\left.-6\sum a^5b+8\sum a^3b^2+15\sum a^4b^2-10\sum a^3b^3
\right).
\end{align*}
Elementary inequalities argument gives
\begin{itemize}
\item $3\sum a^3b-3\sum a^2b^2\geq0$.
\item $\sum a^3b-\sum a^4=-(a-b)^2(\sum a^2+\frac12\sum ab)>-4\sum a^2-2\sum ab$.
\item $8\sum a+8\sum a^3\geq 16\sum a^2$.
\end{itemize}
Add these together we get that the numerator is positive. 
Hence we have $\min_{\mathfrak{E}} s_g\neq0$ on the entire K\"ahler cone $\mathcal{K}$.

\begin{minipage}{0.55\textwidth}

\paragraph{\underline{\underline{Case \hyperref[2E]{2E}}}}

\begin{itemize} 
\item There are three fixed points, $E_1\cap E_1',E_2\cap E_2',E_3\cap E_3'$, with weights $[-2,1],[-2,1],[-2,1]$.
\item $c_1(Y)[\Sigma_{a}]=1,c_1(Y)[\Sigma_{-a}]=-\frac12$.
\item Set $\omega(E_1')=1,\omega(E_2')=a,\omega(E_3')=b$. Then 
$$\omega(E_1)=(a+b)/2,\ \omega(E_2)=(1+b)/2,$$
$$\omega(E_3)=(1+a)/2,\ \omega(F)=1+a+b.$$ Moreover $\omega(c_{\pm})=0$.
\item $c_1[\omega]=1+a+b$.
\end{itemize}

\end{minipage}%
\begin{minipage}{0.4\textwidth}
$$
\begin{tikzpicture}
[decoration={markings, 
    mark= at position 0.5 with {\arrow{stealth }}}
]

\draw[postaction={decorate}][line width=0.65mm][shorten >=-0.3cm,shorten <=-0.3cm](0,0)--node[left]{$ $}(0,1);

\draw[shorten >=-0.3cm,shorten <=-0.3cm](3.5,2)--node[above right]{$E_3'$}(1.5,3)node[sloped,pos=0.3,allow upside down]{\arrowIIn};

\draw[shorten >=-0.2cm,shorten <=-0.2cm](4,1)--node[right]{$E_3$}(3.5,2)node[sloped,pos=0.3,allow upside down]{\arrowIIn};

\draw[line width=0.65mm][postaction={decorate}][shorten >=-0.3cm,shorten <=-0.3cm](4,0)--node[right]{$ $}(4,1);

\draw[shorten >=-0.3cm,shorten <=-0.3cm](2.5,2)--node[below left]{$E_2'$}(1.5,3)node[sloped,pos=0.3,allow upside down]{\arrowIIn};

\draw[shorten >=-0.2cm,shorten <=-0.2cm](3,1)--node[below left]{$E_2$}(2.5,2)node[sloped,pos=0.3,allow upside down]{\arrowIIn};

\draw[postaction={decorate}][line width=0.65mm][shorten >=-0.3cm,shorten <=-0.3cm](3,0)--node[right]{$ $}(3,1);

\draw[line width=0.65mm][shorten >=-0.3cm,shorten <=-0.3cm](0,0)--(4,0)node[
    sloped,
    pos=0.35,
    allow upside down]{\arrowBox};

\draw[shorten >=-0.2cm,shorten <=-0.2cm](0,1)--node[above left]{$E_1$}(0.5,2)node[sloped,pos=0.3,allow upside down]{\arrowIIn};

\draw[shorten >=-0.3cm,shorten <=-0.3cm](0.5,2)--node[above left]{$E_1'$}(1.5,3)node[sloped,pos=0.3,allow upside down]{\arrowIIn};

\node at (-0.5,0) {$c_-$};

\node[left=0.2cm] at (0,1){$[-1,2 ]$};
\node[below left=0.2cm] at (0,0){$[1,0]$};
\node[below right=0.2cm] at (4,0){$[0,1]$};
\node[above left=0.35cm] at (0.5,1.5){$[-2 ,1]$};
\node[above left=0.35cm] at (2,3){$[-1,-1]$};
\node[above right=0.3cm] at (1.3,3){$c_+$};

\end{tikzpicture}.
$$
\end{minipage}

Here $a,b>0$. Hence,
\begin{align*}
\min_{\mathfrak{E}} s_g
=& 8\pi\left(1+\sum a\right)\left(\sum a
+\sum a^3+3\sum ab+6\sum a^2b+\sum a^3b\right)\Big/\\[0.1cm]
&\left(\sum a^2+\sum a^4+\sum ab+4\sum a^2b+4\sum a^3b+2\sum a^4b+3\sum a^2b^2+4\sum a^3b^2+\sum a^4b^2
\right).
\end{align*}
Clearly this is positive. Therefore, $\min_{\mathfrak{E}} s_g>0$ on the entire K\"ahler cone.

\begin{minipage}{0.55\textwidth}

\paragraph{\underline{\underline{Case \hyperref[3Aproblem]{3A}}}}

\begin{itemize} 
\item There are two fixed points $E_1\cap E_1',E_2\cap E_2'$, with weights $[-2,1],[-1,1]$.
\item $c_1(Y)[\Sigma_{a}]=\frac12,c_1(Y)[\Sigma_{-a}]=-1$.
\item Set $\omega(E_3)=1,\omega(E_2)=a$. Then 
$$\omega(E_1)=1-a,\ \omega(E_1')=2a,$$
$$\omega(E_2')=2-a,\ \omega(F)=2.$$ Moreover $\omega(c_{\pm})=0$.
\item $c_1[\omega]=2+a$.
\end{itemize}

\end{minipage}%
\begin{minipage}{0.4\textwidth}
$$
\begin{tikzpicture}
[decoration={markings, 
    mark= at position 0.5 with {\arrow{stealth }}}
]
\draw[line width=0.65mm][postaction={decorate}](0,1.3)--node[left]{}(0,-0.3);
\draw[postaction={decorate}](2+1,2.3)--node[right]{$E_3$}(2+1,-0.3-1);
\draw[postaction={decorate}](-0.3+1,0-1)--node[below]{$E_1'$}(2+0.3+1,0-1);
\draw[line width=0.65mm][postaction={decorate}](0.7,2)--node[above]{}(2+0.3+1,2);

\draw[postaction={decorate}](1.3,0.3) to[out=10,in=165] node[below]{$E_2'$}(3,-1);

\draw[shorten >=-0.5cm,shorten <=-0.5cm][line width=0.65mm](0-0.2,0.8)--node[above left]{$c_-$}(1+0.2,2.2)node[
    sloped,
    pos=0.5,
    allow upside down]{\arrowBox};

\draw(0-0.2,0+0.2)--node[below]{$E_1$}(1+0.2,-1-0.2)node[
    sloped,
    pos=0.5,
    allow upside down]{\arrowIIn};

\draw[shorten >=-0.3cm,shorten <=-0.3cm](0.3,1.3)--node[below]{$E_2$}(1.3,0.3)node[sloped,pos=0.3,allow upside down]{\arrowIIn};

\node[below right] at (3,-1){$c_+$};

\node[left=0.2cm] at (0,1){$[1,0]$};
\node[below left=0.2cm] at (0,0){$[-1,2]$};
\node[above left=0.2cm] at (1,2){$[0,1]$};
\node[right=0.2cm] at (2+1,2){$[-1,2]$};
\node[below right=0.3cm] at (2+1,-1){$[-1,-2]$};
\node[below=0.2cm] at (1,-1){$[-2,1]$};
\node[above right] at (1.3,0.3){$[-1,1]$};
\end{tikzpicture}
$$
\end{minipage}

Here $0<a<1$. Hence,
$$\min_{\mathfrak{E}} s_g=-\frac{8 \pi  (a+2) \left(9 a^3-26 a^2+24
   a-8\right)}{a \left(21 a^4-84 a^3+128
   a^2-96 a+32\right)}.$$
It is easy to verify that this is positive when $0<a<1$. Therefore, $\min_{\mathfrak{E}} s_g>0$ on the entire K\"ahler cone $\mathcal{K}$.

\begin{minipage}{0.55\textwidth}
\paragraph{\underline{\underline{Case \hyperref[3C]{3C}}}}

\begin{itemize} 
\item There are two fixed points $E_1\cap E_1',E_2\cap E_2'$, with weights $[-2,1],[-2,1]$.
\item $c_1(Y)[\Sigma_{a}]=\frac12,c_1(Y)[\Sigma_{-a}]=-\frac12$. 
\item Set $\omega(E_3)=1,\omega(E_2)=a$. Then
$$\omega(E_1)=1-a,\ \omega(E_1')=2a,$$
$$\omega(E_2')=2-2a,\ \omega(F)=2.$$ 
Moreover $\omega(c_{\pm})=0$.
\item $c_1[\omega]=2$.
\end{itemize}
\end{minipage}%
\begin{minipage}{0.4\textwidth}
$$
\begin{tikzpicture}
[decoration={markings, 
    mark= at position 0.5 with {\arrow{stealth }}}
]
\draw[line width=0.65mm][postaction={decorate}](0,1.3)--node[left]{}(0,-0.3-1);
\draw[postaction={decorate}](2+1+1,2.3)--node[right]{$E_3$}(2+1+1,-0.3-1-1);
\draw[postaction={decorate}](-0.3+1,0-1-1)--node[below]{$E_1'$}(2+0.3+1+1,0-1-1);
\draw[line width=0.65mm][postaction={decorate}](0.7,2)--node[above]{}(2+0.3+1+1,2);

\draw[line width=0.65mm][shorten >=-0.6cm,shorten <=-0.6cm](0-0.2,0.8)--node[above left]{$c_-$}(1+0.2,2.2)node[
    sloped,
    pos=0.3,
    allow upside down]{\arrowBox};

\draw(0-0.2,0+0.2-1)--node[below left]{$E_1$}(1+0.2,-1-0.2-1)node[sloped,pos=0.5,allow upside down]{\arrowIIn};

\draw[line width=0.65mm][postaction={decorate}](0.5-0.1,1.5+0.1)--node[below left]{}(1.5+0.2,0.5-0.2);

\draw(1.5-0.2,0.5)--node[below]{$E_2$}(1.5+1+0.2,0.5)node[sloped,pos=0.5,allow upside down]{\arrowIIn};

\draw[postaction={decorate}](2.5,0.5) to[out=-50,in=180] node[left]{$E_2'$}(3+1,-1-1);

\node[left=0.2cm] at (0,1){$[1,0]$};
\node[below left=0.2cm] at (0,0-1){$[-1,2]$};
\node[above left=0.2cm] at (1,2){$[0,1]$};
\node[right=0.2cm] at (2+1+1,2){$[-1,2]$};
\node[below right=0.3cm] at (2+1+1,-1-1)
{$[-1,-2]$};
\node[below=0.2cm] at (1,-1-1){$[-2,1]$};
\node[below left=0.06cm] at (1.5,0.5){$[-1,2]$};
\node[above right=0.06cm] at (1.5+1,0.5) {$[-2,1]$};
\node[below right=0.06cm] at (2+1+1,-1-1)
{$c_+$};

\end{tikzpicture}
$$
\end{minipage}

Here $0<a<1$. Hence,
$$\min_{\mathfrak{E}} s_g=\frac{4\pi}{a-a^2}.$$
Clearly this is positive. Therefore $\min_{\mathfrak{E}} s_g>0$ on the entire K\"ahler cone $\mathcal{K}$.

\begin{minipage}{0.5\textwidth}
\paragraph{\underline{\underline{Case \hyperref[3E]{3E}}}}

\begin{itemize} 
\item There are two fixed points $E_1\cap E_1',E_2\cap E_2'$, with weights $[-2,1],[-3,1]$.
\item $c_1(Y)[\Sigma_{a}]=\frac12,c_1(Y)[\Sigma_{-a}]=-\frac13$.
\item Set $\omega(E_3)=1,\omega(E_2)=a$. Then 
$$\omega(E_1)=1-a,\ \omega(E_1')=2a,$$
$$\omega(E_2')=2-3a,\ \omega(F)=2.$$ 
Moreover $\omega(c_{\pm})=0$.
\item $c_1[\omega]=2-a$.
\end{itemize}
\end{minipage}%
\begin{minipage}{0.45\textwidth}
$$
\scalebox{0.9}{
\begin{tikzpicture}
[decoration={markings, 
    mark= at position 0.5 with {\arrow{stealth }}}
]
\draw[line width=0.65mm][postaction={decorate}](0,1.3)--node[left]{}(0,-0.3-1-1);
\draw[postaction={decorate}](2+1+1+1,2.3)--node[right]{$E_3$}(2+1+1+1,-0.3-1-1-1);
\draw[postaction={decorate}](-0.3+1,0-1-1-1)--node[below]{$E_1'$}(2+0.3+1+1+1,0-1-1-1);
\draw[line width=0.65mm][postaction={decorate}](0.7,2)--node[above]{}(2+0.3+1+1+1,2);
\draw[line width=0.65mm][shorten >=-0.6cm,shorten <=-0.6cm](0-0.2,0.8)--node[above left]{$c_-$}(1+0.2,2.2)node[
    sloped,
    pos=0.3,
    allow upside down]{\arrowBox};

\draw(0-0.2,0+0.2-1-1)--node[below]{$E_1$}(1+0.2,-1-0.2-1-1)node[sloped,pos=0.5,allow upside down]{\arrowIIn};

\draw[line width=0.65mm][postaction={decorate}](0.5-0.1,1.5+0.1)--node[below left]{}(1.5+0.2,0.5-0.2);
\draw[line width=0.65mm][postaction={decorate}](1.5-0.2,0.5)--node[below]{}(1.5+1+0.2,0.5);

\draw(1.5+1,0.5+0.2)--node[right]{$E_2$}(1.5+1,0.5-1-0.2)node[sloped,pos=0.5,allow upside down]{\arrowIIn};


\draw(2.5,0.5-1) to[out=-10,in=180] node[below left]{$E_2'$}(3+1+1,-1-1-1);
\draw[line width=0.00001mm](2.5,0.5-1)  to[out=-10,in=180]  node[sloped,pos=0.5,allow upside down]{\arrowIIn}(3+1+1,-1-1-1);

\node[left=0.2cm] at (0,1){$[1,0]$};
\node[left=0.2cm] at (0,-2){$[-1,2]$};
\node[above left=0.2cm] at (1,2){$[0,1]$};
\node[right=0.2cm] at (2+1+1+1,2){$[-1,2]$};
\node[below right=0.3cm] at (2+1+1+1,-1-1-1)
{$[-1,-2]$};
\node[below=0.2cm] at (1,-1-1-1){$[-2,1]$};
\node[below left=0.06cm] at (1.5,0.5){$[-1,2]$};
\node[above right=0.06cm] at (1.5+1,0.5) {$[-2,3]$};
\node[left=0.06cm] at (1.5+1,0.5-1) {$[-3,1]$};
\node[below right=0.06cm] at (2+1+1+1,-1-1-1)
{$c_+$};

\end{tikzpicture}
}
$$
\end{minipage}

Here $0<a<\frac23$. Hence,
$$\min_{\mathfrak{E}} s_g=\frac{24 \pi  (a-2) \left(35 a^3-42 a^2+24
   a-8\right)}{a \left(475 a^4-1140 a^3+1056
   a^2-480 a+96\right)}.$$
It is easy to verify that this is positive when $0<a<\frac23$. Hence $\min_{\mathfrak{E}} s_g>0$ on the entire K\"ahler cone $\mathcal{K}$.

\begin{minipage}{0.5\textwidth}

\paragraph{\underline{\underline{Case \hyperref[3G]{3G}}}}

\begin{itemize} 
\item There are two fixed points $E_1\cap E_1',E_2\cap E_2'$, with weights $[-3,1],[-3,1]$.
\item $c_1(Y)[\Sigma_{a}]=\frac12,c_1(Y)[\Sigma_{-a}]=-\frac16$.
\item Set $\omega(E_3)=1,\omega(E_2)=a$. Then 
$$\omega(E_1)=1-a,\ \omega(E_1')=3a-1,$$
$$\omega(E_2')=2-3a,\ \omega(F)=2.$$ Moreover $\omega(c_{\pm})=0$.
\item $c_1[\omega]=1$.
\end{itemize}

\end{minipage}%
\begin{minipage}{0.45\textwidth}
$$
\scalebox{0.9}{
\begin{tikzpicture}
[decoration={markings, 
    mark= at position 0.5 with {\arrow{stealth }}}
]
\draw[line width=0.65mm][postaction={decorate}](0,1.3)--node[left]{}(0,-0.3-1);
\draw[postaction={decorate}](2+1+1+1,2.3)--node[right]{$E_3$}(2+1+1+1,-0.3-1-1-1);

\draw(-0.3+1+1,0-1-1-1)--node[below]{$E_1'$}(2+0.3+1+1+1,0-1-1-1)node[sloped,pos=0.3,allow upside down]{\arrowIIn};

\draw[line width=0.65mm][postaction={decorate}](0.7,2)--node[above]{}(2+0.3+1+1+1,2);
\draw[line width=0.65mm][shorten >=-0.6cm,shorten <=-0.6cm](0-0.2,0.8)--node[above left]{$c_-$}(1+0.2,2.2)node[
    sloped,
    pos=0.3,
    allow upside down]{\arrowBox};

\draw[line width=0.65mm][shorten >=-0.2cm,shorten <=-0.2cm][postaction={decorate}](0,0-1)--node[left]{}(1-0.5,-2-0.5);

\draw[shorten >=-0.2cm,shorten <=-0.2cm](0.5,-2.5)--node[below]{$E_1$}(2,-3)node[sloped,pos=0.3,allow upside down]{\arrowIIn};

\draw[line width=0.65mm][postaction={decorate}](0.5-0.1,1.5+0.1)--node[below left]{}(1.5+0.2,0.5-0.2);

\draw[line width=0.65mm][postaction={decorate}](1.5-0.2,0.5)--node[below]{}(1.5+1+0.2,0.5);


\draw[shorten >=-0.2cm,shorten <=-0.2cm](2.5,0.5)--node[right]{$E_2$}(2.5,0.5-1)node[sloped,pos=0.3,allow upside down]{\arrowIIn};

\draw(2.5,0.5-1) to[out=10,in=180] node[below left]{$E_2'$}node[sloped,pos=0.5,allow upside down]{\arrowIIn}(3+1+1,-1-1-1);

\node[left=0.2cm] at (0,1){$[1,0]$};
\node[left=0.2cm] at (0,-1){$[-1,2]$};
\node[above left=0.2cm] at (1,2){$[0,1]$};
\node[right=0.2cm] at (2+1+1+1,2){$[-1,2]$};
\node[below right=0.3cm] at (2+1+1+1,-1-1-1)
{$[-1,-2]$};
\node[below=0.2cm] at (2,-1-1-1){$[-3,1]$};
\node[below left=0.06cm] at (1.5,0.5){$[-1,2]$};
\node[above right=0.06cm] at (1.5+1,0.5) {$[-2,3]$};
\node[left=0.06cm] at (1.5+1,0.5-1) {$[-3,1]$};
\node[left=0.06cm] at (0.5,0.5-3) {$[-2,3]$};
\node[below right=0.06cm] at (2+1+1+1,-1-1-1)
{$c_+$};

\end{tikzpicture}
}
$$
\end{minipage}

Here $\frac{1}{3}<a<\frac23$. Hence,
$$\min_{\mathfrak{E}} s_g=-\frac{216 \pi  \left(6 a^4-12 a^3+12 a^2-6
a+1\right)}{972 a^6-2916 a^5+3564 a^4-2268
a^3+774 a^2-126 a+7}.$$
It is easy to verify that this is positive when $\frac13<a<\frac23$. Hence $\min_{\mathfrak{E}} s_g>0$ on the entire K\"ahler cone.

For the following cases the picard number of $\widehat{M}$ is 1, hence only one K\"ahler class exists on the orbifold.

\begin{minipage}{0.5\textwidth}

    \paragraph{\underline{\underline{Case \hyperref[3H]{3H}}}}
        \begin{itemize} 
            \item There is one fixed point $E_2\cap E_2'$ with weights $[-2,1]$.
            \item $c_1(Y)[\Sigma_a]=\frac{1}{6},c_1(Y)[\Sigma_{-a}]=-\frac{1}{3}$.
            \item Set $\omega(E_2)=1$. Then
                     $$\omega(E_1)=3,\ \omega(E_2')=4,\ \omega(E_3)=2,\ \omega(F)=6.$$
                Moreover, $\omega(c_{\pm})=0$.
            \item $c_1[\omega]=4$.
        \end{itemize}
        \end{minipage}%
        \begin{minipage}{0.45\textwidth}
        $$
        \scalebox{0.9}{
        \begin{tikzpicture}
        [decoration={markings, 
            mark= at position 0.5 with {\arrow{stealth }}}
        ]
        \draw[line width=0.65mm][postaction={decorate}](0,1.3-1)--(0,-0.3-1-2);
        \draw[postaction={decorate}](2+1+2,2.3)--node[right]{$E_3$}(2+1+2,-0.3-1-2);
        \draw[postaction={decorate}](-0.3,0-1-2)--node[below]{$E_1$}(2+0.3+1+2,0-1-2);
        \draw[line width=0.65mm][postaction={decorate}](0.7+1,2)--(2+0.3+1+2,2);
        \draw[line width=0.65mm][shorten >=-0.2cm,shorten <=-0.2cm][postaction={decorate}](1-0.5,1+0.5)--(2,2);
        \draw[line width=0.65mm][shorten >=-0.6cm,shorten <=-0.6cm](0,0)--(0.5,1.5)node[
            sloped,
            pos=0.3,
            allow upside down]{\arrowBox};

        \draw[line width=0.65mm][postaction={decorate}][shorten >=-0.2cm,shorten <=-0.2cm](0.25,0.75)--(1,0.5);

        \draw[shorten >=-0.2cm,shorten <=-0.2cm](1,0.5)--node[right]{$E_2$}(1,0.5-1)node[sloped,pos=0.3,allow upside down]{\arrowIIn};

        \draw[postaction={decorate}](1,0.5-1) to[out=10,in=180] node[below]{$E_2'$}(3+2,-1-2);

        \node[left=0.2cm] at (0,1-1){$[1,0]$};
        \node[below left=0.2cm] at (0,-3){$[-1,2]$};
        \node[above=0.3cm] at (1+1,2){$[-1,2]$};
        \node[right=0.2cm] at (5,2.2){$[-2,3]$};
        \node[below right=0.4cm] at (5,-3){$[-2,-3]$};
        \node[below right=0.1cm] at (5,-3){$c_+$};
        \node[above left=0.2cm] at (0.5,1.5){$[0,1]$};
        \node at (1.7,1) {$[-1,2]$};
        \node at (1.2,-1) {$[-2,1]$};
        \node at (0,1) {$c_-$};
        
        \end{tikzpicture}
        }
        $$
\end{minipage}

We have $\min_{\mathfrak{E}}\ s_g=\frac{24\pi}{7}>0$.

\begin{minipage}{0.5\textwidth}

    \paragraph{\underline{\underline{Case \hyperref[3J]{3J}}}}

    \begin{itemize}
        \item There is one fixed point $E_2\cap E_2'$ with weights $[-3,1]$. 
        \item $c_1(Y)[\Sigma_a]=\frac{1}{6},c_1(Y)[\Sigma_{-a}]=-\frac{1}{6}$.
        \item Set $\omega(E_2)=1$. Then 
            $$\omega(E_1)=3,\ \omega(E_2')=3,\ \omega(E_3)=2,\ \omega(F)=6.$$
            Moreover, $\omega(c_{\pm})=0$.
        \item $c_1[\omega]=3$.
    \end{itemize}

    \end{minipage}%
    \begin{minipage}{0.45\textwidth}
    $$
    \scalebox{0.9}{
    \begin{tikzpicture}
    [decoration={markings, 
        mark= at position 0.5 with {\arrow{stealth }}}
    ]
    \draw[line width=0.65mm][postaction={decorate}](0,1.3-1)--(0,-0.3-1-2);
    \draw[postaction={decorate}](2+1+2,2.3)--node[right]{$E_3$}(2+1+2,-0.3-1-2);
    \draw[postaction={decorate}](-0.3,0-1-2)--node[below]{$E_1$}(2+0.3+1+2,0-1-2);
    \draw[line width=0.65mm][postaction={decorate}](0.7+1,2)--(2+0.3+1+2,2);
    \draw[line width=0.65mm][shorten >=-0.2cm,shorten <=-0.2cm][postaction={decorate}](1-0.5,1+0.5)--(2,2);
    \draw[line width=0.65mm][shorten >=-0.6cm,shorten <=-0.6cm](0,0)--(0.5,1.5)node[
        sloped,
        pos=0.3,
        allow upside down]{\arrowBox};

    \draw[line width=0.65mm][postaction={decorate}][shorten >=-0.2cm,shorten <=-0.2cm](0.25,0.75)--(1,0.5);
    \draw[line width=0.65mm][postaction={decorate}][shorten >=-0.2cm,shorten <=-0.2cm](1,0.5)--(1,0.5-1);

    \draw[shorten >=-0.2cm,shorten <=-0.2cm](1,0.5-1)--node[below]{$E_2$}(1+1,0.5-1)node[sloped,pos=0.3,allow upside down]{\arrowIIn};

    \draw[postaction={decorate}](1+1,0.5-1) to[out=-90,in=180] node[below]{$E_2'$}(3+2,-1-2);

    \node[left=0.2cm] at (0,1-1){$[1,0]$};
    \node[below left=0.2cm] at (0,-3){$[-1,2]$};
    \node[above=0.3cm] at (1+1,2){$[-1,2]$};
    \node[right=0.2cm] at (5,2.2){$[-2,3]$};
    \node[below right=0.4cm] at (5,-3){$[-2,-3]$};
    \node[below right=0.1cm] at (5,-3){$c_+$};
    \node[above left=0.2cm] at (0.5,1.5){$[0,1]$};
    \node at (1.7,1) {$[-1,2]$};
    \node at (1,-1.3) {$[-2,3]$};
    \node at (2.7,-0.4) {$[-3,1]$};
    \node at (0,1) {$c_-$};
    
    \end{tikzpicture}
    }
    $$
\end{minipage}
    
In this case we hace $-T_s+s_0T_1=0$. Therefore it can never support an extremal K\"ahler metric with extremal vector field proportional to $\mathfrak{E}$, by our discussion at the beginning of Section \ref{sec:computation}.

\begin{minipage}{0.5\textwidth}
    \paragraph{\underline{\underline{Case \hyperref[3L]{3L}}}}

    \begin{itemize}
        \item There is one fixed point $E_2\cap E_2'$ with weights $[-4,1]$.
        \item $c_1(Y)[\Sigma_a]=\frac{1}{6},c_1(Y)[\Sigma_{-a}]=-\frac{1}{12}$. 
        \item Set $\omega(E_2)=1$. Then 
            $$\omega(E_1)=3,\ \omega(E_2')=2,\ \omega(E_3)=2,\ \omega(F)=6.$$
        Moreover, $\omega(c_{\pm})=0$.
        \item $c_1[\omega]=2$.
    \end{itemize}

    \end{minipage}%
    \begin{minipage}{0.45\textwidth}
    $$
    \scalebox{0.9}{
    \begin{tikzpicture}
    [decoration={markings, 
        mark= at position 0.5 with {\arrow{stealth }}}
    ]
    \draw[line width=0.65mm][postaction={decorate}](0,1.3-1)--(0,-0.3-1-2);
    \draw[postaction={decorate}](2+1+2,2.3)--node[right]{$E_3$}(2+1+2,-0.3-1-2);
    \draw[postaction={decorate}](-0.3,0-1-2)--node[below]{$E_1$}(2+0.3+1+2,0-1-2);
    \draw[line width=0.65mm][postaction={decorate}](0.7+1,2)--(2+0.3+1+2,2);
    \draw[line width=0.65mm][shorten >=-0.2cm,shorten <=-0.2cm][postaction={decorate}](1-0.5,1+0.5)--(2,2);
    \draw[line width=0.65mm][shorten >=-0.6cm,shorten <=-0.6cm](0,0)--(0.5,1.5)node[
        sloped,
        pos=0.3,
        allow upside down]{\arrowBox};

    \draw[line width=0.65mm][postaction={decorate}][shorten >=-0.2cm,shorten <=-0.2cm](0.25,0.75)--(1,0.5);
    \draw[line width=0.65mm][shorten >=-0.2cm,shorten <=-0.2cm][postaction={decorate}](1,0.5)--(1,0.5-1);
    \draw[line width=0.65mm][shorten >=-0.2cm,shorten <=-0.2cm][postaction={decorate}](1,0.5-1)--(1+1,0.5-1);

    \draw[shorten >=-0.2cm,shorten <=-0.2cm](1+1,0.5-1)--node[right]{$E_2$}(1+1,0.5-1-1)node[sloped,pos=0.3,allow upside down]{\arrowIIn};

    \draw[postaction={decorate}] (1+1,0.5-1-1) to[out=10,in=180] node[below]{$E_2'$}(3+2,-1-2);

    \node[left=0.2cm] at (0,1-1){$[1,0]$};
    \node[below left=0.2cm] at (0,-3){$[-1,2]$};
    \node[above=0.3cm] at (1+1,2){$[-1,2]$};
    \node[right=0.2cm] at (5,2.2){$[-2,3]$};
    \node[below right=0.4cm] at (5,-3){$[-2,-3]$};
    \node[below right=0.1cm] at (5,-3){$c_+$};
    \node[above left=0.2cm] at (0.5,1.5){$[0,1]$};
    \node at (1.7,1) {$[-1,2]$};
    \node at (1.1,-1.2) {$[-2,3]$};
    \node at (2.7,-0.4) {$[-3,4]$};
    \node at (2.1,-2) {$[-4,1]$};
    \node at (0,1) {$c_-$};

    \end{tikzpicture}
    }
    $$
    \end{minipage}

    We have $\min_{\mathfrak{E}}\ s_g=\frac{96\pi}{7}>0$.

    \begin{minipage}{0.5\textwidth}
        \paragraph{\underline{\underline{Case \hyperref[3M]{3M}}}}

        \begin{itemize}
            \item There is one fixed point $E_2\cap E_2'$ with weights $[-5,1]$.
            \item $c_1(Y)[\Sigma_a]=\frac{1}{6},c_1(Y)[\Sigma_{-a}]=-\frac{1}{30}$. 
            \item Set $\omega(E_2)=1$. Then 
                $$\omega(E_1)=3,\ \omega(E_2')=1,\ \omega(E_3)=2,\ \omega(F)=6.$$
            Moreover, $\omega(c_{\pm})=0$.
            \item $c_1[\omega]=1$.
        \end{itemize}    
        \end{minipage}%
        \begin{minipage}{0.45\textwidth}
        $$
        \scalebox{0.9}{
        \begin{tikzpicture}
        [decoration={markings, 
            mark= at position 0.5 with {\arrow{stealth }}}
        ]
        \draw[line width=0.65mm][postaction={decorate}](0,1.3-1)--(0,-0.3-1-2);
        \draw[postaction={decorate}](2+1+2,2.3)--node[right]{$E_3$}(2+1+2,-0.3-1-2);
        \draw[postaction={decorate}](-0.3,0-1-2)--node[below]{$E_1$}(2+0.3+1+2,0-1-2);
        \draw[line width=0.65mm][postaction={decorate}](0.7+1,2)--(2+0.3+1+2,2);
        \draw[line width=0.65mm][shorten >=-0.2cm,shorten <=-0.2cm][postaction={decorate}](1-0.5,1+0.5)--(2,2);
        \draw[line width=0.65mm][shorten >=-0.6cm,shorten <=-0.6cm](0,0)--(0.5,1.5)node[
            sloped,
            pos=0.3,
            allow upside down]{\arrowBox};

        \draw[line width=0.65mm][postaction={decorate}][shorten >=-0.2cm,shorten <=-0.2cm](0.25,0.75)--(1,0.5);
        \draw[line width=0.65mm][shorten >=-0.2cm,shorten <=-0.2cm][postaction={decorate}](1,0.5)--(1,0.5-1);
        \draw[line width=0.65mm][shorten >=-0.2cm,shorten <=-0.2cm][postaction={decorate}](1,0.5-1)--(1+1,0.5-1);
        \draw[line width=0.65mm][shorten >=-0.2cm,shorten <=-0.2cm][postaction={decorate}](1+1,0.5-1)--(1+1,0.5-1-1);

        \draw[shorten >=-0.2cm,shorten <=-0.2cm](1+1,0.5-1-1)--node[below]{$E_2$}(1+1+1,0.5-1-1)node[sloped,pos=0.5,allow upside down]{\arrowIIn};

        \draw[postaction={decorate}] (1+1+1,0.5-1-1) to[out=-90,in=180] node[below left]{$E_2'$}(3+2,-1-2);

        \node[left=0.2cm] at (0,1-1){$[1,0]$};
        \node[below left=0.2cm] at (0,-3){$[-1,2]$};
        \node[above=0.3cm] at (1+1,2){$[-1,2]$};
        \node[right=0.2cm] at (5,2.2){$[-2,3]$};
        \node[below right=0.4cm] at (5,-3){$[-2,-3]$};
        \node[below right=0.1cm] at (5,-3){$c_+$};
        \node[above left=0.2cm] at (0.5,1.5){$[0,1]$};
        \node at (1.7,1) {$[-1,2]$};
        \node at (1.1,-1.2) {$[-2,3]$};
        \node at (2.7,-0.4) {$[-3,4]$};
        \node at (2.1,-2.2) {$[-4,5]$};
        \node at (4,-1.6 ) {$[-5,1]$};
        \node at (0,1) {$c_-$};
        
        \end{tikzpicture}
        }
        $$
        \end{minipage}
        
    We have $\min_{\mathfrak{E}}s_g=\frac{600\pi}{31}>0$.

Till now, we have already proved that, 
\begin{theorem}
For the pairs $(\widehat{M},\mathfrak{E})$ listed in Table \ref{cases}, they can never support special Bach-flat K\"ahler metrics. 
\end{theorem}

Combining Theorem \ref{main:correspondence}, Proposition \ref{prop:added}, and Proposition \ref{prop:complextype} with the results obtained in this section, we have completed the proof of Theorem \ref{main:classificationsu2}.

It is worth noting that in higher dimensions, there also exist Hermitian non-K\"ahler Ricci-flat ALE manifolds. For instance, consider a Kronheimer's hyperk\"ahler ALE 4-manifold $X$, where the space of hyperk\"ahler complex structures is parameterized by $S^2$. Take $\mathbb{C}\subset S^2$ and consider $\mathbb{C}\times X$. Define the complex structure $J$ on $\mathbb{C}\times X$ by 
$$J|_{X_t}=J_t,$$
where here $X_t$ refers to $t\times X$, and $J_t$ refers to the hyperk\"ahler complex structure on $X$ given by $t\in\mathbb{C}\subset S^2$. For any $p\in X$, we require that $J|_{\mathbb{C}\times p}$ is the standard complex structure on $\mathbb{C}$. This defines an integrable complex structure on $\mathbb{C}\times X$, and the product metric on $\mathbb{C}\times X$ is clearly Hermitian non-K\"ahler Ricci-flat ALE. However, this example is trivial in the sense that there are complex structures on $\mathbb{C}\times X$, namely the product complex structures, such that the metric is K\"ahler. This example was pointed out to the author by Junsheng Zhang. It will be interesting to find more nontrivial Hermtian non-K\"ahler Ricci-flat ALE manifolds in the higher dimensional case. 

Finally we would like to remark that given any $\Gamma\subset U(2)$, we could repeat the classification for the pairs $(\widetilde{M},\mathfrak{E})$ and the calculation for $\min_{\mathfrak{E}}s_g$, to derive some non-existence results for Hermitian non-K\"ahler ALE gravitational instantons with structure group $\Gamma$. The difficulty for the general $U(2)$ case is, if we do not fix $\Gamma\subset U(2)$, there could be infinite number of candidates $(\widetilde{M},\mathfrak{E})$.

\appendix

\section{Proof of Theorem \ref{thm:cyclicweights}}

In the appendix we give a proof of Theorem \ref{thm:cyclicweights}. We adopt some notations from toric geometry. For a cyclic group $\Gamma=L(q,p)$ and the fan $F_{L(q,p)}$ of its minimal resolution as shown in Figure \ref{fanminimal}, we set each $v_i=(v_{i,U},v_{i,V})$ and define $v_i^\perp=(-v_{i,V},v_{i,U})$. Denote $\sigma_i$ by the cone $\sigma(v_{i+1},v_i)$ spanned by $v_{i+1},v_i$, so that the fan $F_{L(q,p)}$ is the union of the cones $\sigma_i$. For each cone $\sigma_i=\sigma(v_{i+1},v_i)$ in the fan, there is the dual cone $\sigma_i^\vee=\sigma(-v_{i+1}^\perp,v_i^\perp)$. Denote $\sigma_{\Gamma}$ by the cone $\sigma((0,1),(p,-q))$, hence the fan of $\mathbb{C}^2/\Gamma$ consists of a single cone $\sigma_{\Gamma}$. For a cone $\sigma$, by $S_{\sigma}$ we mean the $\mathbb{C}$-algebra generated by $U^aV^b$ with all integral $(a,b)\in\sigma$. From standard results of toric geometry we have:
\begin{itemize}
\item As a variety $\mathbb{C}^2/\Gamma=\Spec\mathbb{C}[X,Y]^\Gamma=\Spec S_{\sigma_\Gamma^\vee}$. Here, $\mathbb{C}[X,Y]^\Gamma$ refers to the polynomials that are $\Gamma$-invariant.  The $\mathbb{C}$-algebra $S_{\sigma_\Gamma^\vee}$ is related to $\mathbb{C}[X,Y]^\Gamma$ via the isomorphism $U=X^p$ and $V=Y/X^q$.
\item The minimal resolution $\widetilde{\mathbb{C}^2/\Gamma}$ is the affine varieties $\Spec S_{\sigma_i^\vee}$ glued together along $\Spec S_{\sigma_i^\vee\cup\sigma_{i-1}^\vee}$. Notice that $S_{\sigma_i^\vee}\subset S_{\sigma_i^\vee\cup\sigma_{i-1}^\vee}$ is a sub-$\mathbb{C}$-algebra.
\end{itemize}

\begin{theorem}
Consider the $\mathbb{C}^*$-action on $\mathbb{C}^2/L(q,p)$ defined by $t\curvearrowright(x,y)=(t^\theta x,t^\tau y)$ with $\theta\geq\tau\geq0$ and its minimal resolution $\widetilde{\mathbb{C}^2/\Gamma}$ given by Theorem \ref{toricresolution}.
Given vertices $v_0,\ldots,v_k$ as in Theorem \ref{toricresolution}, write $v_i=(v_{i,U},v_{i,V})$ and for $i=0,\ldots,k+1$ define
$$w_i=\theta (pv_{i,V}+qv_{i,U})-\tau v_{i,U}.$$
Then the weights of the lifted action are:
$$
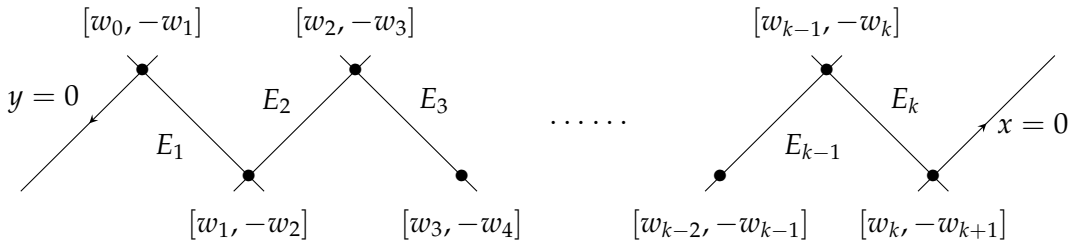
\begin{tikzpicture}
[decoration={markings, 
    mark= at position 0.5 with {\arrow{stealth}}}
]
\draw(-1.6,1.6)--node[below left]{$E_1$}(0.2,-0.2);
\draw(-0.2,-0.2)--node[above left]{$E_2$}(1.6,1.6);
\draw(1.2,1.6)--node[above right]{$E_3$}(3,-0.2);

\node at (4.5,0.7){$\cdots\cdots$};

\draw(6,-0.2)--node[below right]{$E_{k-1}$}(7.8,1.6);
\draw(7.4,1.6)--node[above right]{$E_{k}$}(9.2,-0.2);

\draw[postaction={decorate}](-1.2,1.6)--node[above left]{$y=0$}(-3,-0.2);
\draw[postaction={decorate}](9-0.2,-0.2)--node[right]{$x=0$}(10.6,1.6);

\node at (-1.4,1.4){$\bullet$};
\node at (0,0){$\bullet$};
\node at (1.4,1.4){$\bullet$};
\node at (2.8,0){$\bullet$};
\node at (6.2,0){$\bullet$};
\node at (7.6,1.4){$\bullet$};
\node at (9,0){$\bullet$};

\node[above=0.3cm] at (-1.4,1.4){$[w_0,-w_1]$};
\node[below=0.3cm] at (0,0){$[w_1,-w_2]$};
\node[above=0.3cm] at (1.4,1.4){$[w_2,-w_3]$};
\node[below=0.3cm] at (2.8,0){$[w_3,-w_4]$};
\node[below=0.3cm] at (6.2,0){$[w_{k-2},-w_{k-1}]$};
\node[above=0.3cm] at (7.6,1.4){$[w_{k-1},-w_{k}]$};
\node[below=0.3cm] at (9,0){$[w_k,-w_{k+1}]$};
\end{tikzpicture}
$$
Notice that there is the inductive relation $w_{i+1}=e_iw_i-w_{i-1}$, and we have $w_0=\theta p$, $w_1=\theta q-\tau$, $w_{k}=\theta-\tau v_{k,U}$, and $w_{k+1}=-\tau p$. 
\end{theorem}

\begin{proof} 
As $\mathbb{C}^2/L(q,p)=\Spec\mathbb{C}[X,Y]^{L(q,p)}=\Spec S_{\sigma_\Gamma^\vee}$,
the orbifold point in $\mathbb{C}^2/L(q,p)$ is defined by
$U^aV^b=0$ in $S_{\sigma_\Gamma^\vee}$, for all $U^aV^b\in S_{\sigma_\Gamma^\vee}$.
To find the exceptional set in the minimal resolution, it suffices to consider the preimage of the orbifold point.

\begin{lemma}\label{exceptionallemma}
In the affine piece $\Spec S_{\sigma_{i+1}^\vee}$, 
$E_i$ is parametrized by $U^{v_{i,V}}V^{-v_{i,U}}\in S_{\sigma_{i+1}^\vee}$, and $E_{i+1}$ is parametrized by $U^{-v_{i+1,V}}V^{v_{i+1,U}}\in S_{\sigma_{i+1}^\vee}$.
\end{lemma}
\begin{proof}
The dual cone $\sigma_{i+1}^\vee$ is spanned by $-v_{i}^{\perp}$ and $v_{i+1}^\perp$. Note that $\sigma(-v_0^\perp,v_{k+1}^\perp)$ is exactly $\sigma_{\Gamma}^\vee$.

The algebra $S_{\sigma_{i+1}^\vee}$ is the $\mathbb{C}$-algebra finitely generated by $U^aV^b$ with $(a,b)\in\sigma_{i+1}^\vee$. The dual cone $\sigma_{i+1}^\vee=\sigma(-v_i^{\perp},v_{i+1}^\perp)$ can be decomposed as the union of four cones: 
\begin{itemize}
\item $\sigma(-v_i^\perp,-v_{1}^\perp)$ spanned by $-v_i^\perp,-v_{1}^\perp$; 
\item $\sigma(-v_1^\perp,-v_0^\perp)$ spanned by $-v_1^\perp,-v_0^\perp$; 
\item $\sigma_\Gamma^\vee=\sigma(-v_0^\perp,v_{k+1}^\perp)$ spanned by $-v_0^\perp,v_{k+1}^\perp$;
\item $\sigma(v_{k+1}^\perp,v_{i+1}^\perp)$ spanned by $v_{k+1}^\perp,v_{i+1}^\perp$.
\end{itemize}

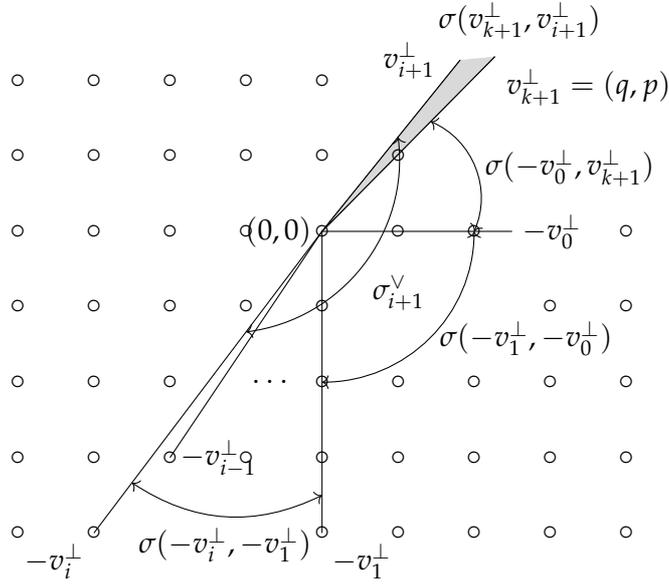
$$
\begin{tikzpicture}
\draw (0,0) to (-3,-4);
\node[below left] at (-3,-4) {$-v_{i}^\perp$};

\draw (0,0) to (-2,-3);
\node[right] at (-2,-3) {$-v_{i-1}^\perp$};

\draw (0,0) to (0,-4);
\node[below right] at (0,-4) {$-v_1^{\perp}$};

\node at (-0.65,-2) {$\cdots$};

\draw (0,0) to (4/2.2,5/2.2);
\node[left=0.2cm] at (4/2.2,5/2.2) {$v_{i+1}^\perp$}; 

\draw (0,0) to (5/2.2,5.1/2.2);
\node[below right=0.02cm] at (5/2.2,5.1/2.2) {$v_{k+1}^\perp=(q,p)$}; 

\draw (0,0) to (2.5,0);
\node[right] at (2.5,0) {$-v_0^\perp$};

\draw[fill=gray!30]    (4/2.2,5/2.2) -- (0,0) -- (5/2.2,5.1/2.2);
\node at (2.6,2.8) {$\sigma({v_{k+1}^\perp,v_{i+1}^\perp})$};

\path[<->] (-3/3,-4/3) edge[bend right=45] node[pos=0.5,below right]{$\sigma_{i+1}^\vee$} (4/4,5/4);

\path[<->] (-0,-2) edge[bend right=45] node[pos=0.5,right]{$\sigma(-v_1^\perp,-v_{0}^\perp)$} (2,0);

\path[<->] (2,0) edge[bend right=45] node[pos=0.5,right]{$\sigma(-v_0^\perp,v_{k+1}^\perp)$} (5/3.5,5.1/3.5);

\path[<->] (-3/1.2,-4/1.2) edge[bend right=30] node[pos=0.5,below]{$\sigma(-v_i^\perp,-v_{1}^\perp)$} (0,-3.5);

\node at (-4,-2) {$\circ$};
\node at (-3,-2) {$\circ$};
\node at (-2,-2) {$\circ$};
\node at (0,-2) {$\circ$};
\node at (1,-2) {$\circ$};
\node at (2,-2) {$\circ$};
\node at (3,-2) {$\circ$};
\node at (4,-2) {$\circ$};
\node at (-4,-1) {$\circ$};
\node at (-3,-1) {$\circ$};
\node at (-2,-1) {$\circ$};
\node at (-1,-1) {$\circ$};
\node at (0,-1) {$\circ$};
\node at (3,-1) {$\circ$};
\node at (4,-1) {$\circ$};
\node at (-4,0) {$\circ$};
\node at (-3,0) {$\circ$};
\node at (-2,0) {$\circ$};
\node at (-1,0) {$\circ$};
\node at (0,0) {$\circ$};
\node[left] at (0,0) {$(0,0)$};
\node at (1,0) {$\circ$};
\node at (2,0) {$\circ$};
\node at (4,0) {$\circ$};
\node at (-4,1) {$\circ$};
\node at (-3,1) {$\circ$};
\node at (-2,1) {$\circ$};
\node at (-1,1) {$\circ$};
\node at (0,1) {$\circ$};
\node at (1,1) {$\circ$};
\node at (-4,2) {$\circ$};
\node at (-3,2) {$\circ$};
\node at (-2,2) {$\circ$};
\node at (-1,2) {$\circ$};
\node at (0,2) {$\circ$};
\node at (-4,-3) {$\circ$};
\node at (-3,-3) {$\circ$};
\node at (-2,-3) {$\circ$};
\node at (-1,-3) {$\circ$};
\node at (0,-3) {$\circ$};
\node at (1,-3) {$\circ$};
\node at (2,-3) {$\circ$};
\node at (3,-3) {$\circ$};
\node at (4,-3) {$\circ$};
\node at (-4,-4) {$\circ$};
\node at (-3,-4) {$\circ$};
\node at (0,-4) {$\circ$};
\node at (1,-4) {$\circ$};
\node at (2,-4) {$\circ$};
\node at (3,-4) {$\circ$};
\node at (4,-4) {$\circ$};
\end{tikzpicture}.
$$
Write $S_{\sigma(-v_i^\perp,-v_{1}^\perp)\cup \sigma(-v_1^\perp,-v_0^\perp)\cup \sigma^\vee_\Gamma\cup \sigma(v_{k+1}^\perp,v_{i+1}^\perp)}$ as the $\mathbb{C}$-algebra generated by $S_{\sigma(-v_i^\perp,-v_{1}^\perp)}$, $S_{\sigma(-v_1^\perp,-v_0^\perp)}$, $S_{\sigma^\vee_\Gamma}$, and $S_{\sigma(v_{k+1}^\perp,v_{i+1}^\perp)}$,
then 
$$S_{\sigma_{i+1}^\vee}=S_{\sigma(-v_i^\perp,-v_{1}^\perp)\cup \sigma(-v_1^\perp,-v_0^\perp)\cup \sigma^\vee_\Gamma\cup \sigma(v_{k+1}^\perp,v_{i+1}^\perp)}.$$
Henceforth, to determine the exceptional set in $\Spec S_{\sigma_{i+1}^\vee}$, it suffices to write down a set of generators of the four $\mathbb{C}$-algebras 
and determine the equations that define the exceptional set. 
We have:
\begin{itemize}
\item For $S_{\sigma_\Gamma^\vee}$:
\begin{itemize}
\item $U^{-v_{j,V}}V^{v_{j,U}}$ is a set of generators of $S_{\sigma^\vee_\Gamma}$ with $j=0,\ldots,k+1$. 
\item All generators of $S_{\sigma^\vee_\Gamma}$ vanish in the preimage of the orbifold point.
\end{itemize}
\item For $S_{\sigma(-v_1^\perp,-v_0^\perp)}$:
\begin{itemize} 
\item $S_{\sigma(-v_1^\perp,-v_0^\perp)}$ is generated by $U,V^{-1}$.
\item In the preimage of the orbifold point, $U=0$. We will see that $V^{-1}=0$ also holds in the preimage of the orbifold point.
\end{itemize}
\item For $S_{\sigma(-v_i^\perp,-v_{1}^\perp)}$:
\begin{itemize}
\item Generators of
$S_{\sigma(-v_i^\perp,-v_{1}^\perp)}$
can be taken as 
$$P_j=U^{v_{j,V}}V^{-v_{j,U}},\ \text{$j=1,\ldots,i$}.$$
They form a set of generators for $S_{\sigma^\vee(-v_i^\perp,-v_{1}^\perp)}$ because $\Det(-v_j^\perp,-v_{j-1}^\perp)=1$.
\item 
In the preimage of the orbifold point, for any $1\leq j<i$, the above generators in $S_{\sigma_{i+1}^\vee}$ have the relations
$$P_j^{-v_{i,V}}=P_i^{-v_{j,V}}\left(V^{-1}\right)^{-v_{j,U}v_{i,V}+v_{i,U}v_{j,V}},$$
where $-v_{j,U}v_{i,V}+v_{i,U}v_{j,V}=\Det(v_i,v_j)>0$ as $i>j$. We also have the relation
$$\left(V^{-1}\right)^{v_{i,U}}=P_iU^{{-v_{i,V}}}$$
in $S_{\sigma_{i+1}^\vee}$ (note $-v_{i,V}>0$).
Thus, in the preimage of the orbifold point in $\Spec S_{\sigma_{i+1}^\vee}$, as $U=0$, we must have $V^{-1}=0$. 
So the equations $P_j=0$ hold for $1\leq j<i$. 
\end{itemize}

\item For $S_{\sigma(v_{k+1}^\perp,v_{i+1}^\perp)}$:
\begin{itemize}
\item Generators of $S_{\sigma(v_{k+1}^\perp,v_{i+1}^\perp)}$ can be taken as 
$$Q_j=U^{-v_{j,V}}V^{v_{j,U}},\ j=i+1,\ldots,k+1,$$
for the same reason as above. 

\item The relations
$$Q_j^{v_{i+1,U}}=Q_{i+1}^{v_{j,U}}U^{v_{i+1,V}v_{j,U}-v_{i+1,U}v_{j,V}}$$
hold for $i+1<j\leq k+1$ where $v_{i+1,V}v_{j,U}-v_{i+1,U}v_{j,V}=-\Det(v_{i+1},v_j)>0$. Thus, in the preimage of the orbifold point in $\Spec S_{\sigma^\vee_{i+1}}$, equations $Q_j=0$ hold for $i+1<j\leq k+1$.

\end{itemize}

\item Moreover, since $v_{i,V}-v_{i+1,V},v_{i+1,U}-v_{i,U}>0$ and $0<(v_{i+1,U}-v_{i,U})/(v_{i,V}-v_{i+1,V})<q/p$, 
$$P_iQ_{i+1}=U^{v_{i,V}-v_{i+1,V}}V^{v_{i+1,U}-v_{i,U}}\in S_{\sigma^\vee_{\Gamma}}.$$
This shows in the preimage of the orbifold point, the equation $P_iQ_{i+1}=0$ holds.
\end{itemize}
Therefore, the preimage of the orbifold point is parametrized by $P_i$ and $Q_{i+1}$ in the affine piece $\Spec S_{\sigma^\vee_{i+1}}$, which exactly correspond to the exceptional curves $E_i$ and $E_{i+1}$. $P_i$ parametrizes $E_i$  and $Q_{i+1}$ parametrizes $E_{i+1}$. 
$E_i$ and $E_{i+1}$ intersects in this affine piece at the point where $P_i=Q_{i+1}=0$. The equation $P_iQ_{i+1}=0$ holds in the preimage of the orbifold point in the affine piece $\Spec S_{\sigma^\vee_{i+1}}$, which means the preimage of the orbifold point in $\Spec S_{\sigma^\vee_{i+1}}$ is exactly $\Spec S_{\sigma^\vee_{i+1}}\cap(E_i\cup E_{i+1})$. 
\end{proof}
Recall that $U=X^p,V=Y/X^q$, and the $\mathbb{C}^*$-action on $X,Y$ is $t\curvearrowright X,Y=t^\theta X,t^\tau Y$. The prameters $P_i,Q_{i+1}$ are given by 
$$P_i=U^{v_{i,V}}V^{-v_{i,U}}=X^{pv_{i,V}}(Y/X^q)^{-v_{i,U}}=X^{pv_{i,V}+qv_{i,U}}Y^{-v_{i,U}}.$$
$$Q_{i+1}=U^{-v_{i+1,V}}V^{v_{i+1,U}}=X^{-pv_{i+1,V}}(Y/X^q)^{v_{i+1,U}}=X^{-pv_{i+1,V}-qv_{i+1,U}}Y^{v_{i+1,U}}.$$
This implies the $\mathbb{C}^*$-action on the exceptional curve $E_i$ parametrized by $P_i$ is given by 
\begin{equation}\label{actionpi}
t\curvearrowright P_i=t^{\theta (pv_{i,V}+qv_{i,U})-\tau v_{i,U}}P_i;
\end{equation}
and the $\mathbb{C}^*$-action on the exceptional curve $E_{i+1}$ parametrized by $Q_{i+1}$ is given by 
\begin{equation}\label{actionqi1}
t\curvearrowright Q_{i+1}=t^{-\theta (pv_{i+1,V}+qv_{i+1,U})+\tau v_{i+1,U}}Q_{i+1}.
\end{equation}
Note that here 
$$\theta (pv_{i,V}+qv_{i,U})-\tau v_{i,U}=\theta (q,p)\cdot (v_{i,U},v_{i,V})-\tau v_{i,U}=w_i.$$
The way that $\mathbb{C}^*$-acts on $E_i,E_{i+1}$ (\ref{actionpi}) (\ref{actionqi1})  show that the intersection point $E_{i}\cap E_{i+1}$ should have weights $[w_i,-w_{i+1}]$.


The ray given by $v_0$ in the fan $F_M$ corresponds to the proper transform of $y=0$, which is parametrized by $U=X^p$. The action on $y=0$ is given by $t\curvearrowright U=t^{p\theta} U$, which implies
the weights at the intersection of $y=0$ and $E_1$ are $[w_0,-w_1]$. Similarly, the ray given by $v_{k+1}=(p,-q)$ corresponds to the proper transform of $x=0$, which is parametrized by $U^qV^p=Y^p$. The action on $x=0$ is given by $t\curvearrowright U^qV^p=t^{\tau p} U^qV^p$, which implies the weights at the intersection of $x=0$ and $E_k$ are $[w_k,-w_{k+1}]$.
\end{proof}

\bibliographystyle{plain}
\bibliography{reference}

\Addresses

\end{document}